%% file: esvp_partitions_draft3.tex
\documentclass[10pt]{article}
\usepackage{paralist}

\usepackage{stmaryrd}

\input{preamble.tex}

\usepackage[backend=biber]{biblatex}
\begin{document}
\title{Enriched set-valued $P$-partitions
and
shifted stable Grothendieck polynomials}
\author{
Joel Brewster Lewis \\ Department of Mathematics  \\ George Washington University \\ {\tt jblewis@gwu.edu}
\and
Eric Marberg \\ Department of Mathematics \\  HKUST \\ {\tt eric.marberg@gmail.com}
}

\date{}

\maketitle

\begin{abstract}
We introduce an enriched analogue of Lam and Pylyavskyy's theory of set-valued $P$-partitions. An an application, we construct a $K$-theoretic version of Stembridge's Hopf algebra of peak quasisymmetric functions. We show that the symmetric part of this algebra is generated by Ikeda and Naruse's shifted stable Grothendieck polynomials. We give the first proof that the natural skew analogues of these power series are also symmetric. A central tool in our constructions is a ``$K$-theoretic'' Hopf algebra of labeled posets, which may be of independent interest. Our results also lead to some new explicit formulas for the involution $\omega$ on the ring of symmetric functions.
\end{abstract}

\setcounter{tocdepth}{2}
\tableofcontents

\section{Introduction}

Stanley developed the now-classical theory of \emph{(ordinary) $P$-partitions} in \cite{St1}. These are certain maps 
from (finite) labeled posets to the positive integers $\PP :=\{1,2,3,\dots\}$. They can be seen as generalizations 
of semistandard Young tableaux, and provide a streamlined description of the tableau generating functions for the skew Schur functions $s_{\lambda/\mu}$.  

Following Stanley's work, the theory of $P$-partitions has been generalized and extended in a number of ways.
Stembridge \cite{Stembridge1997a} introduced \emph{enriched $P$-partitions},
which are certain maps from labeled posets to the \emph{marked integers} $\MM := \{ 1' < 1 < 2'<2< \dots\}$.
They can be seen as generalizations of semistandard shifted (marked) tableaux, whose generating functions give
the \emph{Schur $P$- and $Q$-functions} $P_\lambda$ and $Q_\lambda$.
In \cite{LamPyl}, Lam and Pylyavskyy defined \emph{set-valued $P$-partitions},
which are maps that now take finite nonempty subsets of $\PP$ as values. 
These maps  are generalizations of semistandard set-valued tableaux,
which play a role in the generating functions
for the \emph{stable Grothendieck polynomials} $G^{(\beta)}_\lambda$ studied in \cite{Buch2002,BKSTY,FominKirillov94}.

Our goal in this article is to provide the enriched counterpart to Lam and Pylyavskyy's definition.
Stembridge \cite{Stembridge1997a} remarks that 
``almost every aspect of the theory of ordinary $P$-partitions has an enriched counterpart.''
In a very satisfying sense, it turns out that most features of enriched $P$-partitions
similarly have a set-valued extension. 
In particular, we introduce a theory of
\emph{enriched set-valued $P$-partitions}. These are certain maps that take finite nonempty subsets of $\MM$ as values.
They can be viewed as generalizations of semistandard shifted set-valued tableaux,
which appear in the combinatorial generating functions for Ikeda and Naruse's 
\emph{$K$-theoretic Schur $P$- and $Q$-functions} $\bGP_\lambda$ and $\bGQ_\lambda$ studied in \cite{HKPWZZ,IkedaNaruse,NN2017,NN2018,Naruse2018}.

One motivation for this project was to
compute $\omega(\bGP_\lambda)$ and $\omega(\bGQ_\lambda)$,
where $\omega$ denotes the usual involution of the ring of symmetric functions
that maps $s_\lambda \mapsto s_{\lambda^T}$ for all partitions $\lambda$.
It turns out that we can obtain explicit formulas for 
$\omega$
evaluated at $\bGP_\lambda$ and $\bGQ_\lambda$ by decomposing the latter power series
into quasisymmetric functions
attached to enriched set-valued $P$-partitions
on which $\omega$ acts in a more transparent manner.
Our results along these lines appear at end of this paper in Section~\ref{antipode-sect}.

There are a few other reasons to be interested in enriched set-valued analogues of $P$-partitions.  
%
%
%
They suggest a straightforward definition of $K$-theoretic Schur $P$- and $Q$-functions indexed by skew
shapes. These skew generalizations 
do not seem to have been considered previously;
we establish some of their fundamental properties,
like symmetry.  Our definition also indicates a good notion of \emph{$K$-theoretic Schur $S$-functions}.
The theory of enriched set-valued $P$-partitions leads, moreover, to the construction of a Hopf algebra $\mcoPeak$,
whose elements we call \emph{multipeak quasisymmetric functions}. This is a $K$-theoretic analogue of Stembridge's algebra of peak quasisymmetric functions \cite{Stembridge1997a}, and is perhaps of independent interest. 

In \cite{LamPyl}, Lam and Pylyavskyy  study a diagram of 
six Hopf algebras related to the $K$-theory of the Grassmannian. 
There are two shifted variants of this diagram, one for the maximal orthogonal Grassmannian and
one for the Lagrangian Grassmannian.
The planned sequel to this paper will study the ``shifted $K$-theoretic'' Hopf algebras in these diagrams.  The algebra $\mcoPeak$ and its dual
(a $K$-theoretic analogue of the peak algebra) 
will figure prominently in the shifted diagrams.

We now summarize our main results and outline the rest of this paper.
Section~\ref{prelim-sect} gives some background on Hopf algebras in the category of 
\emph{linearly compact modules} and on \emph{combinatorial Hopf algebras}.
In Section~\ref{mlpset-sect}, we introduce a ``$K$-theoretic'' Hopf algebra of labeled posets 
and use this object to recover several constructions of Lam and Pylyavskyy, including the stable Grothendieck polynomials.
Section~\ref{main-sect} contains our main definitions and results related
to enriched $P$-partitions and associated quasisymmetric functions.
We define the skew analogues of Ikeda and Naruse's $K$-theoretic Schur $P$- and $Q$-functions 
in Section~\ref{shifted-stable-sect}. In Section~\ref{sym-sect}, we prove the symmetry of these power series 
and characterize the subalgebra that they generate.
Finally, Section~\ref{antipode-sect} leverages our results to derive explicit 
formulas for some notable involutions on quasisymmetric functions.

\subsection*{Acknowledgements}

The first author was partially supported by an ORAU Powe award.
The second author was partially supported by Hong Kong RGC Grant ECS 26305218.
We are grateful to Zach Hamaker, Hiroshi Naruse, Brendan Pawlowski, and Alex Yong
for helpful comments.

\section{Preliminaries}\label{prelim-sect}
\subsection{Completions}\label{completions-sect}

In this article, we are often concerned with rings of formal power series of unbounded degree
that are ``too large'' to belong to the category of free modules.
To define monoidal structures on these objects, we need to work in 
the following slightly more exotic setting.

Fix an integral domain $R$ and write $\otimes = \otimes_R$ for the usual tensor product.
An \emph{$R$-algebra} is an $R$-module $A$ with $R$-linear
product $\nabla : A\otimes A \to A$
and unit $\iota : R\to A$ 
maps.
Dually, an \emph{$R$-coalgebra} is an $R$-module $A$ with $R$-linear 
coproduct $\Delta : A \to A\otimes A$
and 
counit $\epsilon : A \to R$ maps.
 The (co)product and (co)unit maps must satisfy several natural associativity axioms; see \cite[\S1]{GrinbergReiner}
 for the complete definitions.
 (Co)algebras form a category in which morphisms are $R$-linear maps commuting 
 with the (co)unit and (co)product maps.


An $R$-module $A$ that is simultaneously an $R$-algebra and an $R$-coalgebra is an \emph{$R$-bialgebra}
if the coproduct and counit maps are algebra morphisms (equivalently, the product and unit are coalgebra morphisms).
Suppose $A$ is an $R$-bialgebra with structure maps $\nabla$, $\iota$, $\Delta$, and $\epsilon$. 
Let $\End(A)$ denote the set of $R$-linear maps $A \to A$. This set is an $R$-algebra with 
product given by the linear map with
$
f\otimes g \mapsto \nabla \circ (f\otimes g) \circ \Delta
$
for $f,g \in \End(A)$
and unit 
given by the linear map with $1_R \mapsto \iota\circ \epsilon$. 
The bialgebra $A$ is a \emph{Hopf algebra} if  $\id : A \to A$ has a
(necessarily unique) two-sided multiplicative inverse 
$\antipode : A \to A$ in the algebra $\End(A)$,  in which case we call $\antipode$ the \emph{antipode} of $A$.

Consider a free $R$-module $A$, and fix a basis $\{a_i\}_{i \in I}$ for $A$.
Suppose that $B$ is an $R$-module and $\langle\cdot,\cdot\rangle : A \times B \to R$
is a \emph{nondegenerate} $R$-bilinear form, in the sense that
$b\mapsto \langle \cdot,b\rangle$ is a bijection $B\to \Hom_R(A,R)$.
Then for each $j \in I$, there exists a unique element $b_j \in B$ with
$\langle a_i, b_j \rangle = \delta_{ij}$ for all $i \in I$,
and we can identify $B$ 
with the product $\prod_{j \in I} Rb_j$, which we view as
the set of arbitrary $R$-linear combinations of the elements $\{b_j\}_{j\in I}$.
We refer to $\{b_i\}_{i \in I}$ as a \emph{pseudobasis} for $B$;
some authors call this a \emph{continuous basis}.

\begin{example}\label{lc-ex0}
Let $A = R[x]$ and $B=R\llbracket x \rrbracket$.
Define $\langle\cdot,\cdot\rangle : A \times B \to R$ to be the nondegenerate $R$-bilinear form with
$\left\langle \sum_{n\geq 0} r_n x^n, \sum_{n \geq 0} s_n x^n\right\rangle = \sum_{n\geq 0} r_n s_n$.
The set $\{x^n\}_{n\geq 0}$ is a basis for $A$ and a pseudobasis for $B$.
\end{example}

Endow $R$ with the discrete topology.
The \emph{linearly compact topology} on $B$ \cite[\S I.2]{Dieudonne} is the coarsest topology 
in which the maps $\langle a_i, \cdot \rangle : B \to R$ are all continuous.
If we identify $B \cong \prod_{j \in I} R b_j$ and give each $Rb_j$ the discrete topology, then
this is the usual product topology.
The linearly compact topology depends on $\langle\cdot,\cdot\rangle$ but not on the choice of basis for $A$.
It is discrete if $A$ has finite rank.
We refer to $B$, equipped with this topology, as a \emph{linearly compact $R$-module},
and say that $B$ is the \emph{dual} of $A$ with respect to the form $\langle\cdot,\cdot\rangle$. 



We will often abbreviate by writing ``LC-'' in place of ``linearly compact.''  LC-modules form a category in which morphisms are continuous $R$-linear maps.

\begin{example}\label{lc-ex1}
With $A = R[x]$ and $B = R\llbracket x \rrbracket$ as in Example~\ref{lc-ex0}, a basis of open subsets in the LC-topology on $R\llbracket x \rrbracket$ is given by sets of power series in $R\llbracket x \rrbracket$ whose coefficients are fixed in a finite set of degrees and unconstrained elsewhere.  We view formal power series rings in multiple variables as LC-modules similarly.
\end{example}


Suppose $A$ is a free $R$-module with basis $S$.
Let $B$ be the
$R$-module of arbitrary $R$-linear combinations of elements of $S$,
 equipped with the nondegenerate bilinear form
$ A \times B \to R$ making $S$ orthonormal.  We say that $B$
is the 
\emph{completion} of $A$ with respect to $S$.
This is a linearly compact $R$-module with $S$ as a pseudobasis.

Let $B$ and $B'$ be linearly compact $R$-modules dual to free $R$-modules $A$ and $A'$, and write $\langle\cdot,\cdot\rangle$ for both of the 
associated bilinear forms. 
Every $R$-linear map $\phi : A' \to A$ has a unique adjoint $\psi : B\to B'$ such that 
$\langle \phi(a'), b\rangle = \langle a',\psi(b)\rangle$ for all $a' \in A'$ and $b \in B$.
A linear map $B \to B'$ is continuous if and only if it arises as the adjoint 
of some linear map $A' \to A$.


The \emph{completed tensor product} of $B$ and $B'$ 
is the $R$-module 
\[B \htimes B' := \Hom_R(A\otimes A',R),\]
given the LC-topology from the tautological pairing 
$(A\otimes A') \times \Hom_R(A\otimes A',R) \to R$.
If $\{b_i\}_{i \in I}$ and $\{b'_j \}_{j \in J}$ are pseudobases for $B$ and $B'$,
then we can realize $B\htimes B'$ concretely as the linearly compact $R$-module
with the set of tensors $\{ b_i \otimes b_j'\}_{(i,j) \in I \times J}$ as a pseudobasis.
There is an inclusion 
$B\otimes B' \hookrightarrow B\htimes B'$, which is an isomorphism if and only if 
$A$ or $A'$ has finite rank.

\begin{example}
If we view $R\llbracket x \rrbracket$ and $R\llbracket y \rrbracket$ as
linearly compact $R$-modules as in Example~\ref{lc-ex1}
then $R\llbracket x \rrbracket \otimes R\llbracket y \rrbracket \neq R\llbracket x \rrbracket\htimes R\llbracket y \rrbracket \cong R\llbracket x,y \rrbracket$,
where we consider $R\llbracket x,y \rrbracket$ as the linearly compact $R$-module dual to $R[x,y]$.
\end{example}

If $\iota : B \to R$ and $\nabla : B \htimes B \to B$ are continuous linear maps,
then these maps are the adjoints of unique linear maps $\epsilon : R \to A$ and $\Delta : A \to A \otimes A$,
and we say that $(B,\nabla,\iota)$ is an \emph{LC-algebra}
if $(A,\Delta,\epsilon)$ is an $R$-coalgebra.
Similarly, we say that continuous linear maps
$\Delta : B \to B\htimes B$ and $\epsilon : B \to R$ 
make $B$ into an \emph{LC-coalgebra}
if $\Delta$ and $\epsilon$ are the adjoints of the product and unit 
maps of an $R$-algebra structure on $A$.
\emph{LC-bialgebras} and \emph{LC-Hopf algebras} are defined analogously.
In each case we say that the (co, bi, Hopf) algebra structures on $A$ and $B$ are duals of each other.
If $B$ is an LC-Hopf algebra then its antipode is defined to be the adjoint of the antipode of 
the Hopf algebra $A$.

One can  reformulate these definitions in terms of commutative diagrams; see \cite[\S2 and \S3]{Marberg2018}.
Linearly compact (co, bi, Hopf) algebras form a category in which morphisms are continuous linear maps 
commuting with (co)products and (co)units.
The completed tensor product of two linearly compact (co, bi, Hopf) algebras is
naturally a linearly compact (co, bi, Hopf) algebra.

\begin{example}\label{lc-ex4}
Again let $A = R[x]$ and $B = R\llbracket x \rrbracket$ but now suppose $R = \ZZ[\beta]$.
 Define $\iota : R \to B$ to be the obvious inclusion and let $\epsilon : B \to R$ be the map setting $x=0$. Let $\nabla : B \htimes B \to B$ be the usual multiplication map 
 but define $\Delta_\beta : B \to B\htimes B$ to be the continuous algebra homomorphism with
 $\Delta_\beta(x) = x\otimes 1 +1\otimes x + \beta  x \otimes x .$
 
  The operations $\iota$ and $\epsilon$ restrict to maps $R \to A$ and $A \to R$.
  Define $\Delta  : A \to A\otimes A$ as the linear map with
 $\Delta(x^n) = \sum_{i+j = n} x^i\otimes x^j$
 and let $\nabla_\beta : A \otimes A \to A$ be
 the commutative, associative, linear map whose
  $(n-1)$-fold iteration maps
\be\label{fold-eq}
\nabla_\beta^{(n - 1)} : x \otimes (x-\beta) \otimes (x-2\beta) \otimes \cdots \otimes (x-(n-1)\beta) \mapsto n! \cdot x^n
\ee
and which is such that $\nabla_\beta \circ (\id \otimes \iota) : A\otimes R \to A$ and $\nabla_\beta \circ (\iota \otimes \id) : R\otimes A \to A$ are the canonical isomorphisms. 
(The existence and uniqueness of $\nabla_\beta$ is not obvious, but follows as an interesting, fairly straightforward exercise.)

The triple $(A,\nabla_\beta,\iota)$ is automatically an algebra, and 
one can show that $\epsilon$ and $\Delta$ are algebra homomorphisms, so $(A,\nabla_\beta,\iota,\Delta,\epsilon)$ is a bialgebra.
One can check, moreover, that the dual of this bialgebra structure via the form in Example~\ref{lc-ex1} 
is precisely 
$(B,\nabla,\iota,\Delta_\beta,\epsilon)$, which is thus an LC-bialgebra.

There are several ways of seeing that $A$ is a Hopf algebra and computing its antipode $\antipode$. Since $\iota \circ \epsilon = \nabla_\beta \circ (\id \otimes \antipode) \circ \Delta$,
we must have $\antipode(x) = -x$. Using this and the fact that $\antipode$ is an algebra anti-automorphism, 
one can show  that 
\[\antipode(x^n) = (-1)^n x(x+\beta)^{n-1}\]
for $n>0$.
It follows by duality that $B$ is an LC-Hopf algebra whose antipode $\hat \antipode : B\to B$
is the continuous linear map with
$\hat \antipode(1) = 1$ and
\[\hat \antipode(x^m) = \sum_{n\geq m} (-1)^n \tbinom{n-1}{m-1}\beta^{n-m} x^n = (\tfrac{-x}{1+\beta x})^m\] for $m>0$.
It is interesting to note that $\nabla$ and $\Delta_\beta$ restrict to well-defined maps $A \otimes A \to A$ and $A \to A\otimes A$,
which give $A$ a second bialgebra structure. This bialgebra is not a Hopf algebra, however, 
since $\hat \antipode$ is not a map $A \to A$.
\end{example}

In a few places we will encounter the following construction, where the added complications of linear compactness seem a little superfluous.
Suppose $H$ is a $\NN$-graded connected\footnote{
A graded bialgebra $H = \bigoplus_{n\geq 0} H_n$ is 
\emph{connected} if the unit and counit maps restrict 
to inverse isomorphisms $R \cong H_0$. 
This condition is not essential, but recurs in many examples and is often convenient to assume.
Any graded connected bialgebra is automatically a Hopf algebra, and 
when defining such objects one just needs to specify the (co)product maps.} Hopf algebra, with 
finite graded rank and with a homogeneous basis $S$.
Let $\hat H$ be the completion of $H$ with respect to $S$.
There is an inclusion $H\hookrightarrow  \hat H$ and all Hopf structure maps 
automatically extend to continuous linear maps, making 
$\hat H$ into an LC-Hopf algebra (namely, the one dual to the graded dual of $H$).

Significantly, LC-Hopf algebras arising as completions
in this way
 often have interesting subalgebras that 
are not themselves completions.

\begin{example}
If we set $\beta=0$ in Example~\ref{lc-ex4}, then $A=\ZZ[x]$ becomes 
a $\NN$-graded connected Hopf algebra with finite graded rank.
The completion of this Hopf algebra with respect to $S = \{x^n\}_{n\geq 0}$
can be identified with
the (proper) LC-Hopf subalgebra of $B=\ZZ\llbracket x \rrbracket$
with pseudobasis $\{ n!\cdot x^n\}_{n\geq 0}$.
\end{example}

\subsection{Quasisymmetric functions}

Continue to let $R$ be an integral domain, and 
suppose $x_1$, $x_2$, \dots are commuting indeterminates.
A power series $f \in R\llbracket x_1,x_2,\dots\rrbracket$ is \emph{quasisymmetric} if
for any choice of exponents $a_1,a_2,\dots,a_k \in \PP$,
the coefficients of $x_1^{a_1}x_2^{a_2}\cdots x_{k}^{a_k}$ and $x_{i_1}^{a_1}x_{i_2}^{a_2}\cdots x_{i_k}^{a_k}$
in $f$ are equal for all $i_1<i_2<\dots<i_k$.

\begin{definition}
Let $\mQSym_R$ denote the $R$-module of all
quasisymmetric power series in $R\llbracket x_1,x_2,\dots \rrbracket$.
Let $\QSym_R$ denote the submodule of power series in $\mQSym_R$ of bounded degree.
\end{definition}

A \emph{composition} is a finite sequence of positive integers 
$\alpha = (\alpha_1,\alpha_2,\dots,\alpha_l)$.
If $n =  \alpha_1 + \alpha_2 + \dots+ \alpha_l$
 then we write $\alpha \vDash n$ and and set $|\alpha| := n$.
The \emph{monomial quasisymmetric function}
of a nonempty composition $\alpha = (\alpha_1,\alpha_2,\dots,\alpha_l)$
is
\[
M_\alpha := \sum_{i_1<i_2<\dots<i_l} x_{i_1}^{\alpha_1} x_{i_2}^{\alpha_2}\cdots x_{i_l}^{\alpha_l} \in \QSym_R.
\]
When $\alpha =\emptyset$ is empty, we set $M_\emptyset := 1$.
Then $\QSym_R$ is a graded ring that is free as an $R$-module with the  set of
power series $M_\alpha$
as a homogeneous basis.
 We identify $\mQSym_R$ with the corresponding completion.
 This makes $\mQSym_R$ into a linearly compact $R$-module, whose
topology coincides with the subspace topology induced by the LC-module $R\llbracket x_1,x_2,\dots \rrbracket$.

Let $\alpha'\alpha''$ denote the concatenation of two compositions $\alpha'$ and $\alpha''$. 
There is a unique $R$-linear map $\Delta : \QSym_R \to \QSym_R \otimes \QSym_R$ such that 
$
\Delta(M_\alpha) = \sum_{\alpha =\alpha'\alpha''} M_{\alpha'} \otimes M_{\alpha''}
$
for each composition $\alpha$.
Let $\epsilon : \QSym_R \to R$ be the linear map with $M_{\emptyset} \mapsto 1$ and $M_\alpha\mapsto 0$
for all nonempty compositions $\alpha$.
With this coproduct and counit, 
 $\QSym_R$ becomes a graded, connected Hopf algebra \cite[\S5.1]{GrinbergReiner}.
(For a discussion of its antipode, see Section~\ref{antipode-subsection}.)
 
Since $\QSym_R$ has finite graded rank, 
its product and coproduct extend to continuous linear maps $\mQSym_R \htimes \mQSym_R \to\mQSym_R$
and
$\mQSym_R \to \mQSym_R \htimes \mQSym_R$ making $\mQSym_R$ into an LC-Hopf algebra.
This algebra has an important universal property, which we presently describe.

Suppose $H $
is a linearly compact $R$-bialgebra
with product $\nabla$, coproduct $\Delta$, unit $\iota$, and counit $\epsilon$.
Let $\XX(H)$ denote the set of continuous algebra morphisms $\zeta : H \to R\llbracket t\rrbracket$
with $\zeta(\cdot )|_{t=0}=\epsilon$.
This set is a monoid under the \emph{convolution product}
$ \zeta * \zeta' := \nabla_{R\llbracket t\rrbracket} \circ (\zeta \htimes \zeta')\circ \Delta$ 
with unit element $\iota\circ \epsilon$.
If $H$ has an antipode $\antipode$ then 
$\zeta\circ\antipode$ is the inverse of $\zeta \in \XX(H)$ under $*$, so
in this case $\XX(H)$ is a group.  This leads to a natural extension of Aguiar, Bergeron, and Sottile's notion of a \emph{combinatorial Hopf algebra}
\cite{ABS} to linearly compact modules.

\begin{definition}
If $H$ is an LC-bialgebra (respectively, LC-Hopf algebra)
and $\zeta \in \XX(H)$, then we refer to $(H,\zeta)$
as a \emph{combinatorial LC-bialgebra} (respectively, \emph{combinatorial LC-Hopf algebra}).
Such pairs form a category in which morphisms $(H,\zeta) \to (H',\zeta')$
are LC-bialgebra morphisms $\phi : H \to H'$ with $\zeta = \zeta'\circ \phi$.
\end{definition}

We view $\mQSym_R$ as a combinatorial LC-Hopf algebra with respect to
the \emph{universal zeta function} $\zetaq : \mQSym_R \to R\llbracket t\rrbracket$ 
given by $\zetaq(f) = f(t,0,0,\dots)$.
One has
$
\zetaq(M_\alpha)  = t^{|\alpha|} 
$
for $\alpha\in\{\emptyset,(1),(2),(3),\dots\}$ 
and $\zetaq(M_\alpha)=0$
for all other compositions $\alpha$.

The next theorem, which is a mild generalization of \cite[Thm.~4.1]{ABS},
shows that $(\mQSym_R, \zetaq)$ is a final object in the category of combinatorial LC-bialgebras.
Given an LC-bialgebra $H$, a character $\zeta \in \XX(H)$, and a nontrivial composition $\alpha
=(\alpha_1,\alpha_2,\dots,\alpha_k)$,
let $\zeta_\alpha : H \to R$ denote the map
sending $h \in H$ to the coefficient of 
$t^{\alpha_1}\otimes t^{\alpha_2}\otimes \cdots \otimes t^{\alpha_k}$
in $\zeta^{\otimes k} \circ \Delta^{(k-1)}(h) \in R\llbracket t\rrbracket^{\htimes k}$,
where $\Delta^{(0)} := \id$.
When $\alpha=\emptyset$ is empty,
let $\zeta_\emptyset = \epsilon$.

\begin{theorem}
\label{abs-thm}
If $(H,\zeta)$ is a combinatorial LC-bialgebra
then there exists a  unique morphism $\Phi : (H,\zeta) \to (\mQSym_R,\zetaq)$, given explicitly 
by the map with
$\Phi(h) = \sum_\alpha \zeta_\alpha(h) M_\alpha$
for $h \in H$,
where the sum is over all compositions $\alpha$.
\end{theorem}

\begin{proof}
By extending the ring of scalars, any such morphism $\Phi$ extends to a morphism over the field of fractions of $R$, and this extension satisfies $\zetaq \circ \Phi = \zeta$.  By \cite[Thm.~7.8]{Marberg2018}, there is a unique such morphism $\Phi$ defined over a field, and this morphism has the given formula $\Phi(h) = \sum_\alpha \zeta_\alpha(h) M_\alpha$.  Since this formula is in fact defined over $R$, the result follows.
\end{proof}

\begin{example}
Suppose $R=\ZZ[\beta]$ and $H = \ZZ[\beta]\llbracket x \rrbracket$ with the LC-Hopf algebra structure in Example~\ref{lc-ex4}.
Each map $\zeta : H \to \ZZ[\beta]\llbracket t\rrbracket$ in $\XX(H)$ is uniquely determined by its value at $x$,
and there exists $\zeta \in \XX(H)$ with $\zeta(x) = f \in \ZZ\llbracket t\rrbracket$
if and only if $f \in t\ZZ[\beta]\llbracket t\rrbracket$ is a power series with no constant term.
If $\zeta  \in \XX(H)$ is the trivial isomorphism mapping $x \mapsto t$,
then the morphism $\Phi$ in Theorem~\ref{abs-thm}
sends $x \mapsto M_{(1)} + \beta M_{(1,1)} + \beta^2 M_{(1,1,1)} + \dots$.
\end{example}

\begin{notation}
For the rest of this article, we
fix the ring of scalars to be $R=\ZZ[\beta]$,
where $\beta$ is an indeterminate,
and define 
\be\label{qsym-beta-convention-eq} \QSym := \QSym_{\ZZ[\beta]}
\qquand \mQSym := \mQSym_{\ZZ[\beta]}.
\ee
All combinatorial Hopf algebras are assumed to be defined over $\ZZ[\beta]$.
\end{notation}

 \section{Set-valued $P$-partitions}
 \label{mlpset-sect}
 
 In this section, we introduce a combinatorial LC-Hopf algebra on labeled posets.
 Applying the canonical morphism from this object to $\mQSym$ produces a family of interesting quasisymmetric functions,
 recovering Lam and Pylyavskyy's generating functions for set-valued $P$-partitions \cite{LamPyl}.

  \subsection{Labeled posets}
 
 Let $P$ be a finite poset with an injective labeling map $\gamma : P \to \ZZ$.
 We refer to 
 the pair $(P,\gamma)$ as a \emph{labeled poset}.

We say that $t \in P$ \emph{covers} $s \in P$ and write $s\lessdot t$ if $\{x \in P : s \leq x < t\} = \{s\}$.
Two labelings $\gamma$ and $\delta$ of $P$ are \emph{equivalent}
if, whenever $s\lessdot t$ in $P$, we have $\gamma(s) > \gamma(t)$ if and only if $\delta(s) > \delta(t)$.
Labeled posets $(P,\gamma)$ and $(Q,\delta)$ are \emph{isomorphic}
if there is a poset isomorphism $\phi : P \xrightarrow{\sim} Q$ such that $\gamma$ and $ \delta\circ \phi$
are equivalent labelings of $P$.  Denote the isomorphism class of $(P,\gamma)$ by $[(P,\gamma)]$.


The isomorphism class $[(P, \gamma)]$ may be represented
uniquely by adding an orientation to the Hasse diagram of $P$, as follows:
for each covering relation $x\lessdot y$, there is an edge
$x\to y$ if $\gamma(x) > \gamma(y)$ and an edge $x \leftarrow y$ 
if $\gamma(x) < \gamma(y)$.

\begin{example}
\label{ori-ex}
Let $P$ be the labeled poset with four elements $s_1,s_2,s_3,s_4$ and four covering relations
 $s_1\lessdot s_2$ and $s_1\lessdot s_3$ and $s_2\lessdot s_4$ and $s_3\lessdot s_4$,
and let $\gamma : P \to \ZZ$ be defined by $\gamma(s_1) = 5$ and $\gamma(s_i) = i$ for $i\in\{2,3,4\}$.
Then $[(P, \gamma)]$ is represented by the oriented Hasse diagram
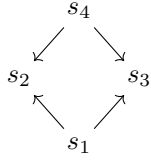
\[
\begin{tikzpicture}[baseline=(c2.base), xscale=0.4, yscale=0.3]
\tikzset{edge/.style = {->}}
  \node (c4) at (0,3) {$s_4$};
\node (c3) at (2,0) {$s_3$};
\node (c2) at (-2,0) {$s_2$};
  \node (c1) at (0,-3) {$s_1$};
  \draw[edge] (c1) -- (c2);
  \draw[edge] (c1) -- (c3);
  \draw[edge] (c4) -- (c2);
  \draw[edge] (c4) -- (c3);
\end{tikzpicture}
\] 
\end{example}

\begin{definition}
Define
$\mLPSet$ to be the linearly compact $\ZZ[\beta]$-module with a pseudobasis given by 
the isomorphism classes of all (finite) labeled posets.
\end{definition}

There is a natural Hopf structure on $\mLPSet$,
which we now describe.
We define a disjoint union operation $\sqcup$ on labeled posets $(P,\gamma)$ and $(Q,\delta)$
so that the oriented Hasse diagram of
the isomorphism class
 $[(P, \gamma) \sqcup (Q, \delta)]$
is the disjoint union of the oriented Hasse diagrams 
of $[(P,\gamma)]$ and $[(Q,\delta)]$: concretely,
$(P,\gamma)\sqcup (Q,\delta) := (P\sqcup Q, \gamma\sqcup \delta)$ is the labeled poset such that
$P \sqcup Q$ is the usual disjoint poset union
and
$\gamma \sqcup \delta : P \sqcup Q \to \ZZ$ is the labeling map
\be\label{sqcup-eq}
(\gamma \sqcup \delta)(s) := \begin{cases}
 \gamma(s) - \max_{x\in P} \gamma(x)&\text{if $s \in P$}
\\
 \delta(s) - \min_{x \in Q} \delta(x)+ 1 &\text{if $s \in Q$}.
 \end{cases}
 \ee 
 If $S \subseteq P$ is any subset of a labeled poset $(P, \gamma)$ then $(S,\gamma|_S)$
  is itself a labeled poset, where 
$S$ inherits the partial order of $P$.  
To unclutter our notation, we will henceforth write $(S,\gamma)$ in place of $(S,\gamma|_S)$.
A subposet $S \subseteq P$ is a \emph{lower set} (respectively, \emph{upper set})
if  
for all $x,y \in P$ with $x<y$, 
$y \in S$ $\Rightarrow$ $x\in S$ (respectively, $x \in S$ $\Rightarrow$ $y \in S$). 
A subset $S\subseteq P$ is an \emph{antichain}
if no elements $x,y \in S$ satisfy $x<y$ in $P$.
Given labeled posets $(P,\gamma)$ and $(Q,\delta)$, define
\[
\nabla([(P,\gamma)] \otimes [(Q,\delta)]) := [(P,\gamma)\sqcup (Q,\delta)]
\]
and
\[ \Delta([(P,\gamma)]) := \sum_{S\cup T = P} \beta^{|S\cap T|} \cdot[(S,\gamma)] \otimes [(T,\gamma)]\]
where the sum is over all ordered pairs $(S,T)$ of subsets of $P$ such that 
$S$ is a lower set, $T$ is an upper set,
$P = S\cup T$, and 
$S\cap T$ is an antichain.
Since there are only finitely many isomorphism classes of labeled posets of a given size,
these operations extend to continuous linear maps
$\nabla : \mLPSet \htimes \mLPSet \to \mLPSet$
and $\Delta : \mLPSet \to \mLPSet \htimes \mLPSet$.

\begin{example}\label{mlp-delta-ex}
The value of $\Delta([(P,\gamma)])$ for $P = \{ a < b\}$ with $\gamma(a) < \gamma(b)$
is 
\[
\ba
 \Delta\(\begin{tikzpicture}[baseline=(c.base), xscale=0.6, yscale=0.4]
\tikzset{edge/.style = {<-}}
\node (c) at (0, -1) {};
  \node (b) at (0,0) {$b$};
  \node (a) at (0,-2) {$a$};
  \draw[edge] (a) -- (b);
\end{tikzpicture}\)
&=
\(\begin{tikzpicture}[baseline=(c.base), xscale=0.6, yscale=0.4]
\tikzset{edge/.style = {<-}}
\node (c) at (0, -1) {};
  \node (a) at (0,-2) {\ };
  \node (b) at (0,0) {\ };
\end{tikzpicture}\)
\otimes 
\(\begin{tikzpicture}[baseline=(c.base), xscale=0.6, yscale=0.4]
\tikzset{edge/.style = {<-}}
\node (c) at (0, -1) {};
  \node (b) at (0,0) {$b$};
  \node (a) at (0,-2) {$a$};
  \draw[edge] (a) -- (b);
\end{tikzpicture}\)
+
\(\begin{tikzpicture}[baseline=(c.base), xscale=0.6, yscale=0.4]
\tikzset{edge/.style = {<-}}
\node (c) at (0, -1) {};
  \node (a) at (0,-2) {$a$};
  \node (b) at (0,0) {\ };
\end{tikzpicture}\)
\otimes 
\(\begin{tikzpicture}[baseline=(c.base), xscale=0.6, yscale=0.4]
\tikzset{edge/.style = {<-}}
\node (c) at (0, -1) {};
  \node (a) at (0,-2) {\ };
  \node (b) at (0,0) {$b$};
\end{tikzpicture}\)
+
\(\begin{tikzpicture}[baseline=(c.base), xscale=0.6, yscale=0.4]
\tikzset{edge/.style = {<-}}
\node (c) at (0, -1) {};
  \node (b) at (0,0) {$b$};
  \node (a) at (0,-2) {$a$};
  \draw[edge] (a) -- (b);
\end{tikzpicture}\)
\otimes 
\(\begin{tikzpicture}[baseline=(c.base), xscale=0.6, yscale=0.4]
\tikzset{edge/.style = {<-}}
\node (c) at (0, -1) {};
  \node (a) at (0,-2) {\ };
  \node (b) at (0,0) {\ };
\end{tikzpicture}\)
\\&\quad+
\beta \(\begin{tikzpicture}[baseline=(c.base), xscale=0.6, yscale=0.4]
\tikzset{edge/.style = {<-}}
\node (c) at (0, -1) {};
  \node (a) at (0,-2) {$a$};
  \node (b) at (0,0) {\ };
\end{tikzpicture}\)
\otimes 
\(\begin{tikzpicture}[baseline=(c.base), xscale=0.6, yscale=0.4]
\tikzset{edge/.style = {<-}}
\node (c) at (0, -1) {};
  \node (b) at (0,0) {$b$};
  \node (a) at (0,-2) {$a$};
  \draw[edge] (a) -- (b);
\end{tikzpicture}\)
+
\beta
\(\begin{tikzpicture}[baseline=(c.base), xscale=0.6, yscale=0.4]
\tikzset{edge/.style = {<-}}
\node (c) at (0, -1) {};
  \node (b) at (0,0) {$b$};
  \node (a) at (0,-2) {$a$};
  \draw[edge] (a) -- (b);
\end{tikzpicture}\)
\otimes
 \(\begin{tikzpicture}[baseline=(c.base), xscale=0.6, yscale=0.4]
\tikzset{edge/.style = {<-}}
\node (c) at (0, -1) {};
  \node (a) at (0,-2) {\ };
  \node (b) at (0,0) {$b$};
\end{tikzpicture}\)
.
\ea
\]
\end{example}

Write  $\iota : \ZZ[\beta]\to \mLPSet$ 
for the linear map that sends $1$ to the isomorphism class of the empty labeled poset.
Write $\epsilon : \mLPSet \to\ZZ[\beta]$
for the continuous linear map whose value at 
$[(P,\gamma)] $ is $1$ if $|P|=0$ and $0$ otherwise. 

\begin{theorem}\label{mlpset-thm}
With respect to the operations $\nabla$, $\Delta$, $\iota$, $\epsilon$ just given,
the $\ZZ[\beta]$-module $\mLPSet$ is a commutative LC-Hopf algebra.
\end{theorem}

\begin{proof}
Every covering relation in $P \sqcup Q$ is a covering relation in either $P$ or $Q$, and the relative order of labels is not altered by the disjoint union.  
Hence $[(P, \gamma) \sqcup (Q, \delta)] = [(Q, \delta) \sqcup (P, \gamma)]$ and so the product $\Delta$ is commutative.

Let $\LPSet$ denote the free $\ZZ[\beta]$-module with a basis given by all isomorphism
classes of labeled posets $[(P,\gamma)]$.
For $n \in \NN$, let $\LPSet_n$ denote the $\ZZ[\beta]$-submodule of $\LPSet$ spanned by 
isomorphism classes $[(P,\gamma)]$ with $|P|=n$.

Consider the nondegenerate bilinear form $ \LPSet \times \mLPSet \to \ZZ[\beta]$
making obvious the (pseudo)bases of isomorphism classes dual to each other.
Let $\nabla^\vee : \LPSet \to \LPSet \otimes \LPSet$, $\Delta^\vee : \LPSet \otimes \LPSet \to \LPSet$, $\iota^\vee : \LPSet \to \ZZ[\beta]$, and $\epsilon^\vee : \ZZ[\beta] \to \LPSet$
denote the respective adjoints of $\nabla$, $\Delta$, $\iota$, and $\epsilon$
relative to this form.
One has $\iota^\vee = \epsilon|_{\LPSet}$ and $\epsilon^\vee = \iota|_{\LPSet}$,
while
\[
\nabla^\vee([(P,\gamma)]) = \sum_{\substack{S\sqcup T = P \\ \text{(as posets)}}} [(S,\gamma)]\otimes [(T,\gamma)].
\]
%
It is slightly more complicated, but still straightforward, to write 
down a similar formula for $\Delta^\vee$;
the important observation is that  $\nabla^\vee$ 
maps $\LPSet_n \to \bigoplus_{i+j =n} \LPSet_i \otimes \LPSet_j$.

To show that $\mLPSet$ is an LC-Hopf algebra, it suffices to
show that the maps $\Delta^\vee$, $\nabla^\vee$, $\epsilon^\vee$, and $\iota^\vee$
make $\LPSet$ into an ordinary Hopf algebra.
It is a routine calculation to show that $\LPSet$ is at least a bialgebra.
Moreover,  $\LPSet$ is evidently graded and connected as a coalgebra
 and filtered as an algebra.
It follows that $f: = \id - \epsilon^\vee \circ \iota^\vee$ is \emph{locally $\star$-nilpotent}
in the sense of \cite[Rem.~1.4.23]{GrinbergReiner}, i.e., that 
for each $x \in \LPSet$, 
$ (\Delta^\vee)^{(k-1)} \circ f^{\otimes k} \circ (\nabla^\vee)^{(k-1)}(x) = 0$
for some sufficiently large $k=k(x)$.
The bialgebra $\LPSet$ therefore has an antipode given by
Takeuchi's formula \cite[Prop.~1.4.22]{GrinbergReiner},
and we conclude that $\LPSet$ is a Hopf algebra, as needed.
\end{proof}

One can obtain an antipode formula for $\mLPSet$ by taking the adjoint of
Takeuchi's antipode formula for $\LPSet$. Both formulas involve substantial cancellation of terms.  
This raises the following problem, which seems to be open:

\begin{problem}
Find a cancellation-free formula for the antipode of $\mLPSet$.
\end{problem}

A solution to this problem would generalize the antipode formulas 
in \cite[\S15.3]{AguiarArdila}, which roughly correspond to the case when $\beta=0$.

Up to equivalence, there is a unique labeled poset $(P,\gamma)$ 
in which $P$ is an $n$-element antichain.
The continuous linear map sending $x^n$ to this poset is an injective morphism of LC-Hopf algebras
$\ZZ[\beta]\llbracket x \rrbracket \to \mLPSet$, where $\ZZ[\beta]\llbracket x \rrbracket$ has the LC-Hopf structure
described in Example~\ref{lc-ex4}.

Define $\zetaLP : \mLPSet \to \ZZ[\beta]\llbracket t\rrbracket$ to be the continuous linear map 
with 
\be\label{incr-zeta-eq} \zetaLP([(P,\gamma)]) = \begin{cases} t^{|P|} &\text{if 
$\gamma(s) < \gamma(t)$ for all $s<t$ in $P$
} \\ 0 &\text{otherwise}.\end{cases}\ee
This is a particularly natural algebra morphism,
which makes $(\mLPSet,\zetaLP)$ into a combinatorial LC-Hopf algebra.
Theorem~\ref{abs-thm} asserts that there is a
 unique morphism
$
(\mLPSet,\zetaLP) \to (\mQSym,\zetaq)
$,
and it is an interesting and \emph{a priori} nontrivial problem to evaluate this map
at the isomorphism class of a given labeled poset.
We turn to this problem in the next section.

\subsection{Set-valued $P$-partitions}
\label{sv-sect}

\emph{Set-valued $P$-partitions} are certain maps assigning sets of integers to
the vertices of a labeled poset $(P,\gamma)$.
It will turn out that these maps exactly parametrize the monomials appearing 
in the quasisymmetric generating 
function associated to $[(P,\gamma)]$ by the unique morphism 
$
(\mLPSet,\zetaLP) \to (\mQSym,\zetaq)
$.

Let $\PSet$ denote the set of finite, nonempty subsets of 
$\PP$.
Given $S,T \in \PSet$, write $S\prec T$ if $\max(S) < \min(T)$
and $S\preceq T$ if $\max(S) \leq \min(T)$.
(In particular, $S \preceq S$ if and only if $|S| = 1$.)
For $S \in \PSet$, define $x^S = \prod_{i \in S} x_i$.

\begin{definition}[{\cite[Def.~5.4]{LamPyl}}] 
\label{svp-def}
Let $(P,\gamma)$ be a labeled poset.
A \emph{set-valued $(P,\gamma)$-partition} is a map $\sigma : P \to \PSet$
such that for each covering relation $s\lessdot t$ in $P$ one has
$\sigma(s) \preceq \sigma(t)$, with $\sigma(s) \prec \sigma(t)$ if $\gamma(s) > \gamma(t)$.
\end{definition}

\begin{example}
Suppose $P = \{1<2<3\}$ is a 3-element chain and $\gamma(1) < \gamma(2) > \gamma(3)$.
If $\sigma$ is a set-valued $(P,\gamma)$-partition,
then $(\sigma(1), \sigma(2), \sigma(3))$ could be $(\{ 2\},\{2,3\}, \{4,5\})$ or $(\{1\},\{2,4\}, \{6\})$,
but not $(\{1,2\},\{3\}, \{3,4\})$.
\end{example}

Lam and Pylyavskyy \cite{LamPyl} introduced this definition while studying a ``$K$-theoretic'' analogue of
the Malvenuto--Reutenauer Hopf algebra of permutations.
The idea
generalizes the classical notion of a \emph{$P$-partition} from \cite{St1},
which is just a set-valued partition whose values are all singleton sets.

Let $\tA(P,\gamma)$ denote the set of all set-valued $(P,\gamma)$-partitions.
The \emph{length} of $\sigma \in \tA(P,\gamma)$ is the 
nonnegative integer $|\sigma| := \sum_{s \in P} |\sigma(s)|$,
while the \emph{weight} of $\sigma$ is the monomial
$x^\sigma := \prod_{s \in P} x^{\sigma(s)}$.
We define the \emph{set-valued weight enumerator} of $(P,\gamma)$
to be the quasisymmetric formal power series 
\be\label{tgam-eq}
\tGamma({P,\gamma}) := \sum_{\sigma \in \tA(P,\gamma)} \beta^{|\sigma| - |P|} x^\sigma
\in \mQSym.
\ee
The power series $\Gamma^{(1)}(P,\gamma)$ obtained from \eqref{tgam-eq} by setting $\beta=1$
is denoted $\tilde K_{P,\gamma}$ in \cite[\S5.3]{LamPyl}.
These generating functions (and by extension, the sets $\tA(P,\gamma)$)
are natural objects to consider on account of the following theorem.

 \begin{theorem}\label{<-thm}
The continuous linear map with $[(P,\gamma)] \mapsto \tGamma(P,\gamma)$
for each labeled poset $(P,\gamma)$
is the unique morphism 
of combinatorial LC-Hopf algebras 
$(\mLPSet,\zetaLP) \to (\mQSym,\zetaq).$
\end{theorem}

 If $(P,\gamma) \cong (Q,\delta)$ then clearly $\tGamma(P,\gamma) = \tGamma(Q,\delta)$,
 so the continuous linear map  
 $\mLPSet \to \mQSym$ described in the theorem is at least well-defined.

 \begin{proof}
 Fix a  labeled poset $(P,\gamma)$.
 For each $k \in \NN$, let 
 $\sP_k$ denote the set of $k$-tuples $(P_1,\dots, P_k)$
 of nonempty sets with $P_1  \cup \dots \cup P_k = P$
 such that
 \ben
 \item[(a)] if $s \in P_i$ and $t \in P_j$ where $i<j$ then $t \not < s$ in $P$, and
 \item[(b)] if $s,t \in P_i$ and $s \lessdot t$ in $P$ then $\gamma(s) <\gamma(t)$.

\een
%
Also let $\sI_k$ be the set of $k$-tuples of positive integers $(i_1,\dots,i_k)$ with $i_1<\dots<i_k$.
According to Theorem~\ref{abs-thm},
the unique morphism of 
combinatorial LC-Hopf algebras
$ (\mLPSet,\zetaLP) \to (\mQSym,\zetaq)$ is the continuous linear map 
with
\[
[(P,\gamma)] \mapsto \sum_{k \in \NN} \sum_{(P_1, \dots, P_k) \in \sP_k} \sum_{(i_1,\dots, i_k) \in \sI_k} 
\beta^{\sum_i |P_i|  - |P|} x_{i_1}^{|P_1|} \cdots x_{i_k}^{|P_k|}.
\]
We claim that the right hand expression is equal to $ \tGamma(P,\gamma)$.
Given tuples
$\pi = (P_1,\dots,P_k) \in \sP_k$ and $I = (i_1<\dots<i_k) \in \sI_k$,
define $\sigma : P \to \PSet$ to be the map with
$\sigma(s) = \{ i_j : 1 \leq j \leq k \text{ and }s \in P_j\}$. 
If $s\lessdot t$ in $P$ and $ i,j $ are indices such that $s \in P_i$ and $t \in P_j$, then property (a) implies that $i \leq j$,
so $\sigma(s) \preceq \sigma(t)$; moreover, if $\gamma(s) > \gamma(t)$,
then  property (b) implies that $s$ and $t$ do not belong to the same $P_i$ and so $\sigma(s) \prec \sigma(t)$.
Thus $\sigma \in \tA(P,\gamma)$,
and we have $|\sigma| = \sum_i |P_i|$ and $x^\sigma = x_{i_1}^{|P_1|}  \cdots x_{i_k}^{|P_k|}$.

It suffices to show that $(\pi,I) \mapsto \sigma$ is a bijection
$\bigsqcup_{k \in \NN} \sP_k \times \sI_k \to \tA(P,\gamma)$.
This is straightforward; the inverse map is $\sigma \mapsto (\pi,I)$ 
where $I$ is the sequence of elements in $\bigcup_{s \in P} \sigma(s)$ arranged in order,
and $\pi = (P_1,P_2,\dots,P_{|I|})$ is the tuple
in which $P_i$ is the set of $s \in P$ such that $\sigma(s)$ contains the $i$th element of $I$.
It is easy to deduce from Definition~\ref{svp-def} that  $\sigma \in \tA(P,\gamma)$ implies  $\pi \in \sP_k$.
 \end{proof}

 As one application of  Theorem~\ref{<-thm}, we obtain some new
 (co)product formulas for the quasisymmetric functions $\tGamma(P,\gamma)$.
 
\begin{corollary}\label{sv-products-cor}
Suppose $(P,\gamma)$ and $(Q,\delta)$ are labeled posets.
Then
\[\tGamma(P,\gamma)\cdot \tGamma(Q,\delta) = \tGamma((P,\gamma)\sqcup (Q,\delta))\]
and 
\[\Delta(\tGamma(P,\gamma)) = \sum_{S\cup T = P} \beta^{|S\cap T|} \cdot \tGamma(S,\gamma)\otimes
\tGamma( T,\gamma)\]
where the sum is over all ordered pairs $(S,T)$ of subsets of $P$ such that 
$S$ is a lower set, $T$ is an upper set,
$P = S\cup T$, and 
$S\cap T$ is an antichain.
\end{corollary}

\subsection{Multifundamental quasisymmetric functions}\label{multi-sect}

In this section we investigate the properties of 
the generating functions $\tGamma(P,\gamma)$
in the special case when $P$ is linearly ordered.
One can show that these quasisymmetric functions form a 
 pseudobasis of $\mQSym$;
following \cite{LamPyl},
we refer to them as \emph{multifundamental quasisymmetric functions}.

Fix an arbitrary labeled poset $(P,\gamma)$.
Lam and Pylyavskyy show in \cite[\S5.3]{LamPyl}
that both $\tA(P,\gamma)$ and $\tGamma(P,\gamma)$ are controlled by the following objects:

\begin{definition}
A finite sequence $w=(w_1,w_2,\dots,w_N)$
is a \emph{linear multiextension} of $P$ 
if it holds that 
$P = \{w_1,w_2,\dots,w_N\}$,  $w_i \neq w_{i+1}$ for each $1 \leq i < N$,
and $\{ i : w_i = a\} \prec \{ i : w_i = b\}$ whenever $a\lessdot b$ in $P$.
\end{definition}

(This differs superficially from Lam and Pylyavskyy's definition in \cite[\S5.3]{LamPyl};
what they call a linear multiextension is the map sending $s \in P$ 
to $\{ i : w_i = s\}$ rather than the sequence $(w_1,w_2,\dots,w_N)$.)

Let $\tL(P)$ denote the set of linear multiextensions of $P$.
This set has a unique element if $P$ is  a chain
and is infinite otherwise.
For each integer $N \in \NN$, 
let $[N]:= \{1<2<\dots<N\}$ be the usual $N$-element chain. 
Given a finite sequence $w=(w_1,w_2,\dots,w_N)$
with an injective map $\gamma : \{w_1,w_2,\dots,w_N\}\to\ZZ$,
let $\delta : [N] \to [N]$ denote the unique bijection with $\delta(i) > \delta(j)$
for $ i < j $
if and only if $\gamma(w_i) > \gamma(w_{j})$,
and define
\[\tA(w,\gamma) := \tA([N], \delta)
\qquand
\tGamma(w,\gamma):=\tGamma([N],\delta).\]
Let $\ell(w) :=N$ denote the length of the finite sequence $w$.
The following is a sort of
``Fundamental Lemma of Set-Valued $(P,\gamma)$-Partitions.''

\begin{theorem}[{\cite[Thm.~5.6]{LamPyl}}]\label{p-thm1}
For each labeled poset $(P,\gamma)$, there is a length- and weight-preserving bijection
$\tA(P,\gamma) \xrightarrow{\sim} \bigsqcup_{w \in \tL(P)} \tA(w,\gamma)$,
and so 
\[\tGamma({P,\gamma}) = \sum_{w \in \tL(P)} \beta^{\ell(w)-|P|} \tGamma({w,\gamma}).\]
\end{theorem}

\begin{remark*}
To be precise, \cite[Thm.~5.6]{LamPyl} is the case of the preceding result with $\beta=1$;
however, both statements have the same proof.
\end{remark*}

Fix a sequence $w$ with $N=\ell(w)$ and an injective map $\gamma : \{w_1,w_2,\dots,w_N\} \to \ZZ$.
Define
$\Des(w,\gamma) := \{ i\in[N-1] : \gamma(w_i)  > \gamma(w_{i+1})\}$. 
We then have
\be\tA(w,\gamma) = \left\{ \sigma : [N] \to \PSet \hs:\hs  \ba \sigma(1)&\preceq \sigma(2)\preceq \dots \preceq \sigma(N)
\text{ and} \\
\sigma(i) &\prec \sigma(i+1)\text{ for }i \in \Des(w,\gamma)
\ea
\right\}.\ee
This suggests the following definition.
For a tuple of sets $S=(S_1,S_2,\dots, S_N)$, let $x^S := \prod_i x^{S_i}$
and $|S| := \sum_i |S_i|$.
For a composition $\alpha=(\alpha_1,\alpha_2,\dots,\alpha_k)$,
let \[I(\alpha) := \{\alpha_1, \alpha_1+\alpha_2,\dots, \alpha_1+\alpha_2 + \dots + \alpha_{k-1}\}.\]
Then $\alpha \mapsto I(\alpha)$ is a bijection from compositions of $N$ to subsets of $[N-1]$.
Define the \emph{multifundamental quasisymmetric function}
of 
$\alpha \vDash N$
to be
\be\label{lbeta-eq} L^{(\beta)}_\alpha   := \sum_{
\substack{
S=(S_1 \preceq S_2 \preceq \dots \preceq S_N) \\
S_i \prec S_{i+1}\text{ if }i \in I(\alpha)}
}
\beta^{|S| - N} x^S
\in \mQSym
\ee
where in the sum each $S_i$ belongs to $\PSet$.
The following holds by definition.

\begin{proposition}\label{by-def-prop}
Fix a choice of $(w,\gamma)$ as above, and
let $\alpha \vDash \ell(w)$ be the unique composition such that $I(\alpha)= \Des(w,\gamma)$.
Then
$L^{(\beta)}_\alpha = \tGamma(w,\gamma)$.
\end{proposition}

Write $\leq$ for the usual \emph{refinement order} on compositions, so that
$\alpha \leq \alpha'$ if and only if
$|\alpha| = |\alpha'|$ and
$I(\alpha) \subseteq I(\alpha')$.
Setting $\beta=0$ in \eqref{lbeta-eq} gives the \emph{fundamental quasisymmetric function}
$
L_\alpha := L_\alpha^{(0)} 
= \sum_{\alpha \leq \alpha'} M_{\alpha'}  \in \QSym.$
If $\beta$ has degree zero and $x_i$ has degree one,
then $L_\alpha$ is the nonzero homogeneous component of $L^{(\beta)}_\alpha$
of lowest degree.
Since the power series $L_\alpha$ form a basis of $\QSym$,
it follows  that
the functions $\{L^{(\beta)}_\alpha\}$ are a pseudobasis of $\mQSym$.

\begin{example}\label{lbeta-ex}
If $\alpha =(2,1)$ then \eqref{lbeta-eq} is the sum over all $S_1 \preceq S_2 \prec S_3$
with $S_i \in \PSet$. This translates to the somewhat more explicit formula
\[ L^{(\beta)}_{(2,1)} = \sum_{n} \sum_{i_1  < i_2<\dots <i_n}  \beta^{n-3} \(\tbinom{n-1}{2}+ \beta \sum_{1 \leq j <k\leq n} x_{i_j} \)x_{i_1}x_{i_2} \cdots x_{i_n}.\]
Setting $\beta=0$ gives $L_{(2,1)} = 
M_{(1,1,1)} + M_{(2,1)}$ as expected.
\end{example}

The finite-variable truncations of $L^{(\beta)}_{\alpha}$
are the \emph{quasisymmetric glide polynomials} in \cite{PechenikSearles},
and each $L^{(\beta)}_{\alpha}$ is a certain ``stable limit'' of Pechenik and Searles's \emph{glide polynomials}.

 Lam and Pylyavskyy \cite[\S5.3]{LamPyl} refer to the specializations $\tilde L_\alpha := L_\alpha^{(1)}$ 
as \emph{multifundamental quasisymmetric functions}.
One recovers $L_\alpha^{(\beta)}$ from $\tilde L_\alpha$ by substituting
$x_i \mapsto \beta x_i$ and then dividing by $\beta^{|\alpha|}$; that is, we have
\be\label{tilde-l-eq} \beta^{|\alpha|}  L_\alpha^{(\beta)} = \tilde L_\alpha(\beta x_1, \beta x_2, \dots).\ee
This lets one rewrite any identities in \cite{LamPyl} involving  $\tilde L_\alpha$ in terms of $L^{(\beta)}_\alpha$.

\begin{remark}\label{product-rmk}
One can use Corollary~\ref{sv-products-cor} to obtain (co)product formulas
for the multifundamental quasisymmetric functions $L^{(\beta)}_\alpha$.
Lam and Pylyavskyy have already derived such formulas
in \cite[\S5.4]{LamPyl}; their results are stated in terms of the functions
$\tilde L_\alpha := L^{(1)}_\alpha$, but easily translate to $L^{(\beta)}_\alpha$
via \eqref{tilde-l-eq}.
The coproduct $\Delta\(L^{(\beta)}_\alpha\)$ is always a finite $\NN[\beta]$-linear combination
of tensors of the form $L^{(\beta)}_{\alpha'}\otimes L^{(\beta)}_{\alpha''}$.
By contrast, the product
$\nabla\(L^{(\beta)}_{\alpha'}\otimes L^{(\beta)}_{\alpha''}\)$ is an
infinite linear combination of $L^{(\beta)}_{\alpha}$'s if $\alpha'$ and $\alpha''$
are both nonempty.
\end{remark}

\subsection{Stable Grothendieck polynomials}\label{stable-sect}

A primary motivation for the definition of set-valued $P$-partitions
comes from the labeled posets associated with Young diagrams of partitions.
As Lam and Pylyavskyy note in \cite{LamPyl},
the weight enumerators of these labeled posets are essentially
the \emph{stable Grothendieck polynomials}
studied in \cite{FominKirillov94,LS1982} (see also \cite{Buch2002}).
We review these here.

Let $\lambda =(\lambda_1\geq \lambda_2 \geq \dots \geq 0)$
and $\mu = (\mu_1\geq \mu_2\geq \dots \geq 0)$ 
be (integer)  partitions
with $\mu \subseteq \lambda$, i.e., with $\mu_i \leq \lambda_i$
for all $i$. 
The \emph{skew diagram} of $\lambda/\mu$ is
\[\D_{\lambda/\mu} := \{ (i,j) \in \PP\times \PP : \mu_i < j \leq \lambda_i\}.\]
Let $\D_\lambda = \D_{\lambda/\emptyset}$.
We consider $\D_{\lambda/\mu}$ to be partially ordered with $(i,j) \leq (i',j')$
 if $i \leq i'$ and $j \leq j'$. 
 Let $n = |\lambda|-|\mu|$
 and fix a bijection $\theta : \D_{\lambda/\mu} \to [n]$
with
\be\label{canonical-eq}
\theta(i,j) < \theta(i,j+1)
\qquand \theta(i,j) > \theta(i+1,j)
\ee
for all relevant positions in $\D_{\lambda/\mu}$.
(The choice of $\theta$ is unimportant,
since all such labelings are equivalent.)
For example, if $\lambda =(5,4,2)$ and $\mu=(2,1)$ then 
the oriented Hasse diagram representing $(\D_{\lambda/\mu},\theta)$ is 
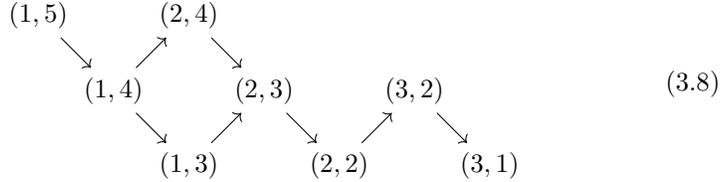
\be\label{dlam-eq}
\begin{tikzpicture}[baseline=(center.base), xscale=1, yscale=1]
\tikzset{edge/.style = {->}}
\node (center) at (0, -1) {};
  \node (a) at (0,0) {$(1,5)$};
  \node (b) at (1,-1) {$(1,4)$};
  \node (c) at (2,0) {$(2,4)$};
  \node (d) at (2,-2) {$(1,3)$};
  \node (e) at (3,-1) {$(2,3)$};
  \node (g) at (4,-2) {$(2,2)$};
  \node (h) at (5,-1) {$(3,2)$};
  \node (i) at (6,-2) {$(3,1)$};
  \draw[edge] (a) -- (b);
  \draw[edge] (b) -- (d);
  \draw[edge] (c) -- (e);
  \draw[edge] (e) -- (g);
  \draw[edge] (h) -- (i);
  \draw[edge] (b) -- (c);
  \draw[edge] (d) -- (e);
  \draw[edge] (g) -- (h);
\end{tikzpicture}
\ee
 
 The elements of $\tA(\D_{\lambda/\mu},\theta)$
 may be identified with \emph{semistandard set-valued tableaux} of shape $\lambda/\mu$ as defined in \cite[\S3]{Buch2002},
 i.e., fillings of $\D_{\lambda/\mu}$ by nonempty finite subsets of positive integers that are weakly increasing 
(in the sense of $\preceq$) along rows and strictly increasing (in the sense of $\prec$) along columns.
Let 
\[ \SetSSYT(\lambda/\mu) := \tA(\D_{\lambda/\mu},\theta).\]
The \emph{stable Grothendieck polynomial} of $\lambda/\mu$ is then
 \be\label{tk-eq} 
 G^{(\beta)}_{\lambda/\mu} := \tGamma({\D_{\lambda/\mu},\theta}) = 
 \sum_{T \in \SetSSYT(\lambda/\mu)} \beta^{|T|-|\lambda/\mu|} x^T.
 \ee
Both definitions are independent of the choice of $\theta$. 

 The power series $G^{(\beta)}_{\lambda/\mu}$ is symmetric in the $x_i$ variables, though of unbounded degree (see \cite[\S2 and Thm.~3.1]{Buch2002}).
The formula \eqref{tk-eq} sometimes appears in the literature with  $\beta$ set to $\pm 1$ \cite{Buch2002,BuchSamuel,LamPyl}.
There is no loss of generality in making such a specialization, but it is more convenient to work with a generic parameter.
Setting $\beta=0$ in \eqref{tk-eq} gives the skew Schur function $s_{\lambda/\mu}$.

Say that a set-valued tableau $T \in \SetSSYT(\lambda/\mu)$ is \emph{standard}
if its entries are disjoint sets, \emph{not containing any consecutive integers},
with union $\{1,2,\dots,N\}$ for some $N\geq n$.
We identify the set $\SetSYT(\lambda/\mu)$ of standard set-valued tableaux of shape $\lambda/\mu$
with the set $\tL(\D_{\lambda/\mu})$ of linear multiextensions of $\D_{\lambda/\mu}$:
each $(w_1,w_2,\dots,w_N) \in \tL(\D_{\lambda/\mu})$
 corresponds to the standard set-valued tableau with $i$ in box $w_i$.
 Then it follows from Theorem~\ref{p-thm1} that
 \be\label{tilde-G-eq}
G^{(\beta)}_{\lambda/\mu} = \sum_\alpha  f^\alpha_{\lambda/\mu} \cdot \beta^{|\alpha| - |\lambda/\mu|} \cdot L^{(\beta)}_\alpha,
 \ee
 where the sum is over compositions $\alpha$ and $ f^\alpha_{\lambda/\mu}$ is the number of 
 standard set-valued tableaux $T \in \SetSYT(\lambda/\mu)$ with $\Des(T,\theta)=I(\alpha)$
 and $|T|=|\alpha|$.

\section{Enriched $P$-partitions}\label{main-sect}

In \cite{Stembridge1997a}, Stembridge introduced an ``enriched'' analogue of ordinary (i.e., not set-valued)
$P$-partitions as a means of constructing skew Schur $Q$-functions.   Stembridge's theory admits a $K$-theoretic analogue, which we describe in this section.
This leads to interesting quasisymmetric 
generalizations of Ikeda and Naruse's \emph{shifted stable Grothendieck polynomials}  \cite{IkedaNaruse}, discussed in Section~\ref{shifted-stable-sect}.
 
\subsection{Labeled posets revisited}

We begin by embedding the LC-Hopf algebra $\mLPSet$ from Section~\ref{mlpset-sect} in a larger algebra.
Given a labeled poset $(P,\gamma)$,
a \emph{valley} $v \in P$ is an element with the property that 
$\gamma(x) > \gamma(v)$ for all covers $x\lessdot v$ in $P$
and $\gamma(v) < \gamma(y)$ for all covers $v \lessdot y$ in $P$.
In other words, a valley is a sink in the oriented Hasse diagram of $(P, \gamma)$.
Define $\ValSet(P,\gamma)$ to be the set of valleys of $(P, \gamma)$.

\begin{definition}
Let $\mLPSet^+$ be the linearly compact $\ZZ[\beta]$-module
with a pseudobasis 
consisting of the
 isomorphism classes $[(P,\gamma,V)]$
for labeled posets $(P,\gamma)$ and subsets $V\subseteq \ValSet(P,\gamma)$,
where we define $(P,\gamma,V) \cong (P',\gamma',V')$
if there exists an isomorphism of labeled posets $ (P,\gamma) \xrightarrow{\sim} (P',\gamma')$ taking $V$ to $V'$.
\end{definition}

The Hopf structure on $\mLPSet$ extends to $\mLPSet^+$ as follows.
Let $\nabla : \mLPSet^+ \htimes \mLPSet^+ \to \mLPSet^+$ 
denote the continuous linear map with
\[
\nabla([(P,\gamma,V)]\otimes [(Q,\delta,W)]) := [(P,\gamma,V)\sqcup (Q,\delta,W)]
\]
where
$(P,\gamma,V)\sqcup (Q,\delta,W) := (P\sqcup Q,\gamma\sqcup \delta,V\sqcup W)$,
with $\gamma\sqcup \delta$ as in \eqref{sqcup-eq}.
Let $\Delta :  \mLPSet^+ \to \mLPSet^+ \htimes \mLPSet^+$
be the continuous linear map with
\[ \Delta([(P,\gamma,V)]) := \sum_{S\cup T = P}\beta^{|S\cap T|}\cdot [(S,\gamma,V\cap S)] \otimes [(T,\gamma,V\cap T)]\]
where the sum is over all ordered pairs $(S,T)$ of subsets of $P$ such that 
$S$ is a lower set, $T$ is an upper set,
$P = S\cup T$, and
$S\cap T$ is an antichain.
Finally, write  $\iota : \ZZ[\beta]\to \mLPSet^+$ 
for the linear map with $1\mapsto [(\varnothing, \varnothing,\varnothing)]$
and $\epsilon : \mLPSet^+ \to\ZZ[\beta]$
for the continuous linear map whose value at 
 $[(P,\gamma,V)] $ is $1$ if $P$ is empty and $0$ otherwise.

\begin{theorem}
With respect to the operations $\nabla$, $\Delta$, $\iota$, $\epsilon$ just given,
the $\ZZ[\beta]$-module $\mLPSet^+$ is a commutative LC-Hopf algebra.
\end{theorem}

\begin{proof}
The proof is the same as for Theorem~\ref{mlpset-thm},
\emph{mutatis mutandis}.
\end{proof}

Identifying  $[(P,\gamma)] = [(P,\gamma,\varnothing)]$
lets us view $\mLPSet \subset \mLPSet^+$ as LC-Hopf algebras.

We now define a map $\uzetaLP$ that turns $\mLPSet^+$ into a combinatorial LC-Hopf algebra. 
It is slightly different from the structure on $\mLPSet$ considered in the previous section.

The map $\zetaLP$ defined in \eqref{incr-zeta-eq} extends to
a continuous algebra morphism $\mLPSet^+ \to \ZZ[\beta]\llbracket t\rrbracket$
with $\zetaLP([(P,\gamma,V)]) := \zetaLP([(P,\gamma)])$.
Define $\dzetaLP$ to be the continuous algebra morphism $\mLPSet^+ \to \ZZ[\beta]\llbracket t\rrbracket$
with 
\[ \dzetaLP([(P,\gamma,V)]) := \begin{cases} 
t^{|P|} &\text{if $V=\varnothing$ and $\gamma(x) > \gamma(y)$ for all $x\lessdot y$ in $P$}, \\
0&\text{otherwise}.
\end{cases}
\]
Finally, write $\uzetaLP : \mLPSet^+ \to \ZZ[\beta]\llbracket t\rrbracket$ for the convolution product 
\be\label{tilde-unimodal-zeta-eq}
\uzetaLP := \dzetaLP * \zetaLP =  \nabla_{\ZZ[\beta]\llbracket t\rrbracket} \circ (\dzetaLP \htimes \zetaLP)\circ \Delta.
\ee
Next, we derive a more explicit formula for this algebra morphism.

For a labeled poset $(P,\gamma)$, let $\PeakSet(P,\gamma)$
denote the set of elements $y \in P$ for which there exist elements $x,z \in P$
with $x\lessdot y \lessdot z$ and $\gamma(x) < \gamma(y) > \gamma(z)$.

\begin{proposition}\label{><-prop}
Let $(P,\gamma)$ be a labeled poset and $V \subseteq \ValSet(P,\gamma)$.
Then 
\[ \uzetaLP([(P,\gamma,V)]) = 
\begin{cases}
 t^{|P|} (2+\beta t)^{|\ValSet(P,\gamma)| - |V|}&\text{if }\PeakSet(P,\gamma)=\varnothing, \\
0&\text{otherwise}.
\end{cases}
\]
\end{proposition}

\begin{proof}
Each pair $(S, T)$ of subsets of $P$ such that $S$ is a lower set, $T$ is an upper set,
$P = S\cup T$, and
$S\cap T$ is an antichain contributes
\be\label{term-eq}
\beta^{|S\cap T|} \cdot \dzetaLP([(S,\gamma,V\cap S)]) \cdot \zetaLP([(T,\gamma)])
\ee
to the value of $\uzetaLP([(P,\gamma,V)])$.  

If there is some element $y \in \PeakSet(P,\gamma)$, then for every such pair $(S, T)$, either $y \in S$ and there is $x\lessdot y$ in $S$ with $\gamma(x) <\gamma(y)$
or $y \in T$ and there is $z \gtrdot y$ in $T$ with $\gamma(y) > \gamma(z)$.  Thus one of
$\dzetaLP([(S,\gamma,V\cap S)])$ or $ \zetaLP([(T,\gamma)])$ is always zero,
and consequently $\uzetaLP([(P,\gamma,V)])=0$, as claimed.

Otherwise, we may assume $\PeakSet(P,\gamma) = \varnothing$. 
We construct a collection of pairs $(S, T)$ of subsets of $P$, as follows:
\begin{enumerate}[(i)]
\item if there is $z \gtrdot y $ in $P$ and $\gamma(y) > \gamma(z)$ then $y \in S$;
\item if there is $x \lessdot y$ in $P$ and $\gamma(x) < \gamma(y)$ then $y \in T$; 
\item if $y \in V$ then $y \in T$; and
\item if $y \in \ValSet(P,\gamma)\setminus V$ then we may assign $y$ to $S$, to $T$, or to both.
\end{enumerate}
Since $\PeakSet(P,\gamma) = \varnothing$, these rules are disjoint.
Suppose $(S, T)$ is one of the pairs so constructed.  
If $y \in S$ and $x\lessdot y$, then rules (ii) and (iv) imply that 
$\gamma(x) > \gamma(y)$, and rule (i) implies that $x \in S$.
Thus $S$ is a lower set with $\dzetaLP([(S,\gamma,\varnothing)]) = t^{|S|}$.
It follows similarly that $T$ is an upper set
with $ \zetaLP([(T,\gamma)])=t^{|T|}$.
Finally, it is clear that $S\cap T$ is an antichain
contained in $\ValSet(P,\gamma)\setminus V$.  Thus, the total contribution \eqref{term-eq} of the pair $(S, T)$ is $t^{|P|} (\beta t)^{|S\cap T|}$.  
Moreover, any pair $(S, T)$ with nonzero contribution must be constructed in this way.  
Since constructing such a pair $(S,T)$ is equivalent to choosing,
independently
for each $v \in \ValSet(P,\gamma)\setminus V$,
whether $v \in S \setminus T$ or $v \in T \setminus S$ or $v \in S\cap T$,
the proposition follows.
\end{proof}

\subsection{Enriched $P$-partitions}

In Section~\ref{sv-sect}, we encountered set-valued $P$-partitions 
by considering Theorem~\ref{abs-thm} applied to the combinatorial 
LC-Hopf algebra
$(\mLPSet,\zetaLP)$.
To obtain ``enriched'' analogues of those definitions,
we proceed in a similar way but  
now consider
$(\mLPSet^+,\uzetaLP)$
in place of $(\mLPSet,\zetaLP)$.

Let $\MM :=\{1'<1<2'<2<\dots\}$ denote
the totally ordered \emph{marked alphabet}. 
Let $\MSet$ denote the set of finite, nonempty subsets of $\MM$.
For $i \in\PP$, let $|i'|: = |i| = i$, and define $x^S := \prod_{i \in S} x_{|i|}$ for $S \in \MSet$.
For $S,T \in \MSet$, write $S\prec T$ if $\max(S) < \min(T)$ and $S \preceq T$ if $\max(S) \leq \min(T)$ (just as before).

\begin{definition}
\label{enriched-def}
Let $(P,\gamma)$ be a labeled poset.
 An \emph{enriched set-valued $(P,\gamma)$-partition} is a map $\sigma : P \to \MSet$
such that for each covering relation $s\lessdot t$ in $P$, 
one has $\sigma(s) \preceq \sigma(t)$ and the following properties hold:
\ben
\item[(a)] if $\gamma(s) < \gamma(t)$ then 
$\sigma(s)\cap \sigma(t) \subset \{1,2,3,\dots\}$, and

\item[(b)] if $\gamma(s) > \gamma(t)$ then  
$\sigma(s)\cap \sigma(t) \subset \{1',2',3',\dots\}$.

\een
\end{definition}

This is a generalization of Stembridge's definition of an \emph{enriched $P$-partition} \cite[\S2]{Stembridge1997a},
which corresponds to the case when $|\sigma(s)| =1$ for all $s \in P$. 
(For other generalizations of this notion, see \cite{Petersen}.)
A set-valued partition
in the sense of Definition~\ref{svp-def} is an
enriched set-valued $(P,\gamma)$-partition with values in $\PSet$.

\begin{example}
Suppose $P = \{1<2<3\}$ and $\gamma(1) < \gamma(2) > \gamma(3)$.
If $\sigma $ is an enriched set-valued $(P,\gamma)$-partition,
then possible values for the sequence  $(\sigma(1), \sigma(2), \sigma(3))$
include $(\{ 2',2\},\{2,3'\}, \{3',3\})$, $(\{2'\},\{2,3'\}, \{3,4',4\})$, 
and $(\{1',2'\}, \{3'\},\{3'\})$,
but not
$(\{1,2'\},\{2',3'\}, \{3,4'\})$
or
$(\{1,2'\},\{2,3\}, \{3,4'\})$.
\end{example}

For a given labeled poset $(P,\gamma)$,
let $\tE(P,\gamma)$ denote the set of all enriched set-valued $(P,\gamma)$-partitions.
For each subset $V \subseteq \ValSet(P,\gamma)$, define 
\be
\tE(P,\gamma, V) := \{ \sigma \in \tE(P,\gamma) : \sigma(v) \subset \PP\text{ for all }v \in V\},
\ee
so that $\tE(P,\gamma) = \tE(P,\gamma,\varnothing)$.
Define the \emph{length} and \emph{weight} of $\sigma \in \tE(P,\gamma)$ 
to be $|\sigma| := \sum_{s \in P} |\sigma(s)|$
and $x^\sigma := \prod_{s \in P} x^{\sigma(s)}$.
The \emph{enriched set-valued weight enumerator} of the triple $(P,\gamma,V)$ is
the quasisymmetric formal power series 
\be
\tOmega(P,\gamma,V) := \sum_{\sigma \in \tE(P,\gamma,V)} \beta^{|\sigma|-|P|} x^\sigma \in \mQSym.\ee
These definitions are natural in view of the following analogue of Theorem~\ref{<-thm}.

 \begin{theorem}\label{><-thm}
The continuous linear map with $[(P,\gamma,V)] \mapsto \tOmega(P,\gamma,V)$
for each labeled poset $(P,\gamma)$ and subset $V\subseteq\ValSet(P,\gamma)$
is the unique morphism of combinatorial LC-Hopf algebras
$(\mLPSet^+,\uzetaLP) \to (\mQSym,\zetaq).$
\end{theorem}

 \begin{proof}
The result follows by a calculation similar to the proof of Theorem~\ref{<-thm}.
Fix a  labeled poset $(P,\gamma)$ and a subset $V \subseteq\ValSet(P,\gamma)$.
 For each $k \in \NN$, let 
 $\sQ_k$ denote the set of $2k$-tuples $\pi=(Q_1',Q_1,\dots, Q_k',Q_k)$
 of sets with $Q_i' \cup Q_i\neq \varnothing$ for all $i\in[k]$ and $Q_1'\cup Q_1  \cup \dots \cup Q_k'\cup Q_k = P$
such that
 \ben[(a)]
\item if $s \in Q_i'\cup Q_i$ and $t \in Q_j'\cup Q_j$ where $i<j$
 then $t \not < s$ in $P$, 
 \item if $s \in Q_i'$ and $t \in Q_i$,
 then $t \not < s$ in $P$,
 \item if $s,t \in Q_i'$ and $s \lessdot t$ in $P$ then $\gamma(s) >\gamma(t)$, 
  \item if $s,t \in Q_i$ and $s \lessdot t$ in $P$ then $\gamma(s) <\gamma(t)$, and
 \item $V \subseteq P \setminus (Q_1' \cup Q_2' \cup \dots \cup Q_k')$.
\een
Define $\sI_k$
to
 be the set of $k$-tuples of positive integers $I=(i_1,i_2,\dots,i_k)$ with $i_1<i_2<\dots<i_k$.
 Given $\pi  \in \sQ_k$ and $I \in \sI_k$,
 let  $|\pi| := |Q_1'| + |Q_1| + \dots + | Q_k'| + |Q_k|$ and 
$x^{(\pi,I)} := x_{i_1}^{|Q_1'| + |Q_1|} \cdots x_{i_k}^{|Q_k'| + |Q_k|}$.
According to Theorem~\ref{abs-thm},
the unique morphism 
$\Phi : (\mLPSet^+,\zetaLP) \to (\mQSym,\zetaq)$ is the continuous linear map 
with
\[
[(P,\gamma,V)] \mapsto \sum_{k \in \NN} \sum_{(\pi,I) \in \sQ_k \times \sI_k} 
\beta^{|\pi|  - |P|} x^{(\pi,I)} .
\]
We claim that the right side is $ \tOmega(P,\gamma,V)$.
For
$\pi = (Q_1',Q_1,\dots,Q_k',Q_k) \in \sQ_k$ and $I = (i_1<\dots<i_k) \in \sI_k$,
define $\sigma : P \to \MSet$ to be the map with
\[\sigma(s) = \{i_j' : j \in [k]\text{ and } s \in Q_j'\} \cup \{ i_j : j \in [k]\text{ and }s \in Q_j\}
\quad\text{for $s \in P$.}\]
Properties (a) and (b) imply that if $s,t\in P$ are such that $s\lessdot t$
then $\sigma(s) \preceq \sigma(t)$.
Given this fact, it is easy to see that properties (c)-(e) imply that
$\sigma \in \tE(P,\gamma,V)$.
Clearly $|\sigma| = |\pi|$ and $x^\sigma = x^{(\pi,I)}$.

It suffices to show that $(\pi,I) \mapsto \sigma$ is a bijection
$\bigsqcup_{k \in \NN} \sQ_k \times \sI_k \to \tE(P,\gamma,V)$.
This is straightforward: the inverse map is $\sigma \mapsto (\pi,I)$ 
where $I$ is the sequence of elements in $\{i_1<i_2<\dots<i_k\} := \bigcup_{s \in P} \{ |i| : i \in \sigma(s)\}$
arranged in order,
and $\pi = (Q_1',Q_1,\dots,Q_k',Q_k)$ is the tuple
in which $Q_j'$ and $Q_j$ are the sets consisting of the elements $s \in P$
with  $i_j' \in \sigma(s)$ and $i_j \in \sigma(s)$, respectively.
\end{proof}

The theorem implies an analogue of Corollary~\ref{sv-products-cor}.
 
\begin{corollary}\label{esv-products-cor}
If $(P,\gamma,V)$ and $(Q,\delta,W)$ are labeled posets
then
\[\tOmega(P,\gamma,V)\cdot \tOmega(Q,\delta,W) = \tOmega((P,\gamma,V)\sqcup (Q,\delta,W))\]
and
\[\Delta(\tOmega(P,\gamma,V)) = \sum_{S\cup T = P} \beta^{|S\cap T|}\cdot\tOmega(S,\gamma,V\cap S)\otimes
\tOmega(T,\gamma,V\cap T)\]
where the sum is over 
all ordered pairs $(S,T)$ of subsets of $P$ such that 
$S$ is a lower set, $T$ is an upper set,
$P = S\cup T$, and
$S\cap T$ is an antichain.
\end{corollary}

Two special cases of $\tOmega(P,\gamma,V)$ deserve their own notation. Namely, let
\be
\tOmega(P,\gamma) := \tOmega(P,\gamma,\varnothing)
\quand
\oOmega(P,\gamma) := \tOmega(P,\gamma,\ValSet(P,\gamma)).
\ee
In this notation, when $V=\varnothing$, Corollary~\ref{esv-products-cor} gives
\be\Delta(\tOmega(P,\gamma)) = \sum_{S\cup T = P} \beta^{|S\cap T|}\cdot \tOmega(S,\gamma)\otimes
\tOmega( T,\gamma )\ee
with the summation as above.
The same formula does not hold if we replace $\tOmega$ by $\oOmega$,
since one may have
$\ValSet(P,\gamma) \cap S \subsetneq \ValSet(S,\gamma)$.
However, it does hold that
\be\oOmega(P,\gamma)\cdot \oOmega(Q,\delta) = \oOmega((P,\gamma)\sqcup (Q,\delta));\ee
this follows by setting  $V =\ValSet(P, \gamma)$ and $W = \ValSet(Q, \delta)$ in Corollary~\ref{esv-products-cor}. 

\subsection{Multipeak quasisymmetric functions}\label{multipeak-sect}

For an arbitrary labeled poset $(P, \gamma)$, the structures of $\tE(P,\gamma)$ and $\tOmega(P,\gamma)$ are again determined by 
the set  $\tL(P)$ of linear multiextensions of $P$.

Given an arbitrary finite sequence $w=(w_1,w_2,\dots,w_N)$ with an injective
map $\gamma : \{w_1,w_2,\dots,w_N\} \to \ZZ$, 
let $\delta : [N] \to [N]$ denote the 
unique bijection such that  if $i<j$ then
\be
\label{delta-eq}
\delta(i)>\delta(j) 
\qquad\text{if and only if} 
\qquad \gamma(w_i) > \gamma(w_{j}),
\ee
and define
\[\tE(w,\gamma) := \tE([N],\delta)
\qquand
\tOmega(w,\gamma) := \tOmega([N],\delta).
\]
The set $\tE(w,\gamma)$ consists of all maps $\sigma : [N] \to \MSet$
with $\sigma(1) \preceq \dots \preceq \sigma(N)$ such that, for each $i \in [N-1]$, 
the following properties hold:
\ben
\item[(a)] if $i\notin \Des(w,\gamma)$
then $\sigma(w_i)\cap \sigma(w_{i+1})\subset  \{1,2,3,\dots\}$, and
\item[(b)] if $i\in \Des(w,\gamma)$ then $\sigma(w_i) \cap \sigma(w_{i+1}) \subset \{1',2',3',\dots\}$.
\een

\begin{theorem}\label{ep-thm1}
For any labeled poset $(P,\gamma)$, there is a length- and weight-preserving bijection
$\tE(P,\gamma) \xrightarrow{\sim} \bigsqcup_{w \in \tL(P)} \tE(w,\gamma)$
and consequently 
\[
\tOmega({P,\gamma}) = \sum_{w \in \tL(P)} \beta^{\ell(w) -|P|} \tOmega({w,\gamma}).
\]
\end{theorem}

Setting $\beta=0$ here recovers 
\cite[Lem.~2.1]{Stembridge1997a}.

\begin{proof}
To make sense of the constructions that follow, 
it may be helpful to consult Example~\ref{enriched-ex}.
Fix $\sigma \in \tE(P,\gamma)$. For each $i \in \PP$,
let 
\[a^{(i)} = (a^{(i)}_1,a^{(i)}_2,\dots,a^{(i)}_{M_i})
\quand
b^{(i)} = (b^{(i)}_1,b^{(i)}_2,\dots,b^{(i)}_{N_i})
\]
 be the sequences of elements $s \in P$ with
$i' \in \sigma(s)$ and $i \in \sigma(s)$, respectively,
arranged so that
$
 \gamma(a^{(i)}_1)>  \dots > \gamma(a^{(i)}_{M_i})
 $
 and
 $
 \gamma(b^{(i)}_1)<   \dots < \gamma(b^{(i)}_{N_i}).
$
The concatenation $a^{(1)} b^{(1)} a^{(2)} b^{(2)}\cdots$ may have adjacent repeated entries, 
but omitting these repetitions produces a linear multiextension $w=(w_1,\dots,w_N) \in \tL(P)$, and
there is a unique non-decreasing surjective map 
$ \fk t : [\sum_i M_i + \sum_i N_i ] \to [N]$
such that $a^{(1)} b^{(1)} a^{(2)} b^{(2)}\cdots = (w_{\fk t(1)}, w_{\fk t(2)}, w_{\fk t(3)},  \cdots)$.
For each $j \in [N]$,
define $\tau(j) \in \MSet$ to be the set containing
$i'$ for $i \in \PP$ if and only if 
\[j \in \fk t\( M_1 + N_1 + \dots + M_{i-1} + N_{i-1} + [M_i]\)\]
and containing $i \in \PP$ if and only if 
\[j \in \fk t\( M_1 + N_1 + \dots + M_{i-1} + N_{i-1} + M_i + [N_i]\).\]
In other words, $\tau(j)$ contains $i'$ (respectively, $i$)
precisely when $a^{(i)}$ (respectively, $b^{(i)}$) contributes the $j$th entry of $w$.
This defines a map $\tau : [N] \to \MSet$.

By construction, $\tau(j) \preceq \tau(j+1)$ for all $j \in [N-1]$.
Write $\delta : [N] \to [N]$ for the unique bijection satisfying \eqref{delta-eq} for $i < j$.
If $\max(\tau(j)) = \min(\tau(j+1)) = i \in \PP$, then $b^{(i)}$ contains $(w_j,w_{j+1})$
as a consecutive subsequence, so $\gamma(w_j) < \gamma(w_{j+1})$
and $\delta(j) < \delta(j+1)$.
If $\max(\tau(j)) = \min(\tau(j+1)) = i'$ for some $i \in \PP$, then $u^{(i)}$ contains $(w_j,w_{j+1})$
as a consecutive subsequence,
so $\gamma(w_j) > \gamma(w_{j+1})$
and $\delta(j) > \delta(j+1)$.
We conclude that  $\tau \in \tE(w,\gamma)$. 
 
Given $w \in \tL(P)$ and $\tau \in \tE(w,\gamma)$, we reconstruct $\sigma$
as the map $s \mapsto \bigcup_{j: w_j = s} \tau(j)$.
The correspondence $\sigma \mapsto (w,\tau)$ is then a bijection from $\cA(P,\gamma)$ to the set of pairs $(w,\tau)$ with
$w \in \tL(P)$ and $\tau \in \tE(w,\gamma)$.  Since it holds by construction that $|\sigma| = |\tau|$ and $x^\sigma = x^\tau$, the theorem follows.
\end{proof}


\begin{example}\label{enriched-ex}
Let $(P, \gamma)$ be the labeled poset of Example~\ref{ori-ex}, 
so that $P=\{s_1,s_2,s_3,s_4\}$ with covering relations
$s_1\lessdot s_2$ and $s_1\lessdot s_3$ and $s_2\lessdot s_4$ and $s_3\lessdot s_4$,
and $\gamma(s_1) = 5$ and $\gamma(s_i) = i$ for $i\in\{2,3,4\}$. 
The map $\sigma : P \to \MSet$ with 
\[
\sigma(s_1) = \{ 1',1,2'\},
\quad
\sigma(s_2) = \{ 2', 2, 3'\},
\quad
\sigma(s_3) = \{2', 3\},
\quad
\sigma(s_4) = \{3\}
\]
is an element of $\tE(P,\gamma)$, illustrated below:
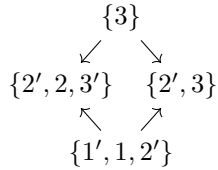
\[
\begin{tikzpicture}[baseline=(c2.base), xscale=0.4, yscale=0.3]
\tikzset{edge/.style = {->}}
  \node (c4) at (0,3) {$\{3\}$};
\node (c3) at (2,0) {$\{2', 3\}$};
\node (c2) at (-2,0) {$\{2', 2, 3'\}$};
  \node (c1) at (0,-3) {$\{1', 1, 2'\}$};
  \draw[edge] (c1) -- (c2);
  \draw[edge] (c1) -- (c3);
  \draw[edge] (c4) -- (c2);
  \draw[edge] (c4) -- (c3);
\end{tikzpicture}
\]
In the notation of the proof of Theorem~\ref{ep-thm1}, 
we have $a^{(1)} =b^{(1)}= (s_1)$, $a^{(2)} = (s_1,s_3,s_2)$, 
$b^{(2)} =  a^{(3)} = (s_2)$, $b^{(3)} = (s_3,s_4)$, and $a^{(i)} = b^{(i)} = \emptyset$ for $i>3$.
Thus
\[ \ba
a^{(1)} b^{(1)} a^{(2)} b^{(2)}\cdots
&=
(s_1;s_1;s_1,s_3,s_2; s_2;s_2;s_3,s_4),
\\
w &= (s_1,s_3,s_2, s_3,s_4), 
\\
(\delta(1),\delta(2),\delta(3),\delta(4),\delta(5)) &= (5, 2,1,3,4).
\ea
\]
The  non-decreasing surjective map $\fk t :\{1,2,\dots,9\} \to \{1,2,3,4,5\}$
has \[
(\fk t(1), \fk t(2),\dots, \fk t(9)) = (1, 1, 1, 2, 3, 3, 3, 4, 5)
\]
and $\tau : \{1,2,3,4,5\} \to \MSet$ is the map with
\[
\tau(1) = \{1',1,2'\},
\quad
\tau(2) = \{2'\},
\quad
\tau(3) = \{2',2,3'\},
\quad
\tau(4) = \tau(5)=\{3\}.
\]
As expected, we have $\tau \in \tE(w,\gamma)$ and $|\sigma|=|\tau| = 9$ with $x^\sigma = x^\tau = x_1^2 x_2^4 x_3^3$.
\end{example}


Fix a sequence $w$ of length $N$ and an injective map $\gamma : \{w_1,w_2,\dots,w_N\} \to \ZZ$.
It turns out that $\tOmega(w,\gamma)$ depends only on the \emph{peak set} of $w$,
given by
\[ \ba
\PeakSet(w,\gamma) &:= \{ i : 1 < i < N,\ \gamma(w_{i-1}) \leq \gamma(w_i) > \gamma(w_{i+1})\}
\\&
= \{ i \in\Des(w,\gamma): i > 1, i-1\notin \Des(w,\gamma)\}.
\ea
\]
(Since we allow $w$ to have repeated entries, the inequality ``$\leq$'' on the first line is meaningful and may sometimes be an equality.)
To express this precisely, we introduce an ``enriched'' analogue of $L^{(\beta)}_\alpha$.
 
The set $\PeakSet(w,\gamma)$ is a finite subset of $\PP$ that does not contain $1$ or any two consecutive integers;
we refer to such sets as \emph{peak sets}.
A \emph{peak composition} is a composition $\alpha$ for which $I(\alpha)$ is a peak set. Equivalently,
$\alpha$ is a peak composition if and only if $\alpha_i \geq 2$ for $1\leq i < \ell(\alpha)$.
We define the \emph{multipeak quasisymmetric function} of a peak composition $\alpha \vDash N$ to be
\be\label{kbeta-eq} 
K^{(\beta)}_\alpha   := \sum_{
S
}
\beta^{|S|-N} x^S
 \in \mQSym
\ee
where the sum is over $N$-tuples $S=(S_1 , S_2 , \dots, S_N)$ of sets $S_i \in \MSet$
with
\begin{itemize}
\item $S_1 \preceq S_2 \preceq \dots\preceq S_N$,
\item $S_i \cap S_{i+1} \subset \{1',2',3',\dots\}$ if $i \in I(\alpha)$, and
\item $S_i \cap S_{i+1} \subset \{1,2,3,\dots\}$  if $i \notin I(\alpha)$.
\end{itemize}
\begin{example}\label{kbeta-ex}
If $\alpha=(2,1)$ then 
\eqref{kbeta-eq} is the sum over all triples $S_1 \preceq S_2 \preceq S_3$
in $\MSet$ with $S_1 \cap S_2 \subset \{1,2,\dots\}$ and $S_2 \cap S_3\subset \{1',2',\dots\}$. 
One can show that in this case \eqref{kbeta-eq} is equivalent
 to the formula
\[ K^{(\beta)}_{(2,1)}
=\sum_{n} \sum_{i_1  <i_2<\dots <i_n}  \beta^{n-3} \(\tbinom{n-1}{2} +\beta \sum_{1\leq j < k \leq n} x_{i_j} \oplus x_{i_k} \) \prod_{j=1}^n x_{i_j}\oplus x_{i_j} \]
where
$x\oplus y := x + y + \beta xy$. There is an amusing formal similarity between this expression and the one
for $L^{(\beta)}_{(2,1)}$ in Example~\ref{lbeta-ex}.
\end{example}

Here we have an analogue of 
Proposition~\ref{by-def-prop}:

\begin{proposition}\label{k1-prop}
Fix a choice of $(w,\gamma)$ as above, and
let $\alpha \vDash \ell(w)$ be the unique composition such that $I(\alpha)= \PeakSet(w,\gamma)$.
Then $K^{(\beta)}_\alpha = \tOmega(w,\gamma)$.
\end{proposition}

\begin{proof}
Let $w$ be a sequence of length $N$, let $\gamma : \{w_1,w_2,\dots,w_N\} \to \ZZ$ be an injective map,
and let $\alpha\vDash N$ be the composition with $I(\alpha) = \PeakSet(w,\gamma)$.
If $\PeakSet(w,\gamma) = \Des(w,\gamma)$
then
 $\tOmega(w,\gamma) = K^{(\beta)}_\alpha$
holds by definition.

Suppose $1<i<N$ is such that $i-1,i \in \Des(w,\gamma)$ and $i+1 \notin\Des(w,\gamma)$.
Let $w'$ be a word of length $N$ with an injective map
$\gamma'  : \{w_1',w_2',\dots,w_N'\}\to\ZZ $ such that $\Des(w',\gamma') = \Des(w,\gamma) \setminus \{i\}$.
Then $\PeakSet(w',\gamma') = \PeakSet(w,\gamma) \neq \Des(w,\gamma)$
and $\sigma(i) \cap \sigma(i+1) \subseteq\{1',2',\dots\}$
for all $\sigma \in \tE(w,\gamma)$.
It suffices to show that $\tOmega(w,\gamma) = \tOmega(w',\gamma')$, 
since iterating this identity 
lets us assume without loss of generality that $\PeakSet(w,\gamma) = \Des(w,\gamma)$.

To this end, it is enough to construct 
a length- and weight-preserving bijection
$\tE(w,\gamma) \xrightarrow{\sim} \tE(w',\gamma')$.
For $\sigma \in \tE(w,\gamma)$, define $\sigma' : [N] \to \MSet$ as follows.
Set $\sigma'(j) = \sigma(j)$ for $j \notin \{i,i+1\}$,
and if $\max(\sigma(i)) < \min(\sigma(i+1)) $ then define 
\[\sigma'(i) = \sigma(i)
\quand \sigma'(i+1) = \sigma(i+1).\]
Suppose $\max(\sigma(i)) = \min(\sigma(i+1)) = a'$ for $a \in \PP$.
If $a \in \sigma(i+1)$ then set 
\[\sigma'(i) = \sigma(i) \sqcup \{a\}\quand \sigma'(i+1) = \sigma(i+1) \setminus\{a'\}\]
 and if $a \notin \sigma(i+1)$ then set 
\[\sigma'(i) = \sigma(i)\setminus\{a'\} \sqcup \{a\} \quand \sigma'(i+1) = \sigma(i+1)\setminus\{a'\} \sqcup \{a\}.\]
The resulting map $\sigma'$ is an element of $\tE(w',\gamma')$
with $|\sigma| = |\sigma'|$ and $x^\sigma = x^{\sigma'}$
and it is easy to see that $\sigma \mapsto \sigma'$ 
is a bijection $\tE(w,\gamma) \xrightarrow{\sim} \tE(w',\gamma')$, as needed.
\end{proof}

Setting $\beta=0$ in \eqref{kbeta-eq} recovers Stembridge's definition \cite[\S2.2]{Stembridge1997a} of
the \emph{peak quasisymmetric function}
$
K_\alpha := K^{(0)}_\alpha =\sum_{\alpha'} 2^{\ell(\alpha')} M_{\alpha'} \in \QSym$,
where the sum is over all compositions $\alpha' \vDash |\alpha|$ with
$I(\alpha) \subseteq I(\alpha') \cup (I(\alpha')+1)$.
If $\beta$ has degree zero and $x_i$ has degree one,
then $K_\alpha$ is the nonzero homogeneous component of $K^{(\beta)}_\alpha$
of lowest degree.

The peak quasisymmetric functions are
a basis for a subalgebra of $\QSym$ \cite[Thm.~3.1]{Stembridge1997a}.
The power series $\{ K^{(\beta)}_\alpha\}$ are therefore a pseudobasis for a linearly compact $\ZZ[\beta]$-submodule
of $\mQSym$.

There is a slight variant of the preceding constructions which is also of interest.
Let $w=(w_1,w_2,\dots,w_N)$ be an arbitrary finite sequence along with 
an injective map $\gamma : \{w_1,w_2,\dots,w_N\} \to \ZZ$.
Extending our earlier notation, let
\[  \oOmega(w,\gamma) := \oOmega([N],\delta)\]
where  
$\delta : [N] \to [N]$ is again the map 
satisfying \eqref{delta-eq} for all $1\leq i < j \leq N$.
For a
 peak composition $\alpha \vDash N$, define
\be\label{okbeta-eq} 
\oK^{(\beta)}_\alpha   := \sum_S
\beta^{|S|-N} x^S
 \in \mQSym,
\ee
where the sum is over $N$-tuples $S=(S_1 , S_2 , \dots, S_N)$ of sets $S_i \in \MSet$
with
\begin{itemize}
\item $S_1 \preceq S_2 \preceq \dots\preceq S_N$,
\item $S_i \prec S_{i+1}$ if $i \in I(\alpha)$, and  $S_{i+1} \subseteq \PP$ if $i \in \{0\} \cup I(\alpha)$, and
\item $S_i \cap S_{i+1} \subseteq \PP$  if $i \notin I(\alpha)$.
\end{itemize}
These are the same as the tuples indexing the sum in \eqref{kbeta-eq} except 
we require the set $S_{i+1}$ to contain only unprimed numbers if $i \in \{0\} \cup I(\alpha)$.

\begin{example}
If $\alpha=(2,1)$ then 
\eqref{okbeta-eq} is the sum over all triples $S_1 \preceq S_2 \prec S_3$
with $S_1, S_3\in \PSet$ and $S_2 \in \MSet$.
In this case \eqref{okbeta-eq} is equivalent
 to
\[ \oK^{(\beta)}_{(2,1)}
=\sum_{n} \sum_{i_1  <i_2<\dots <i_n} \sum_{1\leq j < k \leq n} \beta^{n-3} \cdot \Pi(i;j,k)\cdot x_{i_1}x_{i_2}\cdots x_{i_n} \]
where we define
\[
\Pi(i;j,k) := \begin{cases}
 (1+\beta x_{i_{j}}) \cdot  \prod_{j<t<k} (2 + \beta x_{i_{t}} ) \cdot  (1+\beta x_{i_k}) &\text{if $j+1<k$}\\
 \beta\cdot  (x_{i_j} + x_{i_{j+1}} + \beta x_{i_j} x_{i_{j+1}}) & \text{if }j+1=k.
 \end{cases}
\]
\end{example}

We have a second analogue of 
Proposition~\ref{by-def-prop}:

\begin{proposition}\label{k2-prop}
Fix a choice of $(w,\gamma)$ as above, and
let $\alpha \vDash \ell(w)$ be the unique composition such that $I(\alpha)= \PeakSet(w,\gamma)$.
Then $\oK^{(\beta)}_\alpha = \oOmega(w,\gamma)$.
\end{proposition}

\begin{proof}
Our argument is similar to the proof of Proposition~\ref{k1-prop}
but relies on a different bijection.
Let $w=(w_1,w_2,\dots,w_N)$ be a finite sequence
of length $N$, let $\gamma : \{w_1,w_2,\dots,w_N\} \to \ZZ$ be an injective map,
and let
$\alpha \vDash N$ be the composition with $I(\alpha) = \PeakSet(w,\gamma)$.
We have $\oOmega(w,\gamma) = \oK^{(\beta)}_\alpha$ by definition if 
$\PeakSet(w,\gamma) = \Des(w,\gamma)$.    
Let
$\oE(w, \gamma)
= \cE([N],\delta,\ValSet([N],\delta))$ where $\delta : [N]\to[N]$ satisfies \eqref{delta-eq}, so that
$\oOmega(w, \gamma) = \sum_{\sigma \in \oE(w,\gamma)} \beta^{|\sigma|-N} x^\sigma$.

Suppose $1<i<N$ is such that $i-1,i \in \Des(w,\gamma)$ and $i+1 \notin\Des(w,\gamma)$.
Define $w'$ and $\gamma'$ as in the proof of Proposition~\ref{k1-prop}
so that $\Des(w',\gamma') = \Des(w,\gamma) \setminus \{i\}$
and $\PeakSet(w',\gamma') = \PeakSet(w,\gamma) \neq \Des(w,\gamma)$.
It suffices to show that $\oOmega(w,\gamma) = \oOmega(w',\gamma')$, 
and we do so 
by constructing
a length- and weight-preserving bijection
$\oE(w,\gamma) \xrightarrow{\sim} \oE(w',\gamma')$. 

We have
 $i+1 \in \ValSet(w,\gamma) := 
\{ j \in [N] : \gamma(w_{j-1}) > \gamma(w_j) \leq \gamma(w_{j+1})\}$,
where we define $\gamma(w_0) = \gamma(w_{N+1}) := \infty$.
More specifically, it holds that
\[ \ValSet(w',\gamma') = \ValSet(w,\gamma)\setminus\{i+1\} \sqcup \{i\}.\]
Let  $\sigma \in \tE(w,\gamma)$
and observe that necessarily
\[
\sigma(i) \cap \sigma(i+1) \subset \{1',2',\dots\}
\quand
\sigma(i+1) \subset \PP.\] We must therefore have $\sigma(i) \prec \sigma(i+1)$.
We define a map $\sigma' : [N] \to \MSet$ as follows.
Set $\sigma'(j) = \sigma(j)$ for $j \notin \{i,i+1\}$.
If $\sigma(i) \subseteq \PP$ then define 
\[\sigma'(i) = \sigma(i)
\quand \sigma'(i+1) = \sigma(i+1).\]
Otherwise, there exists a smallest $a\in\PP$ with  $a' \in \sigma(i) \cap \{1',2',\dots\}$.
In this case, if $b := \min(\sigma(i+1)) \in \PP$ and $b' \in \sigma(i)$ then we set
\[\ba
\sigma'(i) &= \{ x \in \sigma(i) : x < a'\} \sqcup \{a\},\\
\sigma'(i+1) &=\{ x \in \sigma(i) : a' < x\}  \sqcup  \sigma(i+1),
\ea\]
while if $b' \notin \sigma(i)$ then we set
\[\ba
\sigma'(i) &= \{ x \in \sigma(i) : x < a'\} \sqcup \{a\},\\
\sigma'(i+1) &=\{ x \in \sigma(i) : a' < x\}  \sqcup  (\sigma(i+1)\setminus\{b\}) \sqcup \{b'\}.
\ea\]
The resulting map $\sigma'$ is an element of $\oE(w',\gamma')$
with $|\sigma| = |\sigma'|$ and $x^\sigma = x^{\sigma'}$
and one can check that $\sigma \mapsto \sigma'$ 
is a bijection $\oE(w,\gamma) \xrightarrow{\sim} \oE(w',\gamma')$, as needed.
\end{proof}
 
Setting $\beta=0$ in \eqref{okbeta-eq} recovers 
the power series denoted
$\oK_\alpha$
in \cite{Stembridge1997a},
which is the nonzero homogeneous component of $\oK^{(\beta)}_\alpha$
of lowest degree.
Since $\oK_\alpha = 2^{-\ell(\alpha)} K_\alpha$,
the power series $\{ \oK^{(\beta)}_\alpha\}$ are a pseudobasis for a (different) linearly compact $\ZZ[\beta]$-submodule
of $\mQSym$. 

Unlike $\oK_\alpha = \oK^{(0)}_\alpha$ and $K_\alpha=K^{(0)}_\alpha$,
the power series  $\oK^{(\beta)}_\alpha$ and $K^{(\beta)}_\alpha$
are not scalar multiples of each other. We investigate their precise relationship next.

\subsection{Poset operators}

Fix a labeled poset $(P,\gamma)$ and let $V\subseteq \ValSet(P,\gamma)$.
It is not hard to show that we always have
$\Omega^{(0)}(P,\gamma,V) = 2^{-|V|} \Omega^{(0)}(P,\gamma)$. 
The relationship between $\tOmega(P,\gamma,V)$ and $\tOmega(P,\gamma)$ is more complicated, 
involving the \emph{vertex doubling operators} $\fkD_v$ that we now define.

Given  $v \in P$,
define $\fkD_v(P,\gamma) := (Q,\delta)$ to be the labeled poset with the following properties.
\begin{itemize}
\item As a set, $Q = P\sqcup\{ v'\}$ is the disjoint of union of $P$ and a new element $v'$.

\item The order on $Q$ is the one extending the order on  $P$ such that
$v\lessdot v'$ and for every $x \in P \setminus \{v\}$, $v'$ is related to $x$ in the same way that $v$ is.

\item The labeling map $\delta$ satisfies  $\delta(v) =\gamma(v)$, $\delta(v') = \gamma(v)+1$, $\delta(x) = \gamma(x)+1$ for $x \in P$ with $\gamma(v) < \gamma(x)$, and $\delta(x) = \gamma(x)$ for all other $x \in P$.
\end{itemize}

\begin{example}\label{poset-ex1}
If we represent labeled posets 
as oriented Hasse diagrams as in 
Example~\ref{ori-ex}, 
then the doubling operator $\fkD_v$ acts as follows:
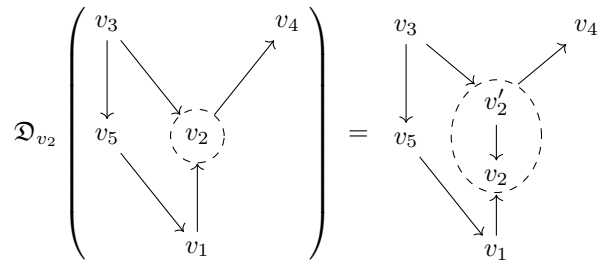
\[
\fkD_{v_2}\(\begin{tikzpicture}[baseline=(b.base), xscale=0.6, yscale=0.5]
\tikzset{edge/.style = {<-}}
\node (c1) at (-2,3) {$v_3$};
\node (c2) at (2,3) {$v_4$};
  \node (b) at  (0,0)  [circle, dashed, draw] {$v_2$};
  \node (a) at (0,-3) {$v_1$};
  \node (d) at (-2, 0) {$v_5$};
  \draw[edge] (b) -- (a);
  \draw[edge] (b) -- (c1);
  \draw[edge] (c2) to (b);
  \draw[edge] (a) to (d);
  \draw[edge] (d) to (c1);
\end{tikzpicture}\)\ =\ 
\begin{tikzpicture}[baseline=(b.base), xscale=0.6, yscale=0.5]
\tikzset{edge/.style = {<-}}
\node (c1) at (-2,3) {$v_3$};
\node (c2) at (2,3) {$v_4$};
\node (b) at (0,0) {};
  \node (b2) at (0,1) {$v_2'$};
  \node (b1) at (0,-1) {$v_2$};
  \node (a) at (0,-3) {$v_1$};
  \node (d) at (-2, 0) {$v_5$};
  \draw[edge] (a) to (d);
  \draw[edge] (d) to (c1);
  \draw[edge] (b1) -- (a);
  \draw[edge] (b2) -- (c1);
  \draw[edge] (b1) -- (b2);
  \draw[edge] (c2) -- (b2);
  \draw[dashed] (0, .1) ellipse [x radius = 1, y radius = 1.5];
\end{tikzpicture}
\]
\end{example}

If $P$ is a chain (i.e., linearly ordered) then $\fkD_v(P,\gamma)$ is also a chain.

\begin{lemma}\label{fkd-lem}
Let $(P,\gamma)$ be a labeled poset.
Suppose $v \in V \subseteq \ValSet(P,\gamma)$ and $(Q,\delta) = \fkD_v(P,\gamma)$. 
There is then a length- and weight-preserving bijection
\[  \tE(P,\gamma,V\setminus\{v\}) \xrightarrow{\sim} \tE(P,\gamma,V) \sqcup \tE(P,\gamma,V) \sqcup \tE(Q,\delta,V).
\]
Consequently, $\tOmega(P,\gamma,V\setminus\{v\}) = 2 \cdot \tOmega(P,\gamma,V) + \beta \cdot \tOmega(Q,\delta,V)$.
\end{lemma}

\begin{proof}
Let $\cE'(P,\gamma,V)$ be the set of enriched set-valued $(P,\gamma)$-partitions
$\tau \in \cE(P,\gamma,V\setminus\{v\})$
such that $\max(\tau(v))$ is the unique primed entry of $\tau(v)$.
Removing the prime from $\max(\tau(v))$
defines a length- and weight-preserving bijection $\cE'(P,\gamma,V) \to \cE(P,\gamma,V)$.
We describe a bijection
$\tE(P,\gamma,V\setminus\{v\}) \to \tE(P,\gamma,V) \sqcup \tE'(P,\gamma,V)
 \sqcup \tE(Q,\delta,V)$.

Fix $\sigma \in \tE(P,\gamma,V\setminus\{v\})$ and define $\sigma' $ as follows.
If $\sigma \in \tE(P,\gamma,V) \sqcup \tE'(P,\gamma,V)$,
then we set $\sigma' = \sigma$.
Otherwise, $\sigma(v)$ contains a least primed integer $a' \neq \max(\sigma(v))$, and we define
$\sigma' \in \tE(Q,\delta,V)$ to be the map with
\[
\sigma'(v) = \{ x \in \sigma(v) : x < a' \} \sqcup \{a\}
\quand
\sigma'(v') = \{ x \in \sigma(v) : a' < x\}
\]
and $\sigma'(s) = \sigma(s)$ for all $s \in P \setminus \{v\} = Q\setminus\{v,v'\}$.
It is easy to see that $\sigma \mapsto \sigma'$ 
is a length- and weight-preserving bijection of the desired type. 
\end{proof}

%

Fix a labeled poset $(P, \gamma)$, a subset $V \subseteq \ValSet(P, \gamma)$,  and a vertex $v \in  P$.
Let $(Q,\delta) := \fkD_v (P,\gamma)$.
 We use the notational shorthand $\fkD_v \cdot \tOmega(P,\gamma, V)$ to mean
\[\fkD_v \cdot \tOmega(P,\gamma, V) := \tOmega(Q,\delta, V),\] so that 
$
\fkD_v \cdot \oOmega(P, \gamma) = \oOmega(\fkD_v(P, \gamma))
$
and
$
\fkD_v \cdot \tOmega(P, \gamma) = \tOmega(\fkD_v(P, \gamma)).
$
In this shorthand, one has by simple telescoping that 
\[  (2 + \beta \fkD_v) \cdot \frac{1}{2} \sum_{n=0}^\infty (-\beta/2)^n \fkD_v^n  \cdot \tOmega(P,\gamma, V) = \tOmega(P,\gamma, V)\]
in $\mQSym_{\QQ[\beta]}$, and we define $(2 + \beta \fkD_v)^{-1}  :=\frac{1}{2} \sum_{n=0}^\infty (-\beta/2)^n \fkD_v^n$.


\begin{theorem}\label{op-thm}
Let $(P,\gamma)$ be a labeled poset, $V$ a subset of $\ValSet(P, \gamma)$, and $U = \ValSet(P, \gamma) \setminus V$.  Then
\[
\ba
\tOmega(P,\gamma,V) &= \prod_{u \in U} (2 + \beta \fkD_u)\cdot  \oOmega(P,\gamma)
\\&= \prod_{v \in V} (2 + \beta \fkD_v)^{-1}\cdot  \tOmega(P,\gamma).
\ea\]
In
particular, we have
$\tOmega(P,\gamma) = \prod_{v \in \ValSet(P,\gamma)} (2 + \beta \fkD_v)\cdot  \oOmega(P,\gamma).$
\end{theorem}

\begin{proof}
Lemma~\ref{fkd-lem} implies
that 
$(2 + \beta \fkD_u) \cdot \tOmega(P,\gamma, V\sqcup\{u\}) =  \tOmega(P,\gamma,V)$
if $u \in U$.  The result follows by repeating this observation.
\end{proof}

As a corollary, we deduce 
that each $K^{(\beta)}_\alpha$ is a finite $\ZZ[\beta]$-linear combination of $\oK^{(\beta)}_{\alpha}$'s
while each $\oK^{(\beta)}_{\alpha}$ is an infinite $\QQ[\beta]$-linear combination of $K^{(\beta)}_{\alpha}$'s.

\begin{corollary}\label{equiexp-cor}
If $\alpha = (\alpha_1,\alpha_2,\dots,\alpha_n)$ is a peak composition then 
\[ K^{(\beta)}_\alpha  = \sum_{\delta \in \{0,1\}^n} 2^{n-|\delta|} \beta^{|\delta|} \oK^{(\beta)}_{\alpha+\delta}
\quand \oK^{(\beta)}_\alpha  = \tfrac{1}{2^n}\sum_{\delta \in (\NN)^n} (-\beta/2)^{|\delta|} K^{(\beta)}_{\alpha+\delta}.\]
\end{corollary}

\begin{proof}
Given Propositions~\ref{k1-prop} and \ref{k2-prop},
it is straightforward to deduce these identities from
Theorem~\ref{op-thm}.
\end{proof}

\subsection{Quasisymmetric subalgebras}

As noted in Section~\ref{multipeak-sect}, both
$\{K^{(\beta)}_\alpha\}$ and
$\{ \oK^{(\beta)}_\alpha\}$ (with $\alpha$ ranging over all peak compositions) are 
linearly independent families 
in $\mQSym$. 
Moreover, all $x$-monomials appearing 
in $K^{(\beta)}_\alpha$ and $\oK^{(\beta)}_\alpha$ have degree at least $|\alpha|$.
We may therefore make the following definitions.

\begin{definition}
Let $\mcoPeak $ and $ \omcoPeak$
denote the linearly compact $\ZZ[\beta]$-modules 
with the multipeak quasisymmetric functions $\{ K^{(\beta)}_\alpha\}$
and $\{ \oK^{(\beta)}_\alpha\}$
($\alpha$ ranging over all peak compositions) as 
respective pseudobases.
Define
$\mcoPeak_{\QQ[\beta]}$ to
be the linearly compact $\QQ[\beta]$-module with $\{ K^{(\beta)}_\alpha\}$ as a pseudobasis.\footnote{For simplicity, we define $\mcoPeak_{\QQ[\beta]}$  over the polynomial ring $\QQ[\beta]$ with rational coefficients, but in fact, it would be sufficient to work with coefficients in $\ZZ[2^{-1}]$ rather than $\QQ$. As in Corollary~\ref{equiexp-cor}, we will only ever need to divide  by powers of the prime $2$.}
\end{definition}

\begin{theorem}\label{inter-thm}
Both $\mcoPeak$ and $\omcoPeak$ are LC-Hopf subalgebras of $\mQSym$,
and
$\omcoPeak = \mQSym \cap \mcoPeak_{\QQ[\beta]} \supseteq \mcoPeak$.
\end{theorem}

On setting $\beta=0$, this reduces to \cite[Cor.~3.4]{Stembridge1997a},
whose proof is similar.

\begin{proof}
The modules $\mcoPeak$ and $\omcoPeak$ are LC-Hopf subalgebras of $\mQSym$
since (in view of 
Theorem~\ref{ep-thm1} and the results in the previous section) they
are the images of $\mLPSet $ and $\mLPSet^+$ under the LC-Hopf algebra morphism
described in Theorem~\ref{><-thm}.

Corollary~\ref{equiexp-cor} implies that $\omcoPeak \subseteq \mQSym \cap \mcoPeak_{\QQ[\beta]}$.
Write $\prec$ for the linear order on compositions with $\alpha \prec \alpha'$ if
$|\alpha| < |\alpha'|$ or 
if $|\alpha| =|\alpha'|$ and $\alpha$ exceeds $\alpha'$ in lexicographic order.
By \eqref{okbeta-eq}, if $\alpha$ is a peak composition then
$\oK^{(\beta)}_\alpha \in M_\alpha + \sum_{\alpha \prec \alpha'} \ZZ[\beta] M_{\alpha'}$,
so any $\QQ[\beta]$-linear combination of $\oK^{(\beta)}_\alpha$'s 
in  $\mQSym$
must have coefficients in $\ZZ[\beta]$.
Thus, in view of Corollary~\ref{equiexp-cor},
the reverse inclusion $\omcoPeak \supseteq \mQSym \cap \mcoPeak_{\QQ[\beta]}$ also holds.
\end{proof}

\begin{corollary}\label{omco-cor}
If $V \subseteq \ValSet(P,\gamma)$
then $\tOmega(P,\gamma,V) \in \omcoPeak$.
\end{corollary}

 \begin{proof}
 This follows from Theorem~\ref{inter-thm} since
$\tOmega(P,\gamma,V)$ belongs to $\mQSym$ by definition and to $\mcoPeak_{\QQ[\beta]}$
 by Theorem~\ref{op-thm}.
 \end{proof}
 
As in Example~\ref{kbeta-ex}, let
$x \oplus y := x + y + \beta xy$.
Also define $\ominus x := \frac{-x}{1+\beta x}$, so that $x\oplus (\ominus x) = 0$.
Elements of $\mcoPeak_{\QQ[\beta]}$ satisfy the following cancellation law,
which would be equivalent to Ikeda and Naruse's \emph{$K$-theoretic $Q$-cancellation property} \cite[Def. 1.1]{IkedaNaruse} if we also required symmetry in the $x_i$ variables.
Setting $\beta=0$ in this statement recovers \cite[Lem.~3.7]{Stembridge1997a}.

\begin{lemma}\label{q-cancel-lem}
If $f \in \mcoPeak_{\QQ[\beta]}$ then 
$f(t,\ominus t,x_3,x_4,\dots) = f(x_3,x_4,\dots)$ where $t$ is an indeterminate
commuting with each $x_i$.
\end{lemma}

\begin{proof}
If $f \in \ZZ[\beta]\llbracket x_1,x_2,\dots \rrbracket$ then let $f(x_1,x_2,\dots,x_n)\in \ZZ[\beta][x_1,x_2,\dots,x_n]$
denote the polynomial obtained by setting $x_i=0$ for all $i>n$.
Suppose $(P,\gamma)$ is a labeled poset. 
It follows from the definition of $\tOmega(P,\gamma)$ that
\[ 
\tOmega(P,\gamma) = \sum_{S\cup T = P}
\beta^{|S\cap T|}\cdot 
\tOmega(S,\gamma)(x_1,x_2)\cdot  \tOmega(T,\gamma)(x_3,x_4,\dots) 
\]
where the sum is over all ordered pairs $(S,T)$ of subsets of $P$ such that 
$S$ is a lower set, $T$ is an upper set,
$P = S\cup T$, and
$S\cap T$ is an antichain.
To prove the lemma, it is therefore enough to check that 
$\tOmega(P,\gamma)(t,\ominus t)=0$ whenever $P$ is nonempty.
Let $\alpha$ be a nonempty peak composition.
By Theorem~\ref{ep-thm1} and Proposition~\ref{k1-prop}, it suffices to show that $K^{(\beta)}_\alpha(t,\ominus t) =0$.

It is clear from \eqref{kbeta-eq} that $K^{(\beta)}_\alpha(x_1,x_2) = 0$ if $\ell(\alpha)>2$.
If $\alpha$ has a single part, then in the notation of \cite{IkedaNaruse} one has $K^{(\beta)}_{\alpha} = \bGQ_\alpha$ 
(see \cite[Thm.~9.1]{IkedaNaruse} and Section~\ref{shifted-stable-sect} below) and the desired identity
is \cite[Prop.~3.1]{IkedaNaruse}.
Assume $\alpha = (j,N-j)\vDash N$ has two parts with $j\geq 2$. 
From \eqref{kbeta-eq},
it is easy to see that $K^{(\beta)}_{\alpha}(x_1,x_2)= \sum_\sigma \beta^{|\sigma|-N} x^\sigma$ 
where the sum is over all maps 
$\sigma:[N]\to \{\text{nonempty subsets of }\{1',1,2',2\}\}$ such that
\begin{itemize}
\item $\sigma(1)$ is $\{1'\}$ or $\{1\}$ or $\{1', 1\}$, and $\sigma(i) =\{1\}$ for $1<i<j$,
\item $\sigma(j)$ is $\{1\}$ or $\{2'\}$ or $\{1,2'\}$, and
\item $\sigma(j+1)$ is $\{2'\}$ or $\{2\}$ or $\{2',2\}$, and $\sigma(i) = \{2\}$ for $j+1<i \leq N$. 
\end{itemize}
Thus $K^{(\beta)}_{\alpha}(x_1,x_2) = (2x_1 + \beta x_1^2) x_1^{j-2}  (x_1 + x_2 + \beta x_1 x_2)  (2x_2 + \beta x_2^2)  x_2^{N-j-1}$,
which we can rewrite as 
$K^{(\beta)}_{\alpha}(x_1,x_2) = (x_1\oplus x_1)(x_1\oplus x_2)(x_2\oplus x_2) x_1^{j-2} x_2^{N-j-1}$.
Since $t \oplus (\ominus t) = 0$, we have $K^{(\beta)}_\alpha(t,\ominus t) = 0$ as needed.
\end{proof}

The LC-Hopf subalgebra $\mcoPeak$ is also a quotient of $\mQSym$. Define
\[\ttheta : \mQSym \to \mcoPeak\]
to be the continuous $\ZZ[\beta]$-linear map with 
$L^{(\beta)}_\alpha \mapsto K^{(\beta)}_{\Lambda(\alpha)}$
for each composition $\alpha$,
where $\Lambda(\alpha)$ is the peak composition of $|\alpha|$ characterized by
\be
\label{Lambda-eq}
I(\Lambda(\alpha)) = \{  i \in I(\alpha) : i > 1, i-1 \notin I(\alpha)\}.\ee
If $\Des(w,\gamma) = I(\alpha)$
then $\PeakSet(w,\gamma)  = I(\Lambda(\alpha))$,
and if $\alpha$ is already a peak composition then $\Lambda(\alpha) = \alpha$.

\begin{corollary}\label{ttheta-cor}
The map $\ttheta$ sends $\tGamma(P,\gamma)\mapsto \tOmega(P,\gamma)$
for all labeled posets $(P,\gamma)$
and is a surjective morphism of LC-Hopf algebras $ \mQSym \to \mcoPeak$.
\end{corollary}

\begin{proof}
Theorems~\ref{p-thm1} and \ref{ep-thm1} show that 
$\ttheta(\tGamma(P,\gamma)) = \tOmega(P,\gamma)$ for all labeled posets $(P,\gamma)$.
Comparing  Corollaries~\ref{sv-products-cor} and \ref{esv-products-cor}
shows that this is a bialgebra morphism, and hence an LC-Hopf algebra morphism.
Since the $K^{(\beta)}_{\alpha}$ form a pseudobasis for $\mcoPeak$, it is also surjective.
\end{proof}

\begin{remark*}
As discussed in Remark~\ref{product-rmk},
results of Lam and Pylyavskyy \cite[\S5.4]{LamPyl}
lead to product and coproduct formulas for the multifundamental quasisymmetric functions
$L^{(\beta)}_\alpha$. Applying $\ttheta$ to such identities gives analogous
(co)product formulas for the multipeak quasisymmetric functions $K^{(\beta)}_\alpha$.
\end{remark*}

\subsection{Shifted stable Grothendieck polynomials}\label{shifted-stable-sect}

A primary motivation for our definition of enriched set-valued $P$-partitions 
comes from the families of \emph{shifted stable Grothendieck
polynomials} $\GQ_\lambda$ and $\GP_\lambda$ introduced by
Ikeda and Naruse in \cite{IkedaNaruse}.
Mirroring the situation in Section~\ref{stable-sect},
we can recover these power series as the enriched set-valued weight enumerators
of labeled posets associated to shifted Young diagrams.
In this section, 
we will also describe skew analogues of $\GQ_\lambda$ and $\GP_\lambda$, which have not been
considered previously.

Suppose $\lambda =(\lambda_1> \lambda_2 > \dots > 0)$
and $\mu = (\mu_1> \mu_2> \dots > 0)$ 
are strict partitions with $\mu \subseteq \lambda$.
The \emph{shifted skew diagram} of $\lambda/\mu$ is
\be\label{sd-eq}\SD_{\lambda/\mu} := \{ (i,i+j-1) \in \PP\times \PP : \mu_i < j \leq \lambda_i\}.\ee
Let $\SD_\lambda = \SD_{\lambda/\emptyset}$.
As usual, we consider $\SD_{\lambda/\mu}$ to be partially ordered with $(i,j) \leq (i',j')$
if $i\leq i'$ and $j\leq j'$.
Let $n = |\lambda| - |\mu|$ and fix a bijection $\theta : \SD_{\lambda/\mu} \to [n]$
satisfying the conditions in \eqref{canonical-eq}.
For example, if $\lambda =(5,4,2)$ and $\mu=(2,1)$ then 
the oriented Hasse diagram representing $(\SD_{\lambda/\mu},\theta)$ is 
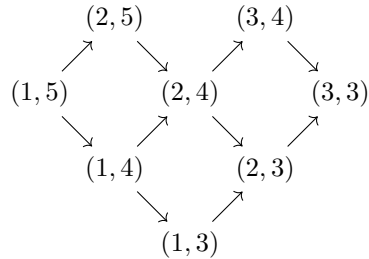
\[ 
\begin{tikzpicture}[baseline=(center.base), xscale=1, yscale=1]
\tikzset{edge/.style = {->}}
\node (center) at (0, -1) {};
  \node (a) at (3,-3) {$(1,3)$};
  \node (b) at (1,-1) {$(1,5)$};
  \node (c) at (2,0) {$(2,5)$};
  \node (d) at (2,-2) {$(1,4)$};
  \node (e) at (3,-1) {$(2,4)$};
  \node (g) at (4,-2) {$(2,3)$};
  \node (h) at (4,0) {$(3,4)$};
  \node (i) at (5,-1) {$(3,3)$};
  \draw[edge] (d) -- (a);
  \draw[edge] (a) -- (g);
  \draw[edge] (b) -- (d);
  \draw[edge] (c) -- (e);
  \draw[edge] (e) -- (g);
  \draw[edge] (g) -- (i);
  \draw[edge] (h) -- (i);
  \draw[edge] (b) -- (c);
  \draw[edge] (d) -- (e);
  \draw[edge] (e) -- (h);
\end{tikzpicture}
\]
It may be helpful to contrast this picture with \eqref{dlam-eq}.

 The elements of $\tE(\SD_{\lambda/\mu},\theta)$
may be identified with \emph{semistandard shifted set-valued (marked) tableaux} of shape $\lambda/\mu$
 as defined in \cite[\S9.1]{IkedaNaruse},
 i.e., fillings of $\SD_{\lambda/\mu}$ by finite nonempty subsets of $\MM$ that are weakly increasing along rows and columns in the sense of 
the relation $\preceq$, such that no unprimed number appears twice in the same column and no primed number appears twice in the same row.
We let \[ \SetSSMT(\lambda/\mu) := \tE(\SD_{\lambda/\mu},\theta)\]
and define the \emph{$K$-theoretic Schur $Q$-function} of $\lambda/\mu$ to be
 \be\label{gq-eq} 
 \bGQ_{\lambda/\mu} := \tOmega({\SD_{\lambda/\mu},\theta}) = 
 \sum_{T \in \SetSSMT(\lambda/\mu)} \beta^{|T|-|\lambda/\mu|}x^T.
 \ee
Both definitions are independent of $\theta$.

The special case $ \bGQ_{\lambda} :=  \bGQ_{\lambda/\emptyset}$
coincides with Ikeda and Naruse's definition of a $K$-theoretic Schur $Q$-function
\cite[Thm.~9.1]{IkedaNaruse}.
We will show in Section~\ref{sym-sect} that the quasisymmetric function $\bGQ_{\lambda/\mu}$ 
is always symmetric in the $x_i$ variables;
when $\mu=\emptyset$, this follows from \cite[Thm.~9.1]{IkedaNaruse}.
Setting $\beta=0$ reduces \eqref{gq-eq} to the definition of the \emph{skew Schur $Q$-function} $Q_{\lambda/\mu}$
described, for example, in \cite[App.~A.1]{Stembridge1997a}.

Recall that a semistandard set-valued tableau is \emph{standard}
if its entries are disjoint sets, not containing any consecutive integers,
with union $\{1,2,\dots,N\}$ for some $N\geq n$.
 The set $\tL(\SD_{\lambda/\mu})$ of linear multiextensions
 of $\SD_{\lambda/\mu}$ is naturally identified with the set of standard set-valued (unmarked) tableaux of shifted shape
 $\lambda/\mu$;
 a sequence $(w_1,w_2,\dots,w_N) \in \tL(\SD_{\lambda/\mu})$
 corresponds to the standard set-valued tableau containing $i$ in box $w_i$.
Let 
\[
\SetSYT_{\text{shifted}}(\lambda/\mu) := \tL(\SD_{\lambda/\mu})
\]
and define  $\PeakSet(T) := \PeakSet(T,\theta)$ for $T \in \SetSYT_{\text{shifted}}(\lambda/\mu)$; 
then $i \in \PeakSet(T)$ if and only if $i-1$, $i$, and $i+1$ all appear in $T$ with $i$ in a strictly greater column
than $i-1$ and a strictly lesser row than $i+1$.
Theorem~\ref{ep-thm1} implies that
 \be\label{tilde-GQ-eq}
 \bGQ_{\lambda/\mu} = \sum_\alpha  g^\alpha_{\lambda/\mu} \cdot \beta^{|\alpha| - |\lambda/\mu|} \cdot K^{(\beta)}_\alpha,
 \ee
 where the sum is over peak compositions $\alpha$ and 
 $g^\alpha_{\lambda/\mu}$ is the number of  tableaux
 $T \in \SetSYT_{\text{shifted}}(\lambda/\mu)$ with $|T|=|\alpha|$ and $\PeakSet(T) = I(\alpha)$.

There is a second family of shifted stable Grothendieck polynomials discussed in \cite{IkedaNaruse},
which arise in a similar way as enriched set-valued weight enumerators.
Continue to let $\lambda$ and $\mu$ be strict partitions with $\mu \subseteq \lambda$.
Let $n=|\lambda|-|\mu|$ and  fix a bijection $\theta : \SD_{\lambda/\mu} \to [n]$ 
satisfying \eqref{canonical-eq} as above.
Define 
\[ 
V_{\lambda/\mu} := \{ (i,j) \in \SD_{\lambda/\mu} : i=j\} \subseteq \ValSet(\SD_{\lambda/\mu},\theta),
\]
where the containment is equality if $\mu = \emptyset$.
The elements of $\tE(\SD_{\lambda/\mu},\theta,V_{\lambda/\mu})$ 
are precisely the semistandard shifted set-valued tableaux in $\SetSSMT(\lambda/\mu)$ whose entries on the main diagonal contain only unprimed numbers.
We define the \emph{$K$-theoretic Schur $P$-function} of $\lambda/\mu$ to be
\be\label{gp-eq}
\bGP_{\lambda/\mu} := \tOmega(\SD_{\lambda/\mu},\theta,V_{\lambda/\mu}) = 
\sum_{\substack{
T\in \SetSSMT(\lambda/\mu) \\ 
T(i,i) \subset \PP\text{ if }(i,i) \in \SD_{\lambda/\mu}
}} \beta^{|T|-|\lambda/\mu|} x^T.
\ee
This formula is again independent of the choice of $\theta$.

The case $ \bGP_{\lambda} :=  \bGP_{\lambda/\emptyset} = \oOmega(\SD_{\lambda/\mu},\theta)$
is Ikeda and Naruse's definition of a $K$-theoretic Schur $P$-function
\cite[Thm.~9.1]{IkedaNaruse}.
As with $\bGQ_\lambda$, 
Ikeda and Naruse prove that $\bGP_\lambda$ is symmetric in the $x_i$ variables;
we extend this result to skew shapes below.
Setting $\beta=0$ in \eqref{gp-eq} recovers the \emph{skew Schur $P$-function} 
$P_{\lambda/\mu} = 2^{\ell(\mu) - \ell(\lambda)} Q_{\lambda/\mu}$.

The essential reference for the properties of $\bGP_\lambda$ and $\bGQ_\lambda$ is
 \cite{IkedaNaruse}.
For more background and various extensions, see \cite{NN2017,NN2018,Naruse2018}.

We obtain a third interesting family of shifted stable Grothendieck polynomials
as a special case of the preceding constructions.
Let $\mu$ and $\lambda$ be arbitrary partitions (i.e., not necessarily strict)
with $\mu\subseteq \lambda$.
Suppose $\lambda$ has $k$ parts and let $\delta_k := (k,k-1,\dots,2,1)$.
The partitions 
\[ \ba
\lambda+\delta_k &:= (\lambda_1 +k, \lambda_2 + k-1, \dots, \lambda_k + 1) \\
\mu+\delta_k &:= (\mu_1 +k, \mu_2 + k-1, \dots, \mu_k + 1)
\ea
\]
are then both strict, so we can set 
\[\bGS_{\lambda/\mu} 
:= \bGP_{(\lambda+\delta_k)/(\mu+\delta_k)}
=\bGQ_{(\lambda+\delta_k)/(\mu+\delta_k)}
\quand
\bGS_\lambda := \bGS_{\lambda/\emptyset}.\]
This definition appears to be new.
We refer to $\bGS_{\lambda/\mu}$ as the 
 \emph{$K$-theoretic Schur $S$-function} of $\lambda/\mu$.
The name makes sense as setting $\beta=0$ recovers the \emph{Schur $S$-function} $S_{\lambda/\mu}$
discussed in \cite[\S A.4]{Stembridge1997a} and \cite[\S III.8]{Macdonald},
which is also the homogeneous
component of $\bGS_{\lambda/\mu}$ of lowest degree. 

Suppose $n = |\lambda|-|\mu|$.
Since 
$\SD_{(\lambda+\delta_k)/(\mu+\delta_k)} \cong \D_{\lambda/\mu}$ 
as posets, we have
by Corollary~\ref{ttheta-cor} that
\[
\bGS_{\lambda/\mu} = \tOmega(\SD_{(\lambda+\delta_k)/(\mu+\delta_k)}, \theta) = \tOmega(\D_{\lambda/\mu}, \theta) = \ttheta(G^{(\beta)}_{\lambda/\mu}).
\]
This shows that $\ttheta$ is a $K$-theoretic analogue of the \emph{superfication map} discussed in 
\cite[\S3.2]{TKLamThesis}.

We expect that the functions $\bGS_\lambda$
are related to the ``$K$-theoretic Stanley symmetric functions'' of classical types B, C, and D introduced
in \cite{KirillovNaruse}, as well as to
$K$-theoretic generalizations of the main result in \cite{MP2019}.
There is one other connection to the recent literature that we can explain more precisely.
DeWitt has shown that
$S_{\mu} = Q_{\nu}$
if $\mu$ and $\nu$ are the partitions
\be\label{mu-nu-eq}
\mu = (m^k) \qquand \nu = (m+k-1,m+k-3,\dots,|m-k|+1)
\ee
for some $m,k\in \PP$;
moreover,
this is the only  possible identity of the form $S_{\lambda/\mu} = c Q_\nu$ with $c\in \ZZ$ 
apart from the equality  $S_{(2,1)/(1)} = Q_{(1)}Q_{(1)} = 2Q_{(2)}$
 \cite[Thm.~IV.3]{Dewitt}.
DeWitt's  result appears to generalize.

\begin{conjecture}
If $\mu$ and $\nu$ are as in \eqref{mu-nu-eq} then 
$\bGS_{\mu} = \bGQ_{\nu}$.
\end{conjecture}

Since $\bGS_{(2,1)/(1)} = \bGQ_{(1)} \bGQ_{(1)} \neq 2\bGQ_{(2)}$
and since the $K$-theoretic Schur $Q$- and $S$-functions are homogeneous
if  $\beta$ has degree $-1$,
this formula would describe all possible identities of the form
$\bGS_{\lambda/\mu} = c \bGQ_\nu$ with $c\in \ZZ[\beta]$.

\section{Symmetric functions}\label{sym-sect}

 The power series $\bGP_{\lambda/\mu}$ and $\bGQ_{\lambda/\mu}$ defined 
 by \eqref{tilde-GQ-eq} and \eqref{gp-eq} are quasisymmetric by construction.  When $\mu = \emptyset$, it follows from \cite[Thm.~9.1]{IkedaNaruse} that these power series are actually symmetric.  In this section, 
 we
 prove that  $\bGP_{\lambda/\mu}$ and $\bGQ_{\lambda/\mu}$ are symmetric for any  $\mu$.
 As motivation, we start by discussing the 
connection between these power series and the $K$-theory of the Grassmannian.

\subsection{$K$-theory of Grassmannians}

Suppose $Z$ is a complex algebraic variety. The
 \emph{$K$-theory ring} $K(Z)$
is the Grothendieck group of locally-free sheaves on $Z$.
The ring multiplication is the operation induced by the tensor product.

Fix integers $0\leq k \leq n$ and let $\Gr(k,\CC^n)$ denote the Grassmannian of
$k$-dimensional subspaces of $\CC^n$.
For each partition $\lambda$ whose diagram fits in the rectangle $[k]\times[n-k]$,
there is an associated \emph{Schubert variety} $X_\lambda \subseteq \Gr(k,\CC^n)$; see \cite[\S3.2]{Manivel}. If $\cO_{X_\lambda}$ denotes the structure sheaf of $X_\lambda$, then
the ring $K(\Gr(k,\CC^n))$ is spanned by the corresponding classes $[\cO_{X_\lambda}]$.

Let $\Gamma$ be the additive group generated by the stable Grothendieck polynomials
$G_\lambda := G^{(-1)}_\lambda$ as $\lambda$ ranges over all integer partitions,
and let
$I_{k,n-k}$ denote the subgroup of $\Gamma$ spanned by the functions $G_\lambda$
indexed by partitions $\lambda\not\subseteq [k]\times [n-k]$.
Results of Buch \cite{Buch2002} show that $\Gamma$ is a ring in which
the subgroup $I_{k,n-k}$ is an ideal.

\begin{theorem}[{\cite[Thm.~8.1]{Buch2002}}]
If we set $[\cO_{X_\lambda}]=0$ when $\lambda\not\subseteq [k]\times [n-k]$,
then  $G_\lambda \mapsto[\cO_{X_\lambda}]$ 
induces a ring isomorphism $\Gamma / I_{k,n-k} \xrightarrow{\sim} K(\Gr(k,\CC^n))$.
\end{theorem}

Thus, the stable Grothendieck polynomials are ``universal'' $K$-theory representatives
for Schubert varieties in type A Grassmannians, and their structure constants 
determine the structure constants of the ring $K(\Gr(k,\CC^n))$.
The shifted stable Grothendieck polynomials have a similar interpretation 
in the context of Lagrangian and maximal orthogonal Grassmannians.

Fix a nondegenerate symmetric bilinear form $\langle\cdot,\cdot\rangle$ on $\CC^n$ and define $\OG(k,n)$ to be the 
\emph{orthogonal Grassmannian} of $k$-dimensional subspaces $V\subseteq \CC^n$ 
that are isotropic, in the sense that $\langle V,V\rangle = 0$.
Let $\cG_n$ be either $\OG(n,2n+1)$ or $\OG(n+1,2n+2)$.
For each strict partition $\lambda$ whose shifted diagram fits in the square $[n]\times[n]$,
or which equivalently has $\lambda \subseteq (n,n-1,\dots,2,1)$,
there is an associated \emph{Schubert variety} $\Omega_\lambda \subseteq \cG_n$ \cite[\S8.1]{IkedaNaruse}.
The classes $[\cO_{\Omega_\lambda}]$ of the corresponding structure sheaves are a basis for $K(\cG_n)$.

Let $\Gamma_P$ be the additive group generated by
$\GP_\lambda := \GP_\lambda^{(-1)} $ as $\lambda$ ranges over all strict partitions.
Let 
$I_{P,n}$ denote the subgroup of $\Gamma_P$ spanned by the functions $\GP_\lambda$
indexed by strict partitions $\lambda\not\subseteq (n,n-1,n\dots,2,1)$.
Results in \cite{CTY,IkedaNaruse} 
(see also \cite{HKPWZZ}) show that $\Gamma_P$ is a ring in which $I_{P,n}$ is an ideal.

\begin{theorem}[{see \cite[\S8.3]{IkedaNaruse}}]
If we set $[\cO_{\Omega_\lambda}]=0$ when 
$ \lambda\not\subseteq(n,n-1,\dots,2,1)$,
then $\GP_\lambda \mapsto[\cO_{\Omega_\lambda}]$ 
induces a ring isomorphism $\Gamma_P / I_{P,n} \xrightarrow{\sim} K(\cG_n)$.
\end{theorem}

A similar result holds for the $K$-theoretic Schur $Q$-functions, with one technical caveat.
Let $\LG(n)$ denote the \emph{Lagrangian Grassmannian}
of $n$-dimensional subspaces in $\CC^{2n}$ that are isotropic with respect to a fixed
nondegenerate skew-symmetric bilinear form.
For each strict partition $\lambda\subseteq (n,n-1,\dots,2,1)$,
there is again an associated \emph{Schubert variety} $\Omega'_\lambda \subseteq \LG(n)$, and
the classes $[\cO_{\Omega'_\lambda}]$ are a basis for $K(\LG(n))$ \cite[\S8.1]{IkedaNaruse}.

Let $\Gamma_Q$ be the additive group generated by
$\GQ_\lambda := \GQ_\lambda^{(-1)} $ as $\lambda$ ranges over all strict partitions,
and let $I_{Q,n}$
be the subgroup spanned by the functions $\GQ_\lambda$
with $\lambda\not\subseteq (n,n-1,\dots,2,1)$.
Let $\hat \Gamma_Q := \prod_{\lambda} \ZZ \GQ_\lambda\supsetneq \bigoplus_{\lambda} \ZZ \GQ_\lambda = \Gamma_Q$ and define $\hat I_{Q,n}$ as the completion of $I_{Q,n}$ relative to its basis of $\GQ_\lambda$'s.
It then follows from \cite[Prop.~3.5]{IkedaNaruse} that
$\hat\Gamma_Q$ is a ring in which $\hat I_{Q,n}$ is an ideal.

\begin{theorem}[{see \cite[\S8.3]{IkedaNaruse}}]
If we set $[\cO_{\Omega'_\lambda}]=0$ when 
$ \lambda\not\subseteq(n,n-1,\dots,2,1)$,
then $\GQ_\lambda \mapsto[\cO_{\Omega'_\lambda}]$ 
induces a ring isomorphism $\hat \Gamma_Q / \hat I_{Q,n} \xrightarrow{\sim} K(\LG(n))$.
\end{theorem}

We have to state this result in terms of the completions $\hat \Gamma_Q$ and $\hat I_{Q,n}$ 
because
it is still an open problem to show that $\Gamma_Q$ is a ring; see \cite[Conj.~3.2]{IkedaNaruse}.
If this holds then  $\GQ_\lambda \mapsto [\cO_{\Omega'_\lambda}]$
would also induce an isomorphism $\Gamma_Q /I_{Q,n} \xrightarrow{\sim} K(\LG(n))$.
To prove Ikeda and Naruse's conjecture, it is enough to show that
$\bGQ_\lambda \bGQ_\mu$ is always a finite linear combination of $\bGQ_\nu$'s.
Results of Buch and Ravikumar \cite{BuchRavikumar} imply that this holds at least when $\lambda$ or $\mu$ has a single part.

\subsection{Fomin--Kirillov operators}

An element $f \in \ZZ[\beta]\llbracket x_1,x_2,\dots \rrbracket$ is \emph{symmetric} if
the coefficients of $x_1^{a_1}x_2^{a_2}\cdots x_{k}^{a_k}$ and $x_{i_1}^{a_1}x_{i_2}^{a_2}\cdots x_{i_k}^{a_k}$
in $f$ are equal for every choice of $a_1,a_2,\dots,a_k \in \PP$ and every choice of $k$ distinct positive integers 
$i_1, i_2, \dots, i_k$. 
Equivalently, $f$ should be invariant under the change of variables swapping 
$x_i$ and $x_{i+1}$ for all $i$.

\begin{definition}
Let $\mSym$ denote the $\ZZ[\beta]$-module
of all symmetric power series in $\ZZ[\beta]\llbracket x_1,x_2,\dots \rrbracket$.
Let $\Sym$ denote the submodule of power series in $\mSym$ of bounded degree.
\end{definition}

The  \emph{monomial symmetric function} of a partition $\lambda$ is 
the power series given by the sum
$m_\lambda := \sum_{\sort(\alpha)=\lambda}M_\alpha$ over all compositions
$\alpha$ that sort to $\lambda$.
It is well-known that $\Sym$ is a graded Hopf subalgebra of $\QSym$,
which is free as a $\ZZ[\beta]$-module
with
a homogeneous basis given by the power series $\{m_\lambda\}$.
We identify $\mSym$ with the completion of $\Sym$ relative to this basis.

The main result of this section is the following theorem, 
which reduces to \cite[Thm.~9.1]{IkedaNaruse} in the case when $\mu=\emptyset$.

\begin{theorem}
\label{skew-sym-thm}
Let $\mu$ and $\lambda$ be strict partitions
with $\mu\subseteq \lambda$.
The power series $\bGP_{\lambda/\mu}$ and $\bGQ_{\lambda/\mu}$
are elements of $\mSym$.
\end{theorem}

We delay the proof of this result until Section~\ref{yb-sect}.
First, we need to introduce two new families of power series closely related to 
$\bGP_{\lambda/\mu}$ and $\bGQ_{\lambda/\mu}$.

Let $\SPart$ be 
the free $\ZZ[\beta]\llbracket x_1, x_2,\ldots\rrbracket$-module with a pseudobasis given by the set of all strict partitions, and write
 $\mSPart$ for the corresponding completion.
 Let
 $ \langle\cdot,\cdot\rangle : \SPart \times \mSPart \to \ZZ[\beta]\llbracket x_1, x_2,\ldots\rrbracket $
 denote the associated form making the natural (pseudo)bases of strict partitions in $\SPart$ and $\mSPart$ dual to each other. In other words, $\langle\cdot,\cdot\rangle$ is the 
 nondegenerate $ \ZZ[\beta]\llbracket x_1, x_2,\ldots\rrbracket$-bilinear form, continuous in the second coordinate,
 such that $\langle \mu,\nu\rangle = \delta_{\mu\nu}$ for all strict partitions $\mu$ and $\nu$.

 Let $\mu$ be a strict partition of $n \in \NN$
  with  shifted diagram $\SD_\mu$ as in \eqref{sd-eq}. 
The \emph{$r$th diagonal} of $\mu$ is the set of positions $(i,j) \in \SD_\mu$ with $j - i =r$.
The \emph{removable boxes} of $\mu$
are the positions $(i,j) \in \SD_\mu$ such that $\SD_\mu \setminus\{(i,j)\}$ is the shifted diagram of 
a strict partition of $n-1$.
The \emph{addable boxes} of  $\mu$
are the positions $(i,j) \notin \SD_\mu$ such that $\SD_\mu\sqcup \{(i,j)\}$ is the shifted diagram of a strict partition of $n+1$.

 For $n \in \NN$, define
 $a_n : \SPart \to \SPart$
 to be  the continuous linear map such that if $\mu$ is a strict partition then
 \[
a_n \mu = \begin{cases}
\beta \cdot  \mu & \textrm{if $\mu$ has removable box on its $n$th diagonal}, \\
 \nu & \textrm{if $\SD_{\nu} = \SD_\mu \sqcup\{(i,j)\}$ where $j-i=n$, and }\\
0 & \textrm{otherwise}.
\end{cases}
\]
When $\beta=0$ the maps $a_i$ specialize to the \emph{diagonal box-adding operators} considered 
in \cite[Ex.~2.4]{FominGreene} or \cite[\S1.4]{Serrano}.
For $x \in \ZZ[\beta]\llbracket x_1,x_2 \ldots\rrbracket$, let 
\be\label{apq-eq}
\ba 
A_n(x) &= 1 + x a_n,\\
\cP_n(x) &= A_n(x)  \cdots A_2(x) A_1(x) A_0(x) A_{1}(x)A_{2}(x) \cdots    A_{n}(x),
\\
\cQ_n(x) &=  A_n(x)  \cdots A_2(x) A_1(x) A_0(x)A_0(x) A_{1}(x) A_{2}(x) \cdots    A_{n}(x).
\ea
\ee
These operators are similar to the ones which Fomin and Kirillov define in \cite{FominKirillov}. 
Fix strict partitions $\mu \subseteq \lambda$, let $n = \lambda_1 \geq \ell(\lambda)$,
and define 
\[
\bGP_{\lambda\ss \mu}
:=
\left\langle \lambda,  \cdots \cP_n(x_2)\cP_n(x_1)  \mu \right\rangle
\quand
\bGQ_{\lambda\ss \mu}
:=
\left\langle \lambda,  \cdots \cQ_n(x_2)\cQ_n(x_1)  \mu \right\rangle.
\]
The first formula is well-defined because  the coefficient of each fixed $x$-monomial 
 in $\langle \lambda, \cP_n(x_N)\cdots \cP_n(x_2)\cP_n(x_1)  \mu\rangle$
eventually stabilizes as $N\to\infty$.
The second formula makes sense for similar reasons. 
(One can replace $n$ in these formulas by any integer greater than $\lambda_1$
without changing the meaning.)

From our definition, it is only clear that $\bGP_{\lambda\ss \mu}$
and $\bGQ_{\lambda\ss \mu}$ are formal power series in $\ZZ[\beta]\llbracket x_1,x_2,\dots \rrbracket$.
These power series are actually quasisymmetric and  related to 
$\bGP_{\lambda/ \mu}$
and $\bGQ_{\lambda/ \mu}$  in the following way.
Let $\IC(\mu)$ be the set of removable boxes of the strict partition $\mu$.

\begin{proposition}
Suppose $\mu\subseteq \lambda$ are strict partitions.  Then  
\[
\bGP_{\lambda\ss \mu} = \sum_{\nu \subseteq \mu} \beta^{|\mu|-|\nu|} \bGP_{\lambda/\nu}
\qquand
\bGQ_{\lambda\ss \mu} = \sum_{\nu \subseteq \mu} \beta^{|\mu|-|\nu|}\bGQ_{\lambda/\nu}
\] 
where both sums are over all strict partitions $\nu \subseteq \mu$ with $\SD_{\mu/\nu}\subseteq \IC(\mu)$.
\end{proposition}

In particular, neither  $\bGP_{\lambda\ss \lambda}$ nor $\bGQ_{\lambda \ss \lambda}$
is equal to $\bGP_{\lambda/\lambda} = \bGQ_{\lambda/\lambda}  =1$.

\begin{proof}
A \emph{vertical strip} (respectively, \emph{horizontal strip}) is a skew shape with no two boxes in the same row (respectively, same column).

We first prove the formula for $\bGQ_{\lambda \ss\mu}$.
Let $n =\lambda_1 \geq \ell(\lambda)$.
From the definition of $a_i$ and $A_i(x)$, one sees that the strict partition $\lambda$ appears with nonzero coefficient in $A_0(x)A_1(x) \cdots A_n(x) \mu$ if and only if we can produce the shifted diagram of $\lambda$ by adding a vertical strip to the shifted diagram of $\mu$.  Moreover, if we define $\cV_{\lambda/\mu}$ to be the collection of vertical strips $V$ 
such that $V \cup \SD_{\mu} = \SD_\lambda$ and $V \cap \SD_{\mu}\subseteq\IC(\mu)$,
then 
\be\label{coeff1-eq}
\left\langle \lambda, A_0(x) A_1(x)\cdots A_n(x) \mu \right\rangle = \sum_{V \in \mathcal{V}_{\lambda/\mu}} \beta^{- |\lambda/\mu|} (\beta x)^{|V|}.
\ee
Similarly,  $\lambda$ appears with nonzero coefficient in $A_n(x) \cdots A_1(x)A_0(x) \mu$ if and only if we can can produce the shifted diagram of $\lambda$ by adding a horizontal strip to the shifted diagram of $\mu$.  Moreover, if we define $\cH_{\lambda/\mu}$ to be the collection of horizontal strips $H$ such that $H \cup \SD_{\mu} = \SD_\lambda$ and $H \cap \SD_{\mu}\subseteq \IC(\mu)$, then
\[
\left\langle \lambda, A_n(x) \cdots A_1(x)A_0(x) \mu \right\rangle = \sum_{H \in \mathcal{H}_{\lambda/\mu}} \beta^{- |\lambda/\mu|} (\beta x)^{|H|}.
\]
Combining these observations, we deduce that 
\be\label{bgq-ss-eq}
\bGQ_{\lambda\ss \mu} =
\sum_{(V_1, H_1, V_2, H_2, \ldots, V_N,H_N)} \beta^{- |\lambda/\mu|} \prod_i (\beta x_i)^{|V_i| + |H_i|}
\ee
where the sum is over all tuples $(V_1,H_1,V_2,H_2,\ldots, V_N, H_N)$
such that for some sequence of strict partitions $\mu = \lambda^0\subseteq \mu^1 \subseteq \lambda^1 \subseteq \mu^2 \subseteq \lambda^2 \subseteq \dots \subseteq \mu^N \subseteq \lambda^{N} = \lambda$
it holds that $V_i \in \cV_{\mu^i/\lambda^{i-1}}$ and $H_i \in \cH_{\lambda^i/\mu^i}$.

Fix strict partitions $\mu \subseteq \lambda$.  Consider the set of semistandard shifted set-valued tableaux $T$ of shape $\lambda/\nu$
where $\nu \subseteq \mu$ is a strict partition with $\SD_{\mu/\nu}\subseteq \IC(\mu)$.
This set is in bijection with the sequences indexing the summands of \eqref{bgq-ss-eq} via the 
map $T \mapsto (V_1,H_1,V_2,H_2,\ldots)$
where
$V_i$ and  $H_i$ are the sets of 
boxes in $T$ containing $i'$ and $i$, respectively.
Moreover, under this bijection, we have
$|T| = \sum_i (|V_i| + |H_i|)$ and $x^T = \prod_i x_i^{|V_i| + |H_i|}$.
Thus,
the desired formula for $\bGQ_{\lambda \ss \mu}$ follows by combining the definition
\eqref{gq-eq} with \eqref{bgq-ss-eq}.

The argument needed to deduce our formula for $\bGP_{\lambda \ss \mu}$ is similar:
one just needs to modify the steps above by
replacing $A_0(x) A_1(x)\cdots A_n(x)$ by $ A_1(x)\cdots A_n(x)$ in \eqref{coeff1-eq}
and requiring all vertical strips $V \in \cV_{\lambda/\mu}$ to contain no positions on the main diagonal.
We omit the details.
\end{proof}

Applying  inclusion-exclusion to the preceding result gives the following.

\begin{corollary}\label{in-ex-cor}
Suppose $\mu\subseteq \lambda$ are strict partitions. Then 
\[
\bGP_{\lambda/\mu} = \sum_{\nu  \subseteq \mu} (-\beta)^{|\mu| - |\nu|} \bGP_{\lambda\ss \nu}
\qquand
\bGQ_{\lambda/\mu} = \sum_{\nu\subseteq \mu} (-\beta)^{|\mu| - |\nu|} \bGQ_{\lambda\ss \nu}
\]
where both sums are over all strict partitions $\nu \subseteq \mu$.
\end{corollary}

\subsection{Yang-Baxter relations}\label{yb-sect}

In view of Corollary~\ref{in-ex-cor},
to prove Theorem~\ref{skew-sym-thm} it suffices to show that $\bGP_{\lambda \ss  \mu}$ and $\bGQ_{\lambda \ss  \mu}$ are symmetric.
For this it is enough to prove that the operators $\cP_n(x)$ and $\cP_n(y)$ (respectively, $\cQ_n(x)$ and $\cQ_n(y)$) commute.  To show this latter fact, we follow the approach of \cite{FominKirillov} (see also \cite{FominGreene,FominKirillov1996,Serrano,Yeliussiozv2019}), proving some Yang--Baxter-type equations satisfied by the factors $A_i(x)$ and $A_j(x)$.  

Extending our earlier notation, for any expressions $x$ and $y$, we write 
\be x \oplus y := x + y + \beta xy,
\quad \ominus x := \tfrac{-x}{1+\beta x},
\quand x\ominus y := x \oplus(\ominus y) = \tfrac{x-y}{1+\beta y}.
\ee
The operators $A_i(x)$ from \eqref{apq-eq} satisfy the following commutation relations.

\begin{lemma}
\label{yang-baxter-lem}
Let $i,j \in \NN$ and $x,y \in \ZZ[\beta]\llbracket x_1,x_2,\dots \rrbracket$.  Then
\begin{enumerate}[(a)]
\item $A_i(x) A_j(y) = A_j(y) A_i(x)$ if $|i - j| > 1$,
\item $A_i(x) A_i(y) = A_i(x \oplus y)$,
\item $A_{i + 1}(x)A_{i}(x \oplus y) A_{i + 1}(y) = A_{i}(y)A_{i + 1}(x \oplus y) A_{i}(x)$ if $i>0$,  and
\item $A_0(x)A_1(x \oplus y)A_0(y)A_1(y \ominus x) = A_1(y \ominus x)A_0(y)A_1(x \oplus y)A_0(x)$.

\end{enumerate}
\end{lemma}

\begin{proof}
The first two statements are clear from the definitions of $a_i$ and $A_i(x)$ in \eqref{apq-eq}.
For part (c), assume $i>0$ and consider $f := A_{i + 1}(x)A_{i}(z) A_{i + 1}(y)$ and $f' := A_{i}(y)A_{i + 1}(z) A_{i}(x)$; later, we will specialize $z$ to $x \oplus y$.  
Write $\diag_j := \{ (a,b) \in \PP\times \PP : b -a = j\}$ for the $j$th diagonal in $\PP\times \PP$.
The behavior of $f$ and $f'$ on a strict partition $\mu$ depends only on local properties of its shifted diagram.
There are finitely many cases to consider:
\begin{itemize}

\item 
It is not possible for
the adjacent diagonals $\diag_i$ and $\diag_{i+1}$ to both contain removable boxes of $\mu$ or to both contain addable boxes of $\mu$.

\item If neither $\diag_i$ nor $\diag_{i + 1}$ contains an addable or removable box of $\mu$, then we have $a_i \mu = a_{i + 1} \mu = 0$, so $f \mu = f' \mu = \mu$.

\item Suppose $\diag_{i + 1}$ contains an addable box of $\mu$ and $\diag_i$ does not contain an addable or removable box.
Let $\mu'$ be the strict partition whose shifted diagram is obtained from $\SD_\mu$ by adding a box in diagonal $i+1$.
Then $\mu'$ has an addable box in diagonal $i$;
write $\mu''$ for the result of adding this box.
Then $\mu''$ does not have an addable or removable box in diagonal $i+1$, so we compute  
$f \mu = \mu + (x \oplus y) \mu' + yz \mu'' $ and $ 
f' \mu = \mu + z \mu' + yz \mu''.$
Thus we have $(f-f') \mu = (x\oplus y - z)\mu'$.
  
  \item If $\diag_i$ contains an addable box of $\mu$ and $\diag_{i + 1}$ does not contain an addable or removable box then by similar reasoning 
  $(f-f')\mu = (z - x\oplus y)\mu'$ where $\mu'$ is the result of adding a box in diagonal $i$ to $\mu$.
  
  \item If $\diag_{i + 1}$ contains a removable box and $\diag_i$ does not contain an addable or removable box,
   then $f\mu = (1+\beta x)(1+\beta y)\mu = (1 + \beta \cdot x \oplus y)\mu$ and $f'\mu = (1 + \beta z)\mu$.
   
\item If $\diag_{i }$ has a removable box and $\diag_{i+1}$ does not contain an addable or removable box then $f\mu = (1 + \beta z)\mu$ and $f'\mu = (1 + \beta x)(1+  \beta y) \mu = (1 + \beta \cdot x \oplus y)\mu$.

\item Suppose $\diag_i$ contains an addable box of $\mu$ and $\diag_{i + 1}$ contains a removable box.
Let $\mu'$ be the strict partition whose shifted diagram is obtained from $\SD_\mu$ by adding a box in diagonal $i$. Then $\mu'$ does not have an addable or removable box in diagonal $i + 1$,
so 
 $f\mu = (1 + \beta x)(1 + \beta y)\mu + (1 + \beta y)z \mu'$ and $f'\mu = (1 + \beta z)\mu + (x \oplus y + \beta yz)\mu'$. Thus $(f-f') \mu = (x \oplus y - z)( \beta \mu - \mu')$.
 
\item If $\diag_i$ contains a removable box of $\mu$ and $\diag_{i + 1}$ contains an addable box then
by similar reasoning $(f-f') \mu = (z - x \oplus y)( \beta \mu - \mu')$,
 where $\mu'$ is the result of adding a box to $\mu$ in diagonal $i + 1$.
 
\end{itemize}
Comparing coefficients in each case, we see that if $z = x \oplus y$ then $f = f'$.

Our proof of part (d) is similar. Fix $w,x,y,z \in \ZZ[\beta]\llbracket x_1,x_2,\dots \rrbracket$ and consider $g := A_0(x)A_1(w)A_0(y)A_1(z)$ and $g' := A_1(z)A_0(y)A_1(w)A_0(x)$ acting on a diagram $S$.  We have four cases: diagonal $0$ is always an addable or removable box, and diagonal $1$ can be addable, removable, or neither, but never the same as $0$.  We illustrate the cases with shapes $( 2 )$, $( 2, 1 )$, $( 3 )$, $( 3, 1 )$, $( 3, 2 )$, and $( 3, 2, 1 )$.
\begin{itemize}

\item If $\diag_0$ contains a removable box of $\mu$ and $\diag_1$ does not contain
an addable or removable box, then 
the last two parts of $\mu$ must be $(2,1)$ and we  have
 $g \mu =g'\mu =  (1 + \beta x)(1 + \beta y) \mu $.

\item Suppose $\diag_0$ contains an addable box of $\mu$ and $\diag_1$ does not contain
an addable or removable box.
Then the smallest part of $\mu$ must be $\geq 3$, and we can assume without loss of generality that $\mu = (3)$.
Let $\mu' = (3,1)$, $\mu''=(3,2)$, and $\mu''' = (3,2,1)$. Then one can check that
 \[
 \ba
 g \mu &=\mu + (x \oplus y)\mu' + wy\mu'' + wxy\mu''', \text{ and } \\
 g' \mu &= \mu +(x \oplus y) \mu' +(wx + xz + yz + \beta wxz + \beta xyz) \mu'' + wxy \mu''',
 \ea
 \]
 so $(g-g')\mu = -(wx + xz + yz + \beta wxz + \beta xyz - wy)\mu''$.

\item Suppose $\diag_0$ contains an addable box of $\mu$ and $\diag_1$ contains a removable box,
so that the last part of $\mu$ must be $2$. Let $\mu'$ be the strict partition formed by adding a box to $\mu$ in $\diag_0$, creating a new part of size $1$.
Then we have
\[
\ba
g \mu &= (1 + \beta w)(1 + \beta z) \mu + ( x  + y +  \beta wx  + \beta xy)(1+\beta z)\mu',
\text{ and }
\\
g' \mu &= (1 + \beta w)(1 + \beta z)\mu + (x +  y + \beta wy + \beta xy)\mu'
\ea
\]
so $(g-g')\mu =  \beta (wx + xz + yz + \beta wxz + \beta x y z - wy) \mu'$.

%

\item Suppose $\diag_0$ contains a removable box of $\mu$ and $\diag_1$ contains an addable box.
Then the last part of $\mu$ must be $1$ and the second-to-last part of $\mu$ must be $\geq 3$,
and we can assume without loss of generality that $\mu = (3,1)$.
Let $\mu'=(3,2)$ and $\mu'' = (3,2,1)$. Then it follows by similar calculations that
$
(g-g')\mu =  (wx + xz + yz + \beta wxz + \beta x y z - wy) (\beta \mu' -\mu'')
$.

\end{itemize}
It follows in each case that we have $g = g'$ provided that 
\begin{equation}
\label{yb-eq}
wy =  wx + xz + yz + \beta wxz + \beta xyz,
\end{equation} 
and in particular this is satisfied if $w = x \oplus y$ and $z = x\ominus y$.
\end{proof}

We can now show that $\bGP_{\lambda/\mu}$ and $\bGQ_{\lambda/\mu}$ are symmetric functions.

\begin{proof}[Proof of Theorem~\ref{skew-sym-thm}]
The relations in Lemma~\ref{yang-baxter-lem} are the same as \cite[(2.1)--(2.4)]{FominKirillov} with ``$h_i$'' replaced by ``$A_i$'' and ``$+$/$-$'' replaced by ``$\oplus$/$\ominus$''.  
Making the same substitutions transforms \cite[Prop.~4.2]{FominKirillov} to the 
assertion that the operators $\cP_{n}(x)$ and $\cP_{n}(y)$ commute for all $x,y$, which is what we need to show 
to deduce that $\bGP_{\lambda\ss\mu}$ and $\bGP_{\lambda/\mu}$ are symmetric.

 Fomin and Kirillov's proof of \cite[Prop.~4.2]{FominKirillov} only depends on the fact that $+$ is a commutative formal group law; hence every formal consequence of the Yang--Baxter equations for their $h_i$'s also holds for  our $A_i$'s, with ``$+$'' and ``$-$'' respectively replaced by ``$\oplus$'' and ``$\ominus$''.  In particular, the proof of \cite[Prop.~4.2]{FominKirillov} carries over to our context, \emph{mutatis mutandis}, and we conclude that 
 $\cP_{n}(x)\cP_{n}(y) = \cP_{n}(y)\cP_{n}(x)$ as desired.

We now explain how to see that the operators $\cQ_n(x)$ and $\cQ_n(y)$ likewise commute.
Fomin and Kirillov's proof of \cite[Prop.~4.2]{FominKirillov} 
is implicitly an inductive argument, which translates when $n=2$, for example,
to the following sequence of transformations. Here, the Yang--Baxter relations from Lemma~\ref{yang-baxter-lem}
are indicated with braces overhead:
{\scriptsize\begin{align*}
\cP_2(x) \cP_2(y) & = A_2(x)A_1(x)A_0(x)A_1(x)\overbrace{A_2(x)A_2(y)}A_1(y) A_0(y)A_1(y)A_2(y) \\
& = A_2(x)A_1(x)A_0(x)\overbrace{A_1(x)A_2(x\oplus y)A_1(y)} A_0(y)A_1(y)A_2(y) \\
& = \overbrace{A_2(x)}A_1(x)\overbrace{A_0(x)A_2(y)}A_1(x\oplus y)\overbrace{A_2(x) A_0(y)}A_1(y)\overbrace{A_2(y)} \\
& =A_2(y) \overbrace{A_2(x\ominus y) A_1(x) A_2(y)} A_0(x) A_1(x \oplus y) A_0(y) \overbrace{A_2(x) A_1(y) A_2(y \ominus x)} A_2(x) \\
& = A_2(y) A_1(y) A_2(x) A_1(x\ominus y) \overbrace{A_0(x) A_1(x \oplus y) A_0(y) A_1(y \ominus x)} A_2(y) A_1(x) A_2(x) \\
& =  A_2(y) A_1(y) A_2(x) \overbrace{A_1(x\ominus y) A_1(y \ominus x)} A_0(y) A_1(x \oplus y) \overbrace{A_0(x)  A_2(y)} A_1(x) A_2(x) \\
& = A_2(y) A_1(y)\overbrace{ A_2(x)A_0(y)} A_1(x \oplus y)  A_2(y)A_0(x)  A_1(x) A_2(x) \\
& = A_2(y) A_1(y) A_0(y) \overbrace{A_2(x) A_1(x \oplus y)  A_2(y)}A_0(x)  A_1(x) A_2(x) \\
& = A_2(y) A_1(y) A_0(y) A_1(y) \overbrace{A_2(x \oplus y)}  A_1(x)A_0(x)  A_1(x) A_2(x) \\
& =  A_2(y) A_1(y) A_0(y) A_1(y) A_2(y) A_1(x) A_1(x)A_0(x)  A_1(x) A_2(x) \\
& = \cP_2(y) \cP_2(x).
\end{align*}}%
Only minor adjustments to this inductive argument are needed to check that 
$\cQ_n(x)\cQ_n(y)= \cQ_n(y)\cQ_n(x)$. 
For example, when $n=2$ one has
{\scriptsize\begin{align*}
\cQ_2(x) \cQ_2(y) & = A_2(x)A_1(x)\overbrace{A_0(x)A_0(x)}A_1(x)\overbrace{A_2(x)A_2(y)}A_1(y) \overbrace{A_0(y)A_0(y)}A_1(y)A_2(y) \\
& = A_2(x)A_1(x)A_0(x \oplus x)\overbrace{A_1(x)A_2(x\oplus y)A_1(y)} A_0(y\oplus y)A_1(y)A_2(y) \\
& = \overbrace{A_2(x)}A_1(x)\overbrace{A_0(x\oplus x)A_2(y)}A_1(x\oplus y)\overbrace{A_2(x) A_0(y \oplus y)}A_1(y)\overbrace{A_2(y)} \\
& =A_2(y) \overbrace{A_2(x\ominus y) A_1(x) A_2(y)} A_0(x\oplus x) A_1(x \oplus y) A_0(y \oplus y) \overbrace{A_2(x) A_1(y) A_2(y \ominus x)} A_2(x) \\
& = A_2(y) A_1(y) A_2(x) A_1(x\ominus y) \overbrace{A_0(x \oplus x) A_1(x \oplus y) A_0(y \oplus y) A_1(y \ominus x)} A_2(y) A_1(x) A_2(x) \\
& =  A_2(y) A_1(y) A_2(x) \overbrace{A_1(x\ominus y) A_1(y \ominus x)} A_0(y\oplus y) A_1(x \oplus y) \overbrace{A_0(x \oplus x)  A_2(y)} A_1(x) A_2(x) \\
& = A_2(y) A_1(y)\overbrace{ A_2(x)A_0(y \oplus y)} A_1(x \oplus y)  A_2(y)A_0(x \oplus x)  A_1(x) A_2(x) \\
& = A_2(y) A_1(y) A_0(y \oplus y) \overbrace{A_2(x) A_1(x \oplus y)  A_2(y)}A_0(x \oplus x)  A_1(x) A_2(x) \\
& = A_2(y) A_1(y) \overbrace{A_0(y \oplus y)} A_1(y) \overbrace{A_2(x \oplus y)}  A_1(x) \overbrace{A_0(x \oplus x)}  A_1(x) A_2(x) \\
& =  A_2(y) A_1(y) A_0(y)A_0(y) A_1(y) A_2(y) A_1(x) A_1(x)A_0(x) A_0(x) A_1(x) A_2(x) \\
& = \cQ_2(y) \cQ_2(x)
\end{align*}}%
where the transformation in the fifth line uses the identity 
\be
A_0(x\ominus y) A_1(x) A_0(x \oplus y) A_1(y) = A_1(y) A_0(x \oplus y) A_1(x)A_0(x\ominus y)
\ee
 which follows from \eqref{yb-eq}.
The general case is similar and we omit the details.
We conclude that the power series
$\bGQ_{\lambda\ss\mu}$ and  $\bGQ_{\lambda/\mu}$ are symmetric. 
 \end{proof}

\subsection{Symmetric subalgebras}

We have seen that $\GP_\lambda$ and $\GQ_\lambda$ are linearly independent
in $\mSym$ and involve only $x$-monomials of degree at least $|\lambda|$,
so
 the following is well-defined.

\begin{definition}
Let $\mGSym$ and $\omGSym$ denote the linearly compact $\ZZ[\beta]$-modules
with the $K$-theoretic Schur $Q$- and $P$-functions $\{\bGQ_\lambda\}$ and $\{\bGP_\lambda\}$
as respective pseudobases ($\lambda$ ranging over all strict partitions).
\end{definition}

The images of $\omGSym$ and $\mGSym$
under the truncation map setting $x_{n+1}=x_{n+2}=\dots=0$
are the rings $G\Gamma_{n}$ and $G\Gamma_{n,+}$ in \cite{IkedaNaruse}.
Ikeda and Naruse's results in \cite{IkedaNaruse} show
that $\omGSym$ is characterized by the following cancellation law.

\begin{theorem}[{\cite[Thm.~3.1 and Prop.~3.4]{IkedaNaruse}}]
\label{q-cancel-thm}
It holds that
\[\omGSym = \{ f \in \mSym : f(t,\ominus t,x_3,x_4,\dots) = f(x_3,x_4,\dots)\}\]
where $t$ is an indeterminate
commuting with each $x_i$.
\end{theorem}

The following is a $K$-theoretic extension of \cite[Thm.~3.8]{Stembridge1997a}.

\begin{theorem}\label{cap-thm}
It holds that 
\[\mGSym = \mSym \cap \mcoPeak
\qquand \omGSym = \mSym \cap \omcoPeak.\]
In particular, we have 
$\mGSym 
\subseteq \omGSym$.
\end{theorem}

\begin{proof}
First, we consider the case of $\omGSym$.
Fix a strict partition $\lambda$.  
By \eqref{gp-eq} and Corollary~\ref{omco-cor},
we have $\bGP_\lambda \in \omcoPeak$.
Since $\bGP_\lambda$ is symmetric,
$ \omGSym \subseteq \mSym \cap \omcoPeak$.
On the other hand, by Lemma~\ref{q-cancel-lem} and Theorem~\ref{q-cancel-thm}
we have
$\omGSym \supseteq  \mSym \cap \mcoPeak_{\QQ[\beta]} = \mSym \cap (\mQSym \cap \mcoPeak_{\QQ[\beta]})$,
so  Theorem~\ref{inter-thm} implies that $\omGSym \supseteq  \mSym \cap \omcoPeak$.

Second, we consider the case of $\mGSym$.  By \eqref{tilde-GQ-eq}, we have $\bGQ_\lambda \in \mcoPeak$
and so $\mGSym \subseteq \mSym \cap \mcoPeak$.  
To finish, we will check that $\mGSym \supseteq \omGSym \cap \mcoPeak$, which implies $\mGSym \supseteq \mSym \cap \mcoPeak$
by Theorem~\ref{inter-thm}.


Let $\mGSym_{\QQ[\beta]}$ be the linearly compact $\QQ[\beta]$-module
with the $K$-theoretic Schur $Q$-functions $\bGQ_\lambda$ as a pseudobasis.
We claim that $\omGSym \subseteq \mGSym_{\QQ[\beta]}$.
Since $\bGQ_\lambda \in \mSym \cap \mcoPeak \subseteq \omGSym$,
we have $\bGQ_\lambda \in \sum_\mu \ZZ[\beta] \bGP_\mu$
where the sum is over all strict partitions.
If $\beta$ has degree $0$ and each $x_i$ has degree $1$ then
the nonzero homogeneous components of $\bGP_\lambda$ and $\bGQ_\lambda$
of lowest degree
are $P_\lambda$ and $Q_\lambda = 2^{\ell(\lambda)} P_\lambda$,
so 
\be
\bGQ_\lambda \in 2^{\ell(\lambda)} \bGP_\lambda + \sum_{|\lambda| < |\mu|} \ZZ[\beta]\bGP_\mu.
\ee
Hence $\bGP_\lambda \in 2^{-\ell(\lambda)} \bGQ_\lambda + \sum_{|\lambda| < |\mu|} \QQ[\beta] \bGP_\mu$ and it follows
that $\bGP_\lambda \in \mGSym_{\QQ[\beta]}$, so $\omGSym \subseteq \mGSym_{\QQ[\beta]}$ as claimed.

The peak quasisymmetric function $K_\lambda$ is the nonzero homogeneous component
of $K^{(\beta)}_\lambda$ of lowest degree, and it is shown in
the proof of \cite[Thm.~3.8]{Stembridge1997a}
that $Q_\lambda \in K_\lambda + \sum_{\lambda \prec \alpha} \ZZ K_\alpha$
where $\prec$ is the partial order on compositions 
described in the proof of Theorem~\ref{inter-thm}.
Since $\bGQ_\lambda \in \mcoPeak$, we must have
\be
\bGQ_\lambda \in K^{(\beta)}_\lambda + \sum_{\lambda \prec \alpha} \ZZ[\beta] K^{(\beta)}_\alpha.
\ee
Thus,
as in the proof of Theorem~\ref{inter-thm},
any $\QQ[\beta]$-linear combination of $\bGQ_\lambda$'s
that belongs to $\mcoPeak$ must have coefficients in $\ZZ[\beta]$.
By the previous paragraph,
this means that $\mGSym \supseteq \mGSym_{\QQ[\beta]} \cap \mcoPeak \supseteq \omGSym \cap \mcoPeak$.

Finally,  $\mGSym 
\subseteq \omGSym$ holds
as $\mcoPeak \subset \omcoPeak$ by Theorem~\ref{inter-thm}.
\end{proof}

\begin{corollary}\label{vab-cor}
Suppose $(P,\gamma)$ is a labeled poset and $V\subseteq \ValSet(P,\gamma)$.
\ben
\item[(a)] 
If $\tOmega(P,\gamma) \in \mSym$ then $\tOmega(P,\gamma) \in \mGSym$.

\item[(b)] If 
$\tOmega(P,\gamma,V) \in \mSym$
then $\tOmega(P,\gamma,V) \in \omGSym$.
\een
\end{corollary}

\begin{proof}
We have $\tOmega(P,\gamma) \in \mcoPeak$ by Theorem~\ref{ep-thm1}
and $\tOmega(P,\gamma,V) \in \omcoPeak$ by Corollary~\ref{omco-cor},
so this follows by Theorem~\ref{cap-thm}.
\end{proof}

We single out one especially important case.
Fix strict partitions $\mu\subseteq \lambda$.

\begin{corollary}\label{single-cor}
It holds that
$\bGP_{\lambda/\mu} \in   \omGSym$
and $\bGQ_{\lambda/\mu}\in \mGSym$.
\end{corollary}

\begin{proof}
This holds by Corollary~\ref{vab-cor} given
 \eqref{tilde-GQ-eq}, \eqref{gp-eq}, and Theorem~\ref{skew-sym-thm}.
\end{proof}

Concretely, Corollary~\ref{single-cor} means 
that we have $\bGP_{\lambda/\mu} \in  \sum_\nu \ZZ[\beta] \bGP_\nu$
and $\bGQ_{\lambda/\mu}\in\sum_\nu \ZZ[\beta] \bGQ_\nu$ where the sums are over all strict partitions $\nu$.
We expect that these sums are actually finite with positive coefficients:

\begin{conjecture}
 $\bGP_{\lambda/\mu} \in \bigoplus_\nu \NN[\beta]\bGP_\nu$
and $\bGQ_{\lambda/\mu}\in \bigoplus_\nu \NN[\beta]\bGQ_\nu$.
\end{conjecture}

This should be a direct consequence of the geometric interpretation of 
$\bGP_\lambda$ and $\bGQ_\lambda$ in \cite{IkedaNaruse}. 
It would be interesting to find a combinatorial proof.

Since $\mGSym 
\subseteq \omGSym$,
we must have $\bGQ_\mu \in \sum_\lambda \ZZ[\beta] \bGP_\mu$.
Computer calculations suggest 
that this expansion is actually always finite, with the following fairly simple description.
As before, a \emph{vertical strip} is a subset of $\PP\times \PP$ that contains at most one position in each row.

\begin{conjecture}
If $\mu$ is a strict partition then 
\be\label{q-to-p-eq}
\bGQ_\mu = 2^{\ell(\mu)} \cdot \sum_{\lambda}  (-1)^{c(\lambda/\mu)} \cdot (-\beta/2)^{|\lambda/\mu|}\cdot  \bGP_\lambda\ee
where the sum is over the finite set of strict partitions $\lambda\supseteq \mu$ 
with $\ell(\mu) =\ell(\lambda)$ such that $\SD_{\lambda/\mu}$ is a vertical strip,
and we define $c(\lambda/\mu)$ to be the number of distinct columns occupied by positions in $\SD_{\lambda/\mu}$.
\end{conjecture}

For example, one has 
$\bGQ_{(3,2)} = 4 \bGP_{(3,2)} + 2\beta\cdot  \bGP_{(4,2)} - \beta^2\cdot \bGP_{(4,3)}.$

\begin{remark*}
The coefficients in \eqref{q-to-p-eq}
have the form $\pm 2^i \beta^j$ for $i,j \in \{0,1,\dots,\ell(\mu)\}$,
and if the coefficient of $\bGP_\lambda$ is nonzero, then its sign is $(-1)^{s(\lambda/\mu)}$ where
$s(\lambda/\mu)$ is the number of nonempty columns of $\SD_{\lambda/\mu}$
with an even number of boxes.
\end{remark*}

We end this section with two results involving 
the power series $\bGS_\lambda$.

\begin{corollary} The map $\ttheta$ 
restricts to a morphism $ \mSym \to \mGSym$.
\end{corollary}

\begin{proof}
This holds since $ \ttheta(G^{(\beta)}_\lambda)=\bGS_\lambda \in \mGSym$ by Corollary~\ref{single-cor}.
\end{proof}

\begin{corollary}
The set of $K$-theoretic Schur $S$-functions $\{\bGS_\lambda\}$,
with $\lambda$ ranging over all strict partitions,
is another pseudobasis for $\mGSym$.
\end{corollary}

\begin{proof}
Fix a strict partition $\lambda$.
We have $\bGS_\lambda \in \mGSym$,
and it follows from \cite[(8.8$'$), \S III.8]{Macdonald}
and \cite[Ex.\ 7, \S III.8]{Macdonald}
that $S_\lambda \in Q_\lambda + \sum_{\mu > \lambda} \ZZ Q_\mu$ 
where $>$ is the dominance order on strict partitions.
Comparing lowest-degree terms, we deduce that $\bGS_\lambda \in \bGQ_\lambda + \sum_{\mu} \ZZ[\beta] \bGQ_\mu$ 
where the sum is over strict partitions $\mu$ with $\mu > \lambda$ or $|\mu|>|\lambda|$,
so the corollary follows.
\end{proof}

\section{Antipode formulas}\label{antipode-sect}

In this section, we show how to expand the multipeak quasisymmetric functions $\{ K^{(\beta)}_\alpha\}$
in terms of the multifundamental quasisymmetric functions $\{ L^{(\beta)}_\alpha\}$. 
We then derive formulas for several 
involutions of $\mQSym$,
including the antipode.

\subsection{Mirroring operators}

We start by defining an operation on posets that adds ``mirror images'' of certain vertices.
Let $(P,\gamma)$ be a labeled poset. Assume
the vertices of $P$ are all positive integers (not necessarily with the usual order on $\ZZ$)
and $\gamma $ takes only positive integer values.
We refer to labeled posets with these properties as \emph{positive}.
Write $\prec$ for the partial order on $P$. 
For each pair $(I, J)$ of subsets such that $P = I\cup J$, we 
define $\fkM_{IJ}(P,\gamma)$ 
to be the labeled poset $\fkM_{IJ}(P,\gamma) := (Q,\delta)$ with the following properties.
\begin{itemize}
\item As a set, we have $Q := I \cup (-J)$.
\item The partial order on $Q$ has $s\prec t$ if and only if $|s| \prec |t|$ in $P$.
\item We have $\delta(s) =\sign(s) \gamma(|s|)$
where $\sign(s) := s/|s| \in \{\pm 1\}$.
\end{itemize}
Thus $|Q| = |P| + |I\cap J|$, and if $s \in I \cap J$ then $-s$ and $s$ are incomparable in $Q$.
If $J=\varnothing$ then $\fkM_{P \varnothing}(P,\gamma)  = (P,\gamma)$, while if $I = \varnothing$ then 
$\fkM_{\varnothing P}(P,\gamma)  = (P, -\gamma)$ reverses all arrows in the oriented Hasse diagram of $(P, \gamma)$.

\begin{example}
Drawing labeled posets as oriented Hasse diagrams,
we have
\[
\fkM_{\{1, 2,3\}\{2, 3, 4\}}\(\
\begin{tikzpicture}[baseline=(center.base), xscale=0.7, yscale=0.5]
\tikzset{edge/.style = {<-}}
\node (center) at (0,0) {};
\node (d) at (0,3) {$4$};
\node (c) at (0,1) {$3$};
  \node (b) at (0,-1) {$2$};
  \node (a) at (0,-3) {$1$};
  \draw[edge] (b) -- (a);
  \draw[edge] (c) -- (b);
  \draw[edge] (c) -- (d);
\end{tikzpicture}\ \)\ =\ 
\begin{tikzpicture}[baseline=(center.base), xscale=0.7, yscale=0.5]
\tikzset{edge/.style = {<-}}
\node (center) at (0,0) {};
\node (d) at (0,3) {$-4$};
\node (c) at (0,1) {$3$};
\node (c2) at (2,1) {$-3$};
  \node (b) at (0,-1) {$2$};
  \node (b2) at (2,-1) {$-2$};
  \node (a) at (0,-3) {$1$};
  \draw[edge] (b) -- (a);   \draw[edge] (b2) -- (a);
  \draw[edge] (c) -- (b);    \draw[edge] (c2) -- (b);
    \draw[edge] (b2) -- (c);    \draw[edge] (b2) -- (c2);
  \draw[edge] (d) -- (c);     \draw[edge] (d) -- (c2);
\end{tikzpicture}
\ .
\]
\end{example}

Now, given a positive labeled poset $(P,\gamma)$,
define
\be
\Psi^{(\beta)}(P,\gamma) :=
\sum_{I \cup J = P} \beta^{|I \cap J|} \cdot [\fkM_{IJ}(P,\gamma)] \in\mLPSet.
\ee
When $\beta=0$, we have
$
\Psi^{(0)}(P,\gamma) = \sum_{\epsilon: P \to \{\pm 1\}} [(P,\epsilon \gamma)]
$
where $\epsilon \gamma$ is the map $s \mapsto \epsilon(s) \gamma(s)$; compare with \cite[Thm.~3.6]{Stembridge1997a}.

Any isomorphism class of labeled posets contains at least one positive element $(P,\gamma)$,
and the value of $\Psi^{(\beta)}(P,\gamma) $ does not depend on the choice of this representative.
The formula for $\Psi^{(\beta)}$ therefore extends uniquely to a continuous $\ZZ[\beta]$-linear map
$\mLPSet \to \mLPSet$.

Write $\Phi_<: \mLPSet \to \mQSym$ and $\Phi_{>|<} :\mLPSet^+ \to \mQSym$ 
for the
morphisms of combinatorial LC-Hopf algebras from Theorems~\ref{<-thm} and \ref{><-thm}.

\begin{theorem}\label{psi-thm}
The map $\Psi^{(\beta)} : \mLPSet \to \mLPSet$ is an LC-Hopf algebra morphism
making the following diagram commute:
\[
\begin{tikzcd}
\mLPSet \arrow[d, "\Phi_<"]  \arrow[drr, "\Phi_{>|<}"] \arrow[rr, "\Psi^{(\beta)}"] && \mLPSet   \arrow[d, "\Phi_{<}"]  
\\   
\mQSym \arrow[rr, "\ttheta"] && \mQSym  
\end{tikzcd}
\]
Consequently, if $(P,\gamma)$ is a positive labeled poset 
then
\be\label{ijk-eq} \tOmega(P,\gamma) = \sum_{I \cup J = P} \beta^{|I \cap J|} \cdot \tGamma(\fkM_{IJ}(P,\gamma)).
\ee
\end{theorem}

\begin{proof}
The lower triangle in the diagram commutes by 
Theorems~\ref{<-thm} and \ref{><-thm} and Corollary~\ref{ttheta-cor}.
To complete the proof, it suffices by Theorem~\ref{><-thm}
to show that $\Psi^{(\beta)}$ is an LC-Hopf algebra morphism  
and that $\zetaq \circ \Phi_< \circ \Psi^{(\beta)} = \uzetaLP$.

We check the second property first. 
Fix a positive labeled poset  $(P,\gamma)$ 
and note that $\zetaq \circ \Phi_< = \zetaLP$.
If $P = I\cup J$
then $ \zetaLP  ([\fkM_{IJ}(P,\gamma)])$
is zero unless 
(i) $y \notin I$ whenever $y\lessdot z$ in $P$ and $\gamma(y) > \gamma(z)$
and 
(ii) $y \notin J$ whenever $x\lessdot y $ in $P$ and $\gamma(x) < \gamma(y)$.
These conditions can only hold if $\PeakSet(P,\gamma)$ is empty (since $I\cup J=P$).
In this case, if $(I, J)$ satisfies conditions (i) and (ii) then 
$ \zetaLP  (\beta^{|I\cap J|} \cdot [\fkM_{IJ}(P,\gamma)])=t^{|P|}(\beta t)^{|I\cap J|}$.
Moreover, the decompositions $I \cup J = P$ with this nonzero contribution are uniquely determined
by independently assigning each element of $\ValSet(P, \gamma)$ to $I \setminus J$, to $J \setminus I$, or to $I \cap J$,
and so 
\[
 \zetaq \circ \Phi_< \circ \Psi^{(\beta)}([(P,\gamma)])
 = \begin{cases}
 t^{|P|}(1 + 1 + \beta t)^{|\ValSet(P, \gamma)|}
&\text{if $\PeakSet(P,\gamma) = \varnothing$},
\\
0&\text{otherwise}.
\end{cases}
\]
This agrees with the formula for $\uzetaLP([(P,\gamma)]) := \uzetaLP([(P,\gamma,\varnothing)])$
from Proposition~\ref{><-prop}.
We conclude that $\zetaq \circ \Phi_< \circ \Psi^{(\beta)} = \uzetaLP$, as desired.

Finally, we must check that the map $\Psi^{(\beta)}$ is an LC-Hopf algebra morphism.
It clearly commutes with the unit, counit, and product maps of $\mLPSet$,
so we only need to check that
$(\Psi^{(\beta)}\htimes \Psi^{(\beta)}) \circ \Delta = \Delta \circ \Psi^{(\beta)}$.
Continue to let $(P,\gamma)$ be a positive labeled poset.
Let $\delta : P \cup (-P) \to \ZZ$ be the map with $\delta(s) = \sign(s) \delta(|s|)$.
Write $\prec$ for the partial order on $P\cup (-P)$ that has $x\prec y$ if and only if $|x| \prec |y|$ in $P$. 
Relative to this order, the labeled poset $(P\cup -P, \delta)$ is the same thing as $\fkM_{PP}(P,\gamma)$.

Define $\sJ(P,\gamma)$ to be the set of tuples
$(I,J,S_1,S_2)$ 
where $I \cup J = P$ 
and, if $(Q,\delta) := \fkM_{IJ}(P,\gamma)$,
then
$S_1$ is a lower set in $Q$ and $S_2$ is an upper set in $Q$ such that
$Q= S_1 \cup S_2$ and $S_1\cap S_2$ is an antichain.
The value of $\Delta \circ \Psi^{(\beta)}([(P,\gamma)])$ is  
\be
\label{cop-eq1}
  \sum
\beta^{|I \cap J| + |S_1\cap S_2|} \cdot [(S_1,\delta)]\otimes [(S_2,\delta)]
\ee
where the sum is
over all $(I,J,S_1,S_2) \in \sJ(P,\gamma)$,
where each $S_i$  is partially ordered by $\prec$.

Next, define $\sK(P,\gamma)$ to be the set of tuples 
$(T_1,T_2,I_1,J_1,I_2,J_2)$
where $T_1$ is a lower set in $P$ 
and $T_2$ is an upper set in $P$
such that
$P = T_1 \cup T_2$ and $T_1 \cap T_2$ is antichain,
and where $T_i = I_i \cup J_i$.
The value of $(\Psi^{(\beta)}\htimes \Psi^{(\beta)}) \circ \Delta([P,\gamma]) $ is 
\be
\label{cop-eq2}
\sum
\beta^{|I_1 \cap J_1| + |I_2 \cap J_2| + |T_1\cap T_2|} \cdot [\fkM_{I_1J_1}(T_1,\gamma)]\otimes 
[\fkM_{I_2J_2}(T_2,\gamma)]
\ee
where the sum is
over all $(T_1,T_2,I_1,J_1,I_2,J_2) \in \sK(P,\gamma)$.

We must show that \eqref{cop-eq1} and \eqref{cop-eq2} coincide.
It suffices
to exhibit a bijection 
$\sJ(P,\gamma)  \xrightarrow{\sim}  \sK(P,\gamma)$
such that if
\be\label{form-eq}(I,J,S_1,S_2) \mapsto (T_1,T_2,I_1,J_1,I_2,J_2)
\ee
then for $i\in\{1,2\}$ it holds that
\be\label{form-eq2}
(S_i,\delta) = \fkM_{I_iJ_i}(T_i, \gamma) \text{ and }
|I \cap J| + |S_1 \cap S_2| = |I_1 \cap J_1| + |I_2 \cap J_2| + |T_1\cap T_2|.
\ee
The desired map 
is as follows.
Given $(I,J,S_1,S_2) \in \sJ(P,\gamma)$, let
$I_i := \{s \in P : s \in S_i\},$ 
$J_i := \{s \in P : -s \in S_i\},$ and
$T_i := \{ s \in P : \{\pm s\} \cap S_i\neq \varnothing\}$
for $i \in \{1,2\}$.
Clearly $T_i = I_i \cup J_i$, 
$P = T_1 \cup T_2$, and $(S_i,\delta) = \fkM_{I_iJ_i}(T_i,\gamma)$ as labeled posets.
Moreover, $T_1$ is a lower set in $P$ and $T_2$ is an upper set,
and
{\small\begin{align*}
|I \cap J| + |S_1 \cap S_2| 
& =  |I \cap J| + |S_1| + |S_2| - |Q| \\
& =   |S_1| + |S_2| - |P| \\
& = (|T_1| + |I_1 \cap J_1|) + (|T_2| + |I_2 \cap J_2|) - (|T_1| + |T_2| - |T_1 \cap T_2|) \\
& = |I_1 \cap J_1| +  |I_2 \cap J_2| +  |T_1 \cap T_2|.
\end{align*}}%
We must also check that $T_1\cap T_2$ is an antichain in $P$.
For this, suppose $x \in T_1 \cap T_2$ and $y \in P$ and $x\prec y$. 
Since $x \in T_1 \cap T_2$, both $S_1$ and $S_2$ must contain at least one of $\pm x$.  
From our assumptions that 
$S_2$ is an upper set
and $S_1\cap S_2$ is an antichain,
it follows that
$\{ \pm y\} \cap (S_1 \cup S_2) \subseteq S_2\setminus S_1$, so $y \in T_2\setminus T_1$.
We conclude that $T_1 \cap T_2$ is an antichain,
so \eqref{form-eq} is at least a well-defined map $\sJ(P,\gamma) \to \sK(P,\gamma)$
satisfying \eqref{form-eq2}.

To invert \eqref{form-eq},
suppose $(T_1,T_2,I_1,J_1,I_2,J_2) \in \sK(P,\gamma)$.
Let
$S_i := I_i \cup (-J_i) $ for $i \in \{1,2\}$
so that
$(S_i,\delta) = \fkM_{I_iJ_i}(T_i,\gamma)$,
and define
\[
I := \{s \in P : s \in S_1 \cup S_2 \}
\quand 
J := \{s \in P : -s \in S_1 \cup S_2 \}.
\]
Since every $s \in P$ has $\{\pm s\} \cap (S_1 \cup S_2) \neq \varnothing$,
we have $P = I \cup J$ and
$(S_1\cup S_2,\delta)= \fkM_{IJ}(P,\gamma)$,
and it is clear that in this labeled poset 
$S_1$ is a lower set, $S_2$ is an upper set,
and $S_1\cap S_2$ is an antichain.
The correspondence 
$
(T_1,T_2,I_1,J_1,I_2,J_2)
\mapsto 
(I,J,S_1,S_2)
$
defined in this way
is a map $\sK(P,\gamma) \to \sJ(P,\gamma)$.
This is the inverse of \eqref{form-eq}, so \eqref{form-eq}
is the required bijection.
\end{proof}

\subsection{Automorphisms}
\label{antipode-subsection}

If $w$ is a finite sequence then we write $w^\r$ for its reversal.
Given a composition $\alpha \vDash n$,
let $\alpha^\c$ denote the unique composition of $n$ with $I(\alpha^\c) =[n - 1] \setminus I(\alpha)$,
and define $\alpha^\t := (\alpha^\c)^\r = (\alpha^\r)^\c$.
If $w$ is a finite sequence of integers with no adjacent repeated entries and $\alpha \vDash \ell(w)$
has $I(\alpha) = \Des(w) := \{ i : w_i > w_{i+1}\}$,
then $I(\alpha^\t) = \Des(w^\r)$.

Recall that the quasisymmetric functions
$L_\alpha := L^{(0)}_\alpha = \sum_{\alpha \leq \alpha'} M_{\alpha'}$
form a homogeneous basis for $\QSym$ and a pseudobasis for $\mQSym$.
Following \cite[\S3.6]{LMW}, we write $\omega, \psi, \rho : \QSym \to \QSym$
for the linear maps with
\be\label{omega-eq1}
\omega(L_\alpha) := L_{\alpha^\t}
\qquand
\psi(L_\alpha) := L_{\alpha^\c}
\qquand
\rho(L_\alpha) := L_{\alpha^\r}
\ee
for all compositions $\alpha$.
Each of these operators is an algebra morphism;
$\psi$ and $\rho$ are coalgebra anti-automorphisms;
and $\omega = \psi \circ \rho = \rho \circ \psi$ is a Hopf algebra automorphism.
Given a peak composition  $\alpha =(\alpha_1,\alpha_2,\dots,\alpha_k)$, let
 \be 
 \label{flat-eq}
 \alpha^\flat  := (\alpha_k + 1 , \alpha_{k-1},\dots, \alpha_2,\alpha_1 -1).
 \ee
If $\lambda$ is a partition, 
 $\alpha$ is a peak composition,
 and  $K_\alpha := K^{(0)}_{\alpha}$,
then one has
\be\label{omega-eq2}
\ba \omega(s_\lambda) =\psi(s_\lambda) &= s_{\lambda^T} \\
\rho(s_\lambda) &= s_\lambda 
\ea
\qquand
\ba
\omega(K_\alpha) =\rho(K_\alpha)&= K_{\alpha^\flat} \\
\psi(K_\alpha) & =K_\alpha
\ea
\ee
by  \cite[\S3.6]{LMW} and \cite[Prop.~3.5]{Stembridge1997a}.
 These maps extend uniquely by continuity to involutions 
  of $\mQSym$ preserving $\mSym$; we denote the extensions by the same symbols.
  
We can evaluate $\omega$, $\psi$, and $\rho$ at other
quasisymmetric functions of interest.
If $f \in \ZZ[\beta]\llbracket x_1,x_2,\dots \rrbracket$ and $a,b \in \ZZ[\beta]$
  then we abbreviate by writing
  \[f(\tfrac{ax}{1-b x}) :=f( \tfrac{ax_1}{1-b x_1}, \tfrac{ax_2}{1 -bx_2},\dots)\] for the power series
  obtained by substituting $x_i \mapsto \frac{ax_i}{1-b x_i} = ax_i + abx_i^2 + \dots$
  for each $i \in \PP$.
  If $(P,\gamma)$ is a labeled poset then let $P^*$ be the dual poset, in which all order relations
are reversed, and define $\gamma^*(s) = -\gamma(s)$ for $s \in P$.

  \begin{proposition}\label{omega-prop}
If $\alpha$ is a composition and $(P,\gamma)$ is a labeled poset  then 
%
%
%
\[
\ba 
\omega\( L^{(\beta)}_\alpha\) &=L^{(\beta)}_{\alpha^\t} (\tfrac{x}{1-\beta x}) \\
\psi\( L^{(\beta)}_\alpha\) &=L^{(\beta)}_{\alpha^\c} (\tfrac{x}{1-\beta x}) \\
\rho\( L^{(\beta)}_\alpha\) &=L^{(\beta)}_{\alpha^\r}   \\
\ea
\qquand 
\ba
  \omega\(\tGamma(P,\gamma)\) &= \tGamma(P^*,\gamma)(\tfrac{x}{1-\beta x}) \\
   \psi\(\tGamma(P,\gamma)\) &= \tGamma(P,\gamma^*)(\tfrac{x}{1-\beta x}) \\
    \rho\(\tGamma(P,\gamma)\) &= \tGamma(P^*,\gamma^*).
    \ea
 \]
  \end{proposition}
  
When $\beta=1$, the top formulas are closely related to \cite[Prop.~38]{Patrias}.
  
  \begin{proof}
  We use the term \emph{word} in this proof to mean a finite sequence of positive integers.
Let $W$ be the linearly compact $\ZZ[\beta]$-module
with the set of all words as a pseudobasis.
Define $\phi_\leq$, $\phi_<$, and $\phi_\geq$
to be the continuous linear maps $W \to \mQSym$
such that if $v$ is a word 
and $\alpha\vDash \ell(v)$ has $I(\alpha) = \Des(v)$, then
\[
\phi_\leq(v) = L_\alpha,
\qquad
\phi_<(v^\r) = L_{\alpha^\t},
\qquand
\phi_>(v) = L_{\alpha^\c}.
\]
Fix a word $w$ with no adjacent repeated letters and suppose $\alpha\vDash \ell(w)$
has $I(\alpha) = \Des(w)$.
Let $[[w]] \in W$ be the sum of all words that yield $w$ when adjacent repeated letters are combined,
so that, for example, $[[21]] = 21 + 221 + 211 + 2221 + 2211 + 2111 + \dots$.
Then \cite[Props.~8.2 and 8.5]{Marberg2018}
assert that
\[ 
\phi_\leq([[w^\r]]) = \tilde L_{\alpha^\t}(\tfrac{x}{1-x}),
\qquad
\phi_<([[w]]) = \tilde L_{\alpha},
\qquand
\phi_>([[w^\r]]) = \tilde L_{\alpha^\r},
\]
where $\tilde L_\alpha := L^{(1)}_\alpha$.
As $\omega\circ \phi_<(v)  = \phi_\leq(v^\r)$ for any word $v$, 
we therefore have
\be\omega(\tilde L_\alpha) = \omega\circ \phi_<([[w]]) =\phi_\leq([[w^\r]]) = \tilde L_{\alpha^\t}(\tfrac{x}{1-x}).\ee
Similarly, 
since $ \rho\circ \phi_<(v) = \phi_>(v^\r)$ for any word $v$, 
we have 
\be \rho(\tilde L_\alpha) = \rho \circ \phi_<([[w]]) = \phi_>([[w^\r]]) = \tilde L_{\alpha^\r}.\ee
Substituting $x_i \mapsto \beta x_i$ 
and applying \eqref{tilde-l-eq} gives the 
desired expressions for $\omega( L^{(\beta)}_\alpha)$ and $\rho( L^{(\beta)}_\alpha)$.
The formulas for 
$ \omega\(\tGamma(P,\gamma)\) $ and $\rho\(\tGamma(P,\gamma)\)$
then follow from
Theorem~\ref{p-thm1} and Proposition~\ref{by-def-prop}
since  $\tL(P^*) = \{w^\r : w \in \tL(P)\}$
and  $\tGamma(w^\r,\gamma^*) =L^{(\beta)}_{\alpha^\r}$ for $w \in \tL(P)$.
Finally, we compute $\psi(L^{(\beta)}_\alpha)$
and $\psi\(\tGamma(P,\gamma)\)$ using the identity $\psi= \omega \circ \rho$.
  \end{proof}
  
It follows from \eqref{omega-eq2} that if $f \in \mSym$ then $\omega(f) = \psi(f)$ and $\rho(f) = f$, so there is only one nontrivial computation to make on symmetric functions.

    \begin{corollary}\label{omega-g-cor}
If $\mu\subseteq \lambda$ are partitions then $\omega\( G^{(\beta)}_{\lambda/\mu}\)
=G^{(\beta)}_{\lambda^T/\mu^T} (\tfrac{x}{1-\beta x})$.
  \end{corollary}

When $\beta=1$ this identity is essentially \cite[Prop.~9.22]{LamPyl}.
  \begin{proof}
 If $\theta$ is a labeling of $\D_{\lambda/\mu}$ satisfying \eqref{canonical-eq} and $\vartheta$ is an analogous labeling of
$\D_{\lambda^T/\mu^T}$, then $(\D_{\lambda/\mu}, \theta^*)$
 and $(\D_{\lambda^T/\mu^T},\vartheta)$ are equivalent labeled posets.
Since we have $ G^{(\beta)}_{\lambda/\mu} = \tGamma(\D_{\lambda/\mu}, \theta)$ by \eqref{tk-eq},
 the result follows from Proposition~\ref{omega-prop}.
  \end{proof}
    
Next, we state some formulas for the multipeak quasisymmetric functions.
  
    \begin{proposition}\label{omega-peak-prop}
If $\alpha$ is a peak composition  
and $(P,\gamma)$ is a labeled poset
 then 
\[
\ba 
\omega\( K^{(\beta)}_\alpha\)&=K^{(\beta)}_{\alpha^\flat} (\tfrac{x}{1-\beta x}) \\
\psi\( K^{(\beta)}_\alpha\) &= K^{(\beta)}_{\alpha} (\tfrac{x}{1-\beta x}) \\
\rho\( K^{(\beta)}_\alpha\) &=K^{(\beta)}_{\alpha^\flat}   \\
\ea
\qquand 
\ba
\omega\(\tOmega(P,\gamma)\) &=\tOmega(P^*,\gamma)(\tfrac{x}{1-\beta x}) \\
\psi\(\tOmega(P,\gamma)\)&= \tOmega(P,\gamma)(\tfrac{x}{1-\beta x})  \\
\rho\(\tOmega(P,\gamma)\)&=\tOmega(P^*,\gamma).
    \ea
 \]
 Moreover, the formulas on the left
 also hold if we replace ``$K^{(\beta)}$'' by ``$\oK^{(\beta)}$''.
%
  \end{proposition}
  
  \begin{proof}
 If $P = I \cup J $  and $(Q,\delta) := \fkM_{IJ}(P,\gamma)$,
  then $(Q,\delta^*) \cong \fkM_{JI}(P,\gamma)$ and $(Q^*,\delta^*) \cong \fkM_{JI}(P^*,\gamma)$
  as labeled posets. 
 The formulas for
  $\psi\(\tOmega(P,\gamma)\) $
  and
  $\rho\(\tOmega(P,\gamma)\)$
  are immediate 
  from Theorem~\ref{psi-thm} and Proposition~\ref{omega-prop},
  and  $\omega\(\tOmega(P,\gamma)\)  = \psi \circ \rho\(\tOmega(P,\gamma)\) =
  \psi \(\tOmega(P^*,\gamma)\) = \tOmega(P^*,\gamma)(\tfrac{x}{1-\beta x})$.
  
  Given these formulas, the expressions for 
  $\omega (K^{(\beta)}_\alpha)$, 
  $\psi (K^{(\beta)}_\alpha)$, and 
  $\rho (K^{(\beta)}_\alpha)$ are clear
  from Proposition~\ref{k1-prop}, noting that
  if $w \in \tL(P)$ has $\PeakSet(w,\gamma) = I(\alpha)$
  then $\PeakSet(w^\r,\gamma) = I(\alpha^\flat)$.
  The analogous set of identities involving $\oK^{(\beta)}_\alpha$
follow in turn by continuous linearity in view of Corollary~\ref{equiexp-cor}.
  \end{proof}

      \begin{corollary}
If $\mu\subseteq \lambda$ are strict partitions then 
\[
\omega\( \bGP_{\lambda/\mu}\)
=\bGP_{\lambda/\mu} (\tfrac{x}{1-\beta x})
\qquand
\omega\( \bGQ_{\lambda/\mu}\)
=\bGQ_{\lambda/\mu} (\tfrac{x}{1-\beta x}).
\]
  \end{corollary}

  \begin{proof}
  Since $\omega$ and $\psi$ take the same value on the symmetric functions
  $\bGP_{\lambda/\mu}$ and $\bGQ_{\lambda/\mu}$, 
this follows from \eqref{gq-eq}, \eqref{gp-eq}, and Proposition~\ref{omega-peak-prop}.
  \end{proof}
  
We can use these formulas 
to prove two facts about stable Grothendieck polynomials.
The following statement generalizes  independent results of
Ardila and Serrano \cite[Thm.~4.3]{ArdilaSerrano} and DeWitt \cite[Thm.~V.5]{Dewitt}
which show that
$s_{\delta_n/\mu} \in \bigoplus_\nu \NN P_\nu$ for all partitions  $\mu \subseteq \delta_n := (n,n-1,\dots,2,1)$.

\begin{theorem}
If $n \in \NN$ and $\mu \subseteq \delta_n$ then $G^{(\beta)}_{\delta_n/\mu} \in \bigoplus_\nu \NN[\beta] \bGP_\nu$.
\end{theorem}

\begin{proof}
Say that a partition $\mu$ is \emph{strictly contained} in $\lambda$ if $\mu_i < \lambda_i$ for $1 \leq i \leq \ell(\mu)$.
It suffices to prove the proposition when $\mu$ is strictly contained in $\delta_n$,
since Corollary~\ref{sv-products-cor} implies that
in general $G^{(\beta)}_{\delta_n/\mu}$ is a product of power series of the form $G^{(\beta)}_{\delta_m/\nu}$
with $\nu$ strictly contained in $\delta_m$,
and
results in \cite{CTY,HKPWZZ}
show that
products of $K$-theoretic Schur $P$-functions are finite $\NN[\beta]$-linear combinations of 
$K$-theoretic Schur $P$-functions. 

Suppose $\mu$ is a partition strictly contained in $\delta_n$.
Define $b_i  = n - \mu^T_i + i$ for $i \in [n]$ and let $a_1<a_2<\dots<a_n$ 
be the elements of $[2n] \setminus \{b_1,b_2,\dots,b_n\}$.
One has $1<b_1<b_2<\dots<b_n=2n$ and $a_i<b_i$ for each $i$.
The second author's paper \cite{Marberg2019a} considers certain power series
$G_w$ and  $\GO_y$ in $\mSym$
indexed by permutations $w \in S_n$ and $y=y^{-1} \in S_{n}$.
Let 
\[y := (a_1,b_1)(a_2,b_2)\cdots(a_n,b_n) \in S_{2n}
\quand w:= b_1a_1 b_2 a_2 \cdots b_n a_n \in S_{2n}.\]
It follows from \cite[Thm.~2.3 and Prop.~5.5]{Marberg2019a} that 
$\GO_y = G_{w^{-1}}$,
while
 \cite[Thm.~3.1]{Matsumura} implies\footnote{
 Matsumura's result concerns certain polynomials $\fk G_\sigma(x,\xi) \in \ZZ[\beta][x_1,x_2,\dots,\xi_1,\xi_2,\dots]$
indexed by permutations $w \in S_\infty$.
These are related to the power series $G_w$ by the 
identity $G_w = \lim_{m\to \infty} \fk G_{1^m\times w}(x,0)$,
where convergence is in the sense of formal power series.
 } that 
$G_w = G^{(\beta)}_{\delta_n/\mu^T}$.
By  \cite[Lem.~5.3]{Marberg2019a} and Corollary~\ref{omega-g-cor}, we have 
$
\omega(\GO_y)=\omega(G_{w^{-1}}) = G_w(\frac{x}{1-\beta x}) =\omega(G^{(\beta)}_{\delta_n/\mu})$,
so $G^{(\beta)}_{\delta_n/\mu} = \GO_y$.
Finally,  $\GO_y \in \bigoplus_\nu \NN[\beta] \bGP_\nu$ by \cite[Thm.~1.9]{Marberg2019a}.
\end{proof}

The previous result holds in a stronger form when $\mu = \emptyset$.

\begin{proposition}
For a partition $\lambda$ and a strict partition $\nu$, one has that
 $G^{(\beta)}_{\lambda} = \bGP_{\nu}$ if and only if $\lambda =\nu= \delta_n$ for some $n \in \NN$.
\end{proposition}

For a short proof that $s_{\delta_n} = P_{\delta_n}$, see \cite[Ex.~3, \S III.8]{Macdonald}.


\begin{proof}
First, suppose $G^{(\beta)}_{\lambda} = \bGP_{\nu}$.  The $\beta$-degree-$0$ terms of the left and right sides are respectively $s_\lambda$ and $P_\nu$,
and the lowest-degree component of $P_\nu$ is $s_\nu$, so it must be the case that $\lambda = \nu$.
The papers \cite{HMP1,HMP4} study certain symmetric functions $\iF_z$
indexed by involutions $z \in S_\infty$.
It follows from \cite[Thm.~4.20]{HMP4} that every Schur $P$-function $P_\nu$ occurs as $\iF_z$ for some $z$, while
\cite[Prop.~3.34 and Thm.~3.35]{HMP1} assert that  $s_\lambda= \iF_z$ for some $ z $ 
only if $\lambda =\delta_n$ for some $ n \in \NN$.  Thus it must be that $\lambda = \nu = \delta_n$ for some $n \in \NN$.

Conversely, the papers \cite{Marberg2019a,MP2019b} study certain symmetric functions $\GSp_z$ 
indexed by fixed-point-free involutions $z \in S_{2n}$.
When $z = (1,n+1)\cdots(n,2n) \in S_{2n}$,  
 \cite[Thm.~4.17]{MP2019b} shows
that $\GSp_z = \bGP_{\delta_{n-1}}$
while \cite[Thms.~2.5 and 5.2 and Prop.~5.5]{Marberg2019a} 
imply that $\GSp_z = G^{(\beta)}_{\delta_{n-1}}$.
Thus $\bGP_{\delta_n} = G^{(\beta)}_{\delta_n}$.
\end{proof}

The antipode $\antipode$
of the Hopf algebra $\QSym$ is the linear map $\QSym \to \QSym$ with $L_\alpha \mapsto (-1)^{|\alpha|} L_{\alpha^\t}$
for all compositions $\alpha$ \cite[\S3.6]{LMW};
the antipode of $\mQSym$ is the unique continuous extension of this map.
One can evaluate $\antipode$ at many elements 
of interest in $\mQSym$ using the following observation.
Let $f^{(\beta)} \in \ZZ[\beta]\llbracket x_1,x_2,\dots \rrbracket$ and define $f^{(-\beta)} := f^{(\beta)}|_{\beta\mapsto -\beta}$.

\begin{observation}
Assume $f^{(\beta)} \in \mQSym$ is homogeneous of degree $n$
when we define $\deg(\beta) = -1$ and $\deg (x_i) = 1$.
Then $\antipode\(f^{(\beta)}\) = (-1)^n \omega\( f^{(-\beta)}\)$.
\end{observation}

This hypothesis 
applies whenever  $f^{(\beta)} \in \{ L^{(\beta)}_\alpha, 
K^{(\beta)}_\alpha, \oK^{(\beta)}_{\alpha},\bGQ_{\lambda/\mu}, \dots\}$, 
in particular.
Taking $f^{(\beta)} = L^{(\beta)}_\alpha$
 and then specializing to $\beta=1$ recovers \cite[Thm.~41]{Patrias}, while 
taking $f^{(\beta)} = \bGP_\lambda$ gives 
\be
\antipode\(\bGP_\lambda\) = (-1)^{|\lambda|} \GP^{(-\beta)}_\lambda(\tfrac{x}{1+\beta x})
\ee
for all strict partitions $\lambda$.
Comparing with \cite[Prop.~3.5]{HKPWZZ}
shows that the power series which Hamaker \emph{et al.}
denote as $\GP_\lambda$ and $K_\lambda$
may be given in our notation as $\GP_\lambda := \GP_\lambda^{(-1)}$ and 
$K_\lambda :=\GP_\lambda^{(1)}(\tfrac{x}{1-x})$,
and are related by the identity $K_\lambda = (-1)^{|\lambda|}\antipode(\GP_\lambda)$.

Using these formulas, one can check directly that $\antipode$ 
 preserves the cancelation laws in Lemma~\ref{q-cancel-lem} and Theorem~\ref{q-cancel-thm},
 as must hold for $\mGSym$ and
 $\omGSym$
to be LC-Hopf subalgebras of $\mSym$.
By contrast, the involutions $\omega$ and $\psi$ of $\mQSym$
do not preserve $\mGSym$, $\omGSym$, $\mcoPeak$, or $\omcoPeak$.

Although $\omega(\mSym) =\mSym$,
the family of stable Grothendieck polynomials $\{G^{(\beta)}_\lambda\}$ 
is not itself preserved by $\omega$.
To correct this,
Yessulizov \cite{Yeliussizov2017} has introduced a two-parameter generalization $G^{(\alpha,\beta)}_\lambda \in \ZZ[\alpha,\beta]\llbracket x_1,x_2,\dots \rrbracket$
of $G^{(\beta)}_\lambda =: G^{(0,\beta)}_\lambda$ that has $\omega(G^{(\alpha,\beta)}_\lambda)  = G^{(\beta,\alpha)}_{\lambda^T}$. 
It would be interesting to describe analogous
  two-parameter generalizations $\GP^{(\alpha,\beta)}_\lambda$ and $\GQ^{(\alpha,\beta)}_\lambda$
  of the $K$-theoretic Schur $P$- and $Q$-functions
  with $\omega(\GP^{(\alpha,\beta)}_\lambda) = \GP^{(\beta,\alpha)}_\lambda$
  and $\omega(\GQ^{(\alpha,\beta)}_\lambda) = \GQ^{(\beta,\alpha)}_\lambda$.

\printbibliography

\end{document}

%% file: preamble.tex
\usepackage{latexsym}
\usepackage{amssymb}
\usepackage{amsthm}
\usepackage{amscd}
\usepackage{amsmath}
\usepackage{mathrsfs}
\usepackage{graphicx}
\usepackage{shuffle}
\usepackage{hyperref}
\usepackage{mathtools}
\usepackage[bbgreekl]{mathbbol}
\usepackage{mathdots}
\usepackage{ytableau}
\usepackage[enableskew,vcentermath]{youngtab}
\usepackage{tikz-cd}

\usepackage[all]{xy}
\input xy \xyoption{frame}
\xyoption{dvips}

\usepackage[colorinlistoftodos]{todonotes}
\usetikzlibrary{chains,scopes}
\usetikzlibrary{shapes.geometric,positioning}

\DeclareSymbolFontAlphabet{\mathbb}{AMSb}
\DeclareSymbolFontAlphabet{\mathbbl}{bbold}

\newcommand{\SetSSYT}{\ensuremath\mathrm{SetSSYT}}
\newcommand{\SetSYT}{\ensuremath\mathrm{SetSYT}}

\newcommand{\SetSSMT}{\ensuremath\mathrm{SetSSMT}}

\numberwithin{equation}{section}
\allowdisplaybreaks[1]
\UseCrayolaColors

\theoremstyle{definition}
\newtheorem* {theorem*}{Theorem}
\newtheorem* {conjecture*}{Conjecture}
\newtheorem{theorem}{Theorem}[section]

\newtheorem{problem}[theorem]{Problem}
\theoremstyle{definition}
\newtheorem{observation}[theorem]{Observation}

\newtheorem* {example*}{Example}

\newtheorem{lemma}[theorem]{Lemma}
\theoremstyle{definition}
\newtheorem{definition}[theorem]{Definition}
\theoremstyle{definition}
\newtheorem* {notation}{Notation}
\newtheorem{conjecture}[theorem]{Conjecture}
\newtheorem{proposition}[theorem]{Proposition}
\newtheorem{corollary}[theorem]{Corollary}

\newtheorem*{remark*}{Remark}
\theoremstyle{definition}
\newtheorem{remark}[theorem]{Remark}
\theoremstyle{definition}
\newtheorem {example}[theorem]{Example}
\theoremstyle{definition}

\theoremstyle{definition}

\theoremstyle{definition}

\xyoption{dvips}

\def\({\left(}
\def\){\right)}
\newcommand{\sP}{\mathscr{P}}
\newcommand{\sQ}{\mathscr{Q}}
\newcommand{\sI}{\mathscr{I}}
\newcommand{\sJ}{\mathscr{J}}
\newcommand{\sK}{\mathscr{K}}
\newcommand{\CC}{\mathbb{C}}
\newcommand{\QQ}{\mathbb{Q}}
\newcommand{\cP}{\mathcal{P}}
\newcommand{\cQ}{\mathcal{Q}}

\newcommand{\cO}{\mathcal{O}}

\def\Hom{\mathrm{Hom}}
\def\End{\mathrm{End}}
\def\CC{\mathbb{C}}

\def\ZZ{\mathbb{Z}}

\def\diag{\mathrm{diag}}

\newcommand{\cL}{\mathcal{L}}

\def\fk{\mathfrak}

\def\barr{\begin{array}}
\def\earr{\end{array}}
\def\ba{\begin{aligned}}
\def\ea{\end{aligned}}
\def\be{\begin{equation}}
\def\ee{\end{equation}}

\def\qquand{\qquad\text{and}\qquad}
\def\quand{\quad\text{and}\quad}

\def\cH{\mathcal H}

\def\hs{\hspace{0.5mm}}

\def\id{\mathrm{id}}

\def\ben{\begin{enumerate}}
\def\een{\end{enumerate}}

\def\cE{\mathcal E}

\def\hs{\hspace{0.5mm}}

\def\D{\hat D}

\def\Des{\mathrm{Des}}

\def\c{{\tt c}}
\def\t{{\tt t}}
\def\r{{\tt r}}

\newcommand{\Sp}{\textsf{Sp}}

\newcommand{\cA}{\mathcal{A}}

\def\iF{\hat F}

\def\arcstart{\ \xy<0cm,-.15cm>\xymatrix@R=.1cm@C=.3cm }
\newcommand{\arcstartc}[1]{\ \xy<0cm,-.15cm>\xymatrix@R=.1cm@C=#1cm}

\def\cG{\mathcal{G}}

\def\antipode{{\tt S}}

\def\Gr{\mathrm{Gr}}

\def\htimes{\mathbin{\hat\otimes}}

\newcommand{\Sym}{\textsf{Sym}}

\newcommand{\QSym}{\textsf{QSym}}

\def\zetaq{\zeta_{\textsf{Q}}}

\newcommand{\coPeak}{\Pi\Sym}

\newcommand{\mcoPeak}{\mathfrak{m}\coPeak}

\newcommand{\omcoPeak}{\mathfrak{m}\bar{\Pi}\Sym}

\def\PSet{\textsf{Set}(\PP)}
\def\MSet{\textsf{Set}(\MM)}

\def\tA{\cA}
\def\tGamma{\Gamma^{(\beta)}}
\def\tL{\cL}
\def\tE{\cE}
\def\tOmega{\Omega^{(\beta)}}
\def\ttheta{\Theta^{(\beta)}}

\def\oOmega{\bar{\Omega}^{(\beta)}}
\def\oE{\bar{\cE}}
\def\oK{\bar{K}}

\def\whSym
{\mathfrak{m}\textsf{Sym}}

\def\whQSym
{\mathfrak{m}\textsf{QSym}}

\newcommand{\mSym}{\whSym}
\newcommand{\mQSym}{\whQSym}

\def\mGSym{\m\Gamma\Sym}
\def\omGSym{\m\bar\Gamma\Sym}

\def\LPSet{\textsf{LPoset}}
\def\mLPSet{\mathfrak{m}\textsf{LPoset}}

\def\XX{\mathbb{X}}

\def\zetaLP{\zeta_{<}}
\def\dzetaLP{\bar{\zeta}_{>}}
\def\uzetaLP{\zeta_{>|<}}

\def\sort{\textsf{sort}}
\def\m{\mathfrak{m}}

\def\D{\textsf{D}}
\def\SD{\textsf{SD}}

\def\PeakSet{\mathrm{Peak}}
\def\ValSet{\mathrm{Val}}

\def\GS{G\hspace{-0.2mm}S}
\def\GQ{G\hspace{-0.2mm}Q}
\def\GP{G\hspace{-0.2mm}P}

\def\bGQ{\GQ^{(\beta)}}
\def\bGP{\GP^{(\beta)}}
\def\bGS{\GS^{(\beta)}}

\def\fkD{\fk D}
\def\fkM{\fk M}
\def\sign{\mathrm{sign}}

\def\ss{/\hspace{-1mm}/}

\def\OG{\mathrm{OG}}
\def\LG{\mathrm{LG}}

\def\SPart{\textsf{SPart}}
\def\mSPart{\m\textsf{SPart}}
\def\IC{\mathsf{Rem}}
\def\cV{\mathcal{V}}

\def\GSp{\GP^{\Sp}}
\def\GO{\GP^{\textsf{O}}}

\def\MM{\mathbb{M}}
\def\PP{\ZZ_{>0}}
\def\NN{\ZZ_{\geq 0}}
\def\diag{\mathfrak{d}}